\documentclass[12pt]{article}

\usepackage{ogonek}
\usepackage{mathrsfs}

\usepackage{graphicx,multicol}
\usepackage{amssymb, amsthm, amsfonts, amsxtra, amsmath}
\usepackage{latexsym}
\usepackage[all]{xy}
\usepackage{graphics}
\usepackage{makeidx}
\usepackage{color}

\usepackage{syntonly}
\usepackage{nomencl}
\usepackage{pdfpages}

\usepackage{parskip} 
\makeatletter 
\def\thm@space@setup{%
  \thm@preskip=\parskip \thm@postskip=0pt
}
\makeatother

\usepackage{wasysym}

\theoremstyle{change}

\DeclareMathOperator{\id}{id}

\DeclareMathOperator{\alg}{\mathrm{alg}}
\DeclareMathOperator{\atimes}{\stimes{\alg}}
\DeclareMathOperator{\aotimes}{\stimes{\alg}}
\DeclareMathOperator{\op}{\mathrm{op}}
\DeclareMathOperator{\X}{\mathbb{X}}
\DeclareMathOperator{\Y}{\mathbb{Y}}
\DeclareMathOperator{\R}{\mathbb{R}}
\DeclareMathOperator{\C}{\mathbb{C}}
\DeclareMathOperator{\N}{\mathbb{N}}
\DeclareMathOperator{\Hsp}{\mathcal{H}}
\DeclareMathOperator{\G}{\mathbb{G}}
\DeclareMathOperator{\Mor}{\mathrm{Mor}}
\DeclareMathOperator{\Ad}{\mathrm{Ad}}
\DeclareMathOperator{\rd}{\mathrm{d}\!}
\DeclareMathOperator{\Tr}{\mathrm{Tr}}
\DeclareMathOperator{\red}{\mathrm{red}}
\DeclareMathOperator{\Ker}{\mathrm{Ker}}
\DeclareMathOperator{\Z}{\mathbb{Z}}

\DeclareMathOperator{\stab}{\mathrm{stab}}

\newcommand{\mult}{\mathrm{mult}}

\newcommand{\sgn}{\mathrm{sgn}}

\newcommand{\stimes}[1]{\underset{#1}{\otimes}}
\newcommand{\Ind}{\mathrm{Ind}}

\newtheorem{Theorem}{Theorem}[section]
\newtheorem{Lem}[Theorem]{Lemma}

\newtheorem{Prop}[Theorem]{Proposition}
\newtheorem{Cor}[Theorem]{Corollary}

\theoremstyle{definition}
\newtheorem{Def}[Theorem]{Definition}
\newtheorem{Rem}[Theorem]{Remark}

\newtheorem{Exa}[Theorem]{Example}
\newtheorem{Not}[Theorem]{Notation}

\newtheorem{Term}[Theorem]{Terminology}
\newtheorem{Exe}[Theorem]{Exercise}

\date{}

\numberwithin{equation}{section}

\title{Actions of compact quantum groups}
\author{Kenny De Commer\thanks{Department of Mathematics, Vrije Universiteit Brussel, VUB, B-1050 Brussels, Belgium, email: {\tt kenny.de.commer@vub.ac.be}} \thanks{Supported by the FWO grant G.0251.15N.}}
\date{}

\begin{document}
\maketitle
\date

\begin{abstract} \noindent These lecture notes deal with aspects of the theory of actions of compact quantum groups on C$^*$-algebras (locally compact quantum spaces). After going over the basic notions of isotypical components and reduced and universal completions, we look at crossed and smash product C$^*$-algebras, up to the statement of the Takesaki-Takai-Baaj-Skandalis duality. We then look at two special types of actions, namely homogeneous actions and free actions. We study the actions which combine both types, the quantum torsors, and show that more generally any homogeneous action can be completed to a free action with a discrete, classical set of `quantum orbits'. We end with a combinatorial description of the homogeneous actions for the free orthogonal quantum groups.  
\end{abstract}

\hyphenation{re-pre-sen-ta-tions -re-pre-sen-ta-tions e-qui-va-ri-ant}

\section*{Introduction}

These lecture notes are intended to bring together some general results on actions of compact quantum groups on C$^*$-algebras. Their scope is quite modest: apart from basic material, we will treat the following topics:

\begin{itemize}
\item reduced and universal actions,
\item crossed and smash products, and
\item free and homogeneous actions.
\end{itemize}

We will not touch upon the theory of actions of compact quantum groups on von Neumann algebras, which is very similar to and slightly easier than the one for C$^*$-algebras, nor will we deal in depth with the Tannaka-Kre$\breve{\textrm{\i}}$n theory for compact quantum groups and their actions, although some comments will be spent on it.  

We make however the deliberate choice to treat in general and in depth the case of compact quantum group actions on \emph{locally compact} quantum spaces, that is, non-unital C$^*$-algebras. Indeed, this is a necessity in the last part of these notes, where homogeneous and free actions are put into correspondence with each other. Overall, this theory is not much harder than the one for actions on compact quantum spaces, and certainly much more manageable than the theory of actions of \emph{locally compact} quantum groups on locally compact quantum spaces. Nevertheless, we hope that these notes will provide the student of operator algebraic quantum groups with a good starting point towards the latter.   

Apart from the preliminary material on compact quantum groups and C$^*$-algebras and some results in the later sections, I have tried to make these notes as self-contained as possible. Most arguments have been written out completely, as to present the reader with a good appreciation of the techniques involved. Although none of the material in these notes is truly original, we hope that it will nevertheless prove beneficial to have a single treatment on compact quantum group actions from a unified point of view.

The precise content of these notes is as follows. 

In the \emph{first section}, we briefly recall the basics of the theory of compact quantum groups, mainly to introduce the notation that we will adhere to. In the \emph{second section}, we introduce Podle\'{s} notion of \emph{action of a compact quantum group on a C$^*$-algebra}, and discuss some elementary properties. In the \emph{third section}, we look at the \emph{algebraic core} of an action and the closely related \emph{isotypical components}. In the \emph{fourth section}, we treat the results of H. Li on completions of the algebraic core. In the \emph{fifth section}, we look at the different C$^*$-algebraic completions of the \emph{crossed} and \emph{smash} product associated respectively to \emph{actions} and \emph{coactions} of compact quantum groups, and we briefly discuss the Takesaki-Takai-Baaj-Skandalis duality. 

We then shift our focus to special classes of actions. In the \emph{sixth section}, we discuss the compact quantum group analogue of a \emph{homogeneous action}, treating in detail the theory of F. Boca, and in the \emph{seventh section} we look at the analogue in this context of a \emph{free action}. In the \emph{eighth section}, we discuss those compact quantum group actions which are at the same time free and homogeneous, the so-called \emph{compact quantum torsors}. In the final \emph{ninth section}, we show how general quantum homogeneous actions can be put into correspondence with free actions whose quantum orbit space is discrete and classical, and we have a look at quantum homogeneous actions for a particular class of compact quantum groups, namely the \emph{free orthogonal quantum groups} of Van Daele and Wang.

\emph{Acknowledgments}: These notes are based on the lecture series on compact quantum groups actions that I presented at the Summer school "Topological quantum groups", B\k{e}dlewo 2015. I would like to thank Uwe Franz, Adam Skalski and Piotr So\l tan for the invitation and the excellent organisation.

\section{Preliminaries}\label{SecPrel}

These notes will presuppose a good working knowledge of the basic theory concerning \emph{C$^*$-algebras} and \emph{Hopf algebras}, and more specifically \emph{compact quantum groups}. If the reader lacks any of the prior knowledge, he can consult the following sources:

\begin{itemize}
\item C$^*$-algebras: \cite{Mur90},
\item Hilbert C$^*$-modules: \cite{Lan95},
\item Hopf algebras: \cite{Swe69},
\item Compact quantum groups: \cite{Wor95,MVD98}.
\end{itemize}

We also refer to the introductory notes on compact quantum groups in this volume.

Let us comment on our conventions and notations. Will use basic notation as presented in the introductory notes of this volume, such as leg numbering notation and Sweedler notation, although we will write the latter without the summation sign for even more notational reduction, so \[\Delta(h) = h_{(1)}\otimes h_{(2)}.\] 

For $S$ a subset of a normed space, we write \[\lbrack S \rbrack = \textnormal{closed linear span of }S.\] For $X,Y$ closed subspaces of a Banach algebra, we write \[\lbrack XY \rbrack= \left\lbrack \{xy \mid x\in X,y\in Y\}\right\rbrack,\] and the same notation will be used for Banach modules.

We will write a general C$^*$-algebra as $C_0(\X)$, and refer to the underlying symbol $\X$ as the underlying \emph{locally compact quantum space}. When the C$^*$-algebra is unital, we write it $C(\X)$ and refer to $\X$ as the underlying \emph{compact quantum space}. When $\X=X$ is an actual (locally) compact Hausdorff space, $C(X)$ (resp.~ $C_0(X)$) will be the function algebra of complex continuous functions (vanishing at infinity) on $X$. For C$^*$-algebras, the symbol $\otimes$ will always refer to the \emph{minimal} tensor product. For general algebras over the ground field $\C$, we write $\atimes$ for the algebraic tensor product over $\C$.

In the rest of this section, we will recall some basic results and introduce notation concerning Hopf algebras and compact quantum groups, mainly to reconcile the approaches from pure algebra and operator algebra.

For $H = (H,\Delta)$ a Hopf algebra, we write $\varepsilon$ for the counit, and $S$ for the antipode. If $H$ is a Hopf $^*$-algebra, one automatically has that $\varepsilon$ is a $^*$-homomorphism, while $S$ satisfies the important formula \[S(S(h)^*)^* = h, \qquad \forall h\in H.\]

For the definition of a compact quantum group, we refer to the introductory notes of this volume. We will however add to the definition the requirement that the comultiplication is injective. Indeed, there are at the moment no known examples of compact quantum groups with the comultiplication not injective!

The most important result to get all of the theory started is the existence of the \emph{Haar state} \[\varphi_{\G}: C(\G)\rightarrow \C\] for a compact quantum group $\G$, which we will also refer to as the \emph{Haar measure} for the compact quantum group. We denote the GNS-construction for $\varphi_{\G}$ by \[(L^2(\G),\pi_{\red},\xi_{\G}),\] and we write \[C(\G_{\red}) = \pi_{\red}(C(\G)).\] The coproduct on $C(\G)$ then descends to a coproduct \[\Delta: C(\G_{\red}) \rightarrow C(\G_{\red})\otimes C(\G_{\red}).\] We refer to $\G_{\red}$ as the \emph{reduced} compact quantum group associated to $\G$. We note that \[x\in \Ker(\pi_{\red}) \quad \Leftrightarrow \quad \varphi_{\G}(x^*x)=0,\] and consequently the Haar measure on $\G_{\red}$, which is given as the vector state associated to $\xi_{\G}$, is faithful.

The notion of \emph{representation} of a compact quantum group is all-important. By $\G$-representation $\pi$ for a compact quantum group $\G$ we will always mean a finite-dimensional unitary corepresentation $U_{\pi} \in B(\Hsp_{\pi})\otimes C(\G)$ with $\Hsp_{\pi}$ a finite-dimensional Hilbert space.

\begin{Not} For $\pi$ a $\G$-representation and $\xi,\eta\in \Hsp_{\pi}$, we write \[U_{\pi}(\xi,\eta) = (\xi^*\otimes 1)U(\eta\otimes 1) \in C(\G).\]
\end{Not}

Here we interpret $\Hsp_{\pi} \cong B(\C,\Hsp_{\pi})$, so that \[\xi^* T \eta = \langle \xi,T\eta\rangle,\quad \xi,\eta\in \Hsp_{\pi}, T\in B(\Hsp_{\pi}). \]

\begin{Theorem} Let \[\mathscr{O}(\G) = \{U_{\pi}(\xi,\eta)\mid \pi \textnormal{ a } \G\textnormal{\textit{-representation, }}\xi,\eta\in \Hsp_{\pi}\}.\]
Then 
\begin{itemize}
\item $(\mathscr{O}(\G),\Delta)$ is a Hopf $^*$-algebra,
\item $\mathscr{O}(\G)$ is dense in $C(\G)$ for the operator norm, and
\item the map \[\delta_{\pi}: \Hsp_{\pi}\rightarrow \Hsp_{\pi}\otimes \mathscr{O}(\G),\quad \xi\mapsto U_{\pi}(\xi\otimes 1)\] is an \emph{$\mathscr{O}(\G)$-comodule}, by which we mean
\begin{itemize}
\item[$\bullet$] $(\id\otimes \Delta)\circ\delta_{\pi} = (\delta_{\pi}\otimes \id)\circ\delta_{\pi}$,
\item[$\bullet$] $(\id_{\Hsp_{\pi}}\otimes \varepsilon)\delta_{\pi}= \id_{\Hsp_{\pi}}$.
\end{itemize}
\end{itemize}
\end{Theorem}

One can then show that $\pi_{\red}$ is injective on $\mathscr{O}(\G)$, and that in fact $\mathscr{O}(\G) = \mathscr{O}(\G_{\red})$. In particular, one obtains the following lemma.

\begin{Lem}\label{LemPhiFaith} The state $\varphi_{\G}$ is faithful on $\mathscr{O}(\G)$:\[\forall h\in \mathscr{O}(\G), \quad \varphi_{\G}(h^*h)=0\quad \Rightarrow \quad h=0.\]
In fact, if $h\in \mathscr{O}(\G)$ is positive in $C(\G)$ and $\varphi_{\G}(h)=0$, then $h=0$.
\end{Lem} 

The following property is called the \emph{strong left invariance} of the Haar state.

\begin{Lem} For $a\in C(\G)$ and $h\in \mathscr{O}(\G)$, we have \[(\id_{\G}\otimes \varphi_{\G})(\Delta(a)(1_{\G}\otimes h)) = S^{-1} (\id_{\G}\otimes \varphi_{\G})((1_{\G}\otimes a)\Delta(h)).\]
\end{Lem}
\begin{proof}We have \begin{eqnarray*} (\id_{\G}\otimes \varphi_{\G})(\Delta(a)(1_{\G}\otimes h)) &=& (\id_{\G}\otimes \varphi_{\G})(\Delta(a)(h_{(2)}S^{-1}(h_{(1)})\otimes h_{(3)})) \\ &=& (\id_{\G}\otimes \varphi_{\G})(\Delta(ah_{(2)})(S^{-1}(h_{(1)})\otimes 1_{\G})) \\ &=&  S^{-1}((\id_{\G}\otimes \varphi_{\G})((1_{\G}\otimes a)\Delta_{\G}(h))).\end{eqnarray*}
\end{proof} 

For $(\mathscr{O}(\G),\Delta)$, we have the following formulas.

\begin{Lem} With $\{e_i\}$ an orthonormal basis of $\Hsp_{\pi}$, we have \[\Delta(U_{\pi}(\xi,\eta)) = \sum_i U_{\pi}(\xi,e_i)\otimes U_{\pi}(e_i,\eta),\quad \xi,\eta\in \Hsp_{\pi}.\]  Moreover, \[\varepsilon(U_{\pi}(\xi,\eta)) = \langle \xi,\eta\rangle,\quad S(U_{\pi}(\xi,\eta))^* = U_{\pi}(\eta,\xi).\]
\end{Lem}

We can further define direct sums $\pi_1\oplus \pi_2$ and tensor products $\pi_1\otimes \pi_2$ of $\G$-representations $\pi_1$ and $\pi_2$, in a unique way such that $\Hsp_{\pi_1\oplus \pi_2} = \Hsp_{\pi_1}\oplus \Hsp_{\pi_2}$ and $\Hsp_{\pi_1\otimes \pi_2} = \Hsp_{\pi_1}\otimes \Hsp_{\pi_2}$, and \[U_{\pi_1\oplus \pi_2}(\xi_1\oplus \xi_2,\eta_1\oplus \eta_2) = U_{\pi_1}(\xi_1,\eta_1)+ U_{\pi_2}(\xi_2,\eta_2),\]\[ U_{\pi_1\otimes \pi_2}(\xi_1\otimes \xi_2,\eta_1\otimes \eta_2) = U_{\pi_1}(\xi_1,\eta_1)U_{\pi_2}(\xi_2,\eta_2).\]

Let us comment on the following, alternative view on unitary $\G$-representations. 

\begin{Def} Let $\G$ be a compact quantum group. A \emph{finite-dimensional unitary left $\mathbb{G}$-module}, also called \emph{finite-dimensional unitary right $(C(\G),\Delta)$-comodule}, consists of 
\begin{itemize}
\item a finite-dimensional Hilbert space $\Hsp$ and
\item a linear map \[\delta: \Hsp \rightarrow \Hsp \otimes C(\mathbb{G})\]
\end{itemize} 
such that 
\begin{itemize}
\item the \emph{right comodule property} is satisfied,  \[(\id_{\Hsp}\otimes \Delta)\circ\delta = (\delta\otimes \id_{\G})\circ\delta,\]
\item $\delta$ is isometric, \[\delta(\xi)^*\delta(\eta) =   \langle \xi,\eta\rangle 1_{\mathbb{G}}.\]
\end{itemize}
\end{Def}

\begin{Lem} There is a one-to-one correspondence between unitary $\G$-repre-sentations and unitary $\G$-modules by means of the correspondence \[U \mapsto \delta_U,\quad \delta_U(\xi) = U(\xi\otimes 1).\]
\end{Lem}

\begin{proof} It is clear that the above sets up a bijective correspondence between unitary $\G$-modules and \emph{isometries} $U \in B(\Hsp)\otimes C(\G)$ satisfying the corepresentation property. It is hence sufficient to show that the latter are automatically unitary.

But let $U$ be an isometry satisfying the corepresentation property. Write \[P=UU^*,\quad Q = (\id_{B(\Hsp)}\otimes \varphi_{\G})P.\] Since $U$ is a corepresentation, we find \[(\id_{B(\Hsp)}\otimes \Delta)P = U_{12}P_{13}U_{12}^*,\] so applying $\varphi_{\G}$ to the last leg gives \[\label{EqMajQ} Q\otimes 1_{\G} =U(Q\otimes 1_{\G})U^*.\] Multiplying to the right with $U$ shows that $Q$ commutes with $U$. Since $Q$ is positive, we find that $U$ is a direct sum of isometric corepresentations on the eigenspaces of $Q$. If we can now show that the eigenspace $\Ker(Q)$ is zero, we are done, since the inequality \[Q\otimes 1_{\G} =U(Q\otimes 1_{\G})U^* \leq UU^* = P\] shows that $P$ is then a projection bounded from below by an invertible operator, and must hence be the unit. 

But if $\Ker(Q)\neq \{0\}$, we may in fact suppose that $Q=0$. Then by definition $(\id_{B(\Hsp)}\otimes \varphi_{\G})(UU^*)=0$, so $(\id_{B(\Hsp)}\otimes \pi_{\red})U^*=0$. But this is a contradiction since $U$ is isometric. 
\end{proof}

We remark that the above lemma still holds for locally compact quantum groups, see \cite[Corollary 4.16]{BDS13}.

In the following, if $\pi$ is a representation of $\G$, we denote by $\delta_{\pi}$ the associated unitary $\G$-module. We will also use the Sweedler notation for comodules, so \[\delta(\xi) = \xi_{(0)}\otimes \xi_{(1)},\quad (\id_{\Hsp}\otimes \Delta)\delta(\xi) = \xi_{(0)}\otimes \xi_{(1)}\otimes \xi_{(2)},\quad \ldots\]

In the next theorem, we will use the convolution $^*$-algebra structure on the linear dual $\mathscr{O}(\G)^*$ of $\mathscr{O}(\G)$, \[(\omega*\theta)(x) = (\omega\otimes \theta)\Delta(x),\quad \omega^*(x) = \overline{\omega(S(x)^*)},\quad \omega,\theta \in \mathscr{O}(\G)^*,x\in \mathscr{O}(\G).\] We then further write \[\omega*a  = (\id_{\G}\otimes \omega)\Delta(a),\quad a*\omega = (\omega\otimes \id_{\G})\Delta(a),\quad a\in \mathscr{O}(\G),\omega \in \mathscr{O}(\G)^*.\] 

\begin{Theorem}\label{TheoWorChar} There exists a unique convolution invertible functional $f\in \mathscr{O}(\G)^*$ such that \begin{enumerate}
\item $Q_{\pi} = (\id_{\Hsp_{\pi}}\otimes f)U_{\pi}$ is positive for each representation $\pi$, 
\item $f^{-1}*a*f = S^2(a)$ for all $a\in \mathscr{O}(\G)$, 
\item with $\sigma(a) = f*a*f$, one has that $\sigma$ is an algebra isomorphism and \[\varphi_{\G}(ab) = \varphi_{\G}(b\sigma(a)),\quad \forall a,b\in \mathscr{O}(\G).\] One calls $\sigma$ the \emph{modular} (or \emph{Nakayama}) \emph{automorphism} for $\varphi_{\G}$.
\item  For all irreducible $\pi$, one has \[\varphi_{\G}(U_{\pi}(\xi,\eta)U_{\pi}(\zeta,\chi)^*) = \frac{\langle \xi,\zeta\rangle \langle \chi,Q_{\pi}\eta\rangle}{\mathrm{Tr}(Q_{\pi})},\]\[\varphi_{\G}(U_{\pi}(\xi,\eta)^*U_{\pi}(\zeta,\chi)) = \frac{\langle \zeta,Q_{\pi}^{-1}\xi\rangle \langle \eta,\chi\rangle}{\mathrm{Tr}(Q_{\pi})}.\]
\end{enumerate}
\end{Theorem}

Note that each $Q_{\pi}$ is invertible, as $f$ is convolution invertible. If then $\pi$ is a unitary representation of $\G$, one obtains by the orthogonality relations a unitary representation of $\G$ on the dual $\Hsp_{\bar{\pi}} = \Hsp_{\pi}^*$ by endowing $\Hsp_{\pi}^*$ with the new inner product \[\langle \langle \xi^*,\eta^* \rangle \rangle = \langle \eta,Q_{\pi}\xi\rangle\] and the comodule structure \[\delta_{\bar{\pi}}(\xi^*) = \delta_{\pi}(\xi)^*.\] Note that by this definition \[U_{\pi}(\xi,\eta)^* = U_{\bar{\pi}}((Q_{\pi}^{-1}\xi)^*,\eta^*),\quad S(U_{\pi}(\xi,\eta)) = U_{\bar{\pi}}((Q_{\pi}^{-1}\eta)^*,\xi^*),\] while \[ Q_{\bar{\pi}}\eta^* = (Q_{\pi}^{-1}\eta)^*.\]

For $\pi$ an arbitrary representation, the number $\dim_q(\pi) = \mathrm{Tr}(Q_{\pi})$ is called the \emph{quantum dimension} of $\pi$. One then has for any two $\G$-representations $\pi,\pi'$ that \begin{itemize}
\item $\dim_q(\pi\oplus \pi') = \dim_q(\pi) +\dim_q(\pi')$,
\item $\dim_q(\pi\otimes \pi') = \dim_q(\pi)\dim_q(\pi')$,
\item  $\dim_q(\pi) = \dim_q(\bar{\pi})$.
\end{itemize}

One can in fact define for each $z\in \C$ a functional $f^z$ on $\mathscr{O}(\G)$ such that \[(\id\otimes f^z)U_{\pi}= Q_{\pi}^z.\] The functionals $f^z$, called the \emph{Woronowicz characters}, satisfy \[f^z*f^w = f^{z+w},\quad (f^z)^* = f^{\bar{z}},\quad f^z(ab) = f^z(a)f^z(b),\quad a,b\in \mathscr{O}(\G).\] One can then also define for each $z\in \C$ automorphisms \[\sigma_z(a) = f^{iz} *a*f^{iz},\quad \tau_z(a) = f^{-iz}*a*f^{iz},\] so that $\sigma_{-i} = \sigma$ and $\tau_{-i} = S^2$. We have that $\sigma_t$ and $\tau_t$ are $^*$-preserving for $t\in \R$, while in general \[\tau_z(g)^* = \tau_{\bar{z}}(g^*),\quad \sigma_z(g)^* = \sigma_{\bar{z}}(g^*),\quad z\in \C,g\in \mathscr{O}(\G).\]  Further $\tau_z\circ \tau_w = \tau_{z+w}$, $\sigma_{z}\circ \sigma_w = \sigma_{z+w}$ and  \[\varphi_{\G} \circ \sigma_z = \varphi_{\G}\circ \tau_z  = \varphi_{\G}, \quad \forall z\in \C.\] 

Note that for a  general $\mathscr{O}(\G)$-comodule structure $(V,\delta)$ on a finite-dimensional vector space $V$, there exists at least one Hilbert space structure on $V$ for which it is a unitary $\G$-representation. For example, choosing an arbitrary Hilbert space structure $(V,\langle\,\cdot\,,\,\cdot\,\rangle)$, one can define \[\langle\langle \xi,\eta\rangle\rangle = \varphi_{\G}(\delta(\xi)^*\delta(\eta)),\] with respect to which $\delta$ becomes a unitary comodule structure.

\begin{Def}[Space of intertwiners]\label{DefIntSpa}  Let $\pi_1$ and $\pi_2$ be two $\G$-representa-tions. We define the space of \emph{intertwiners} (or \emph{morphisms}) between $\pi_1$ and $\pi_2$ as \begin{eqnarray*} \Mor(\pi_1,\pi_2)&=& \{T: \Hsp_1\rightarrow \Hsp_2\mid \delta_{\pi_2}\circ T = (T\otimes \id_{\G})\circ \delta_{\pi_1}\}\\ &=& \{T: \Hsp_1\rightarrow \Hsp_2\mid U_{\pi_2}(T\otimes 1_{\G}) = (T\otimes 1_{\G})U_{\pi_1}\}\\ &\subseteq& B(\Hsp_1,\Hsp_2).\end{eqnarray*}
\end{Def} 
\begin{Lem}  Let $\pi_1$ and $\pi_2$ be two $\G$-representations.
\begin{itemize}
\item If $T \in \Mor(\pi_1,\pi_2)$ and $T'\in \Mor(\pi_2,\pi_3)$, then $T'\circ T\in \Mor(\pi_1,\pi_3)$. 
\item If $T\in \Mor(\pi_1,\pi_2)$, then $T^*\in \Mor(\pi_2,\pi_1)$.
\end{itemize}
\end{Lem}

The Haar state is also the key to proving the semisimplicity of compact quantum groups.

\begin{Def} Let $\G$ be a compact quantum group. A $\G$-representation is called \emph{indecomposable} if it is not isomorphic to a direct sum of two non-zero representations. 

A non-zero $\G$-representation $\pi$ is called \emph{irreducible} if for any non-zero representation $\pi'$ the existence of $T\in \Mor(\pi',\pi)$ with $T^*T =\id_{\Hsp_{\pi'}}$ implies $TT^*= \id_{\Hsp_{\pi}}$.
\end{Def} 

\begin{Prop} Let $\mathbb{G}$ be a compact quantum group.
\begin{itemize}
\item A $\G$-representation is indecomposable if and only if it is irreducible. 
\item Any $\G$-representation is isomorphic to a  direct sum of irreducible representations.
\end{itemize}
\end{Prop}

\begin{Def} For $\pi$ an irreducible $\G$-representation, we define \[C(\G)_{\pi} = \textrm{linear span } \{U_{\pi}(\xi,\eta)\mid \xi,\eta\in \Hsp_{\pi}\}.\]
\end{Def}

By semisimplicity, we have that $\mathscr{O}(\G)$ is the direct sum of all $C(\G)_{\pi}$.

\begin{Lem}
Each $C(\mathbb{G})_{\pi}$ is a finite-dimensional coalgebra of dimension $\dim(\Hsp_{\pi})^2$, and for $\pi,\pi'$ non-equivalent irreducible representations, one has that $C(\G)_{\pi}$ and $C(\G)_{\pi'}$ are orthogonal with respect to \[\langle g,h\rangle = \varphi_{\G}(g^*h).\] 
\end{Lem}

\section{Actions of compact quantum groups}\label{SecAc}

Let $\mathscr{O}(\X)$ be an algebra, and $(\mathscr{O}(\G),\Delta)$ a Hopf algebra. We recall that a \emph{coaction} of $(\mathscr{O}(\G),\Delta)$ on $\mathscr{O}(\X)$ is an algebra homomorphism \[\alpha: \mathscr{O}(\X)  \rightarrow \mathscr{O}(\X)\atimes \mathscr{O}(\G)\] such that 
\begin{itemize}
\item $(\id_{\X}\otimes \Delta)\circ \alpha = (\alpha\otimes \id_{\G})\circ \alpha$,
\item $(\id_{\X}\otimes \varepsilon)\circ \alpha = \id_{\X}$. 
\end{itemize}

In case we are dealing with $^*$-algebras, we ask that $\alpha$ is $^*$-preserving. We will also use Sweedler notation for coactions, \[\alpha(a) = a_{(0)}\otimes a_{(1)},\quad (\id_{\X}\otimes \Delta)\alpha(a)=a_{(0)}\otimes a_{(1)}\otimes a_{(2)},\quad\ldots\] 

Turning to C$^*$-algebras and compact quantum groups, we can a priori not state in general the counit condition. But, as in the definition of a compact quantum group this condition can be substituted by a density condition leading to the following definition introduced in \cite[Definition 1.4]{Pod95}.

\begin{Def}\label{DefActPod} A (continuous) \emph{right action} $\mathbb{X}\curvearrowleft \mathbb{G}$ consists of
\begin{itemize}
\item a compact quantum group $\mathbb{G}$,
\item a locally compact quantum space $\mathbb{X}$, and
\item a $^*$-homomorphism, called \emph{right coaction}, \[\alpha: C_0(\mathbb{X}) \rightarrow C_0(\mathbb{X})\otimes C(\mathbb{G})\]
\end{itemize}
such that
\begin{itemize}
\item the \emph{coaction property} holds,  \[(\alpha\otimes \id_{\G})\circ \alpha = (\id_{\X}\otimes \Delta)\circ \alpha,\] and
\item the following density condition, called \emph{Podle\'{s} condition}, holds,\[\lbrack \alpha(C_0(\mathbb{X}))(1_{\X}\otimes C(\mathbb{G}))\rbrack = C_0(\mathbb{X})\otimes C(\mathbb{G}).\]
\end{itemize}
\end{Def} 

In this case, we also write $\mathbb{G}\curvearrowright C_0(\mathbb{X})$, where sides are changed to mimick the contravariant nature of taking function algebras. 

Unlike for the comultiplication, we do not assume from the outset that the coaction map $\alpha$ is injective. Indeed, in this case it is easy to give examples where the injectivity does not hold, see for instance Example \ref{ExSolNI}.

Of course, one can as well define the notion of \emph{left} action of a compact quantum group on a locally compact quantum space. If $\G$ is a compact quantum group, denote by $\G^{\mathrm{op}}$ the compact quantum group determined by \[C(\G^{\mathrm{cop}}) = C(\G),\quad \Delta_{\G^{\mathrm{cop}}} = \Delta_{\G}^{\mathrm{op}} = \varsigma \circ \Delta,\] where \[\varsigma: C(\G)\otimes C(\G)\rightarrow C(\G)\otimes C(\G),\quad g\otimes h\mapsto h\otimes g.\] Then we have a one-to-one correspondence 
 \[\G \overset{\alpha}{\curvearrowright} \X\quad \leftrightarrow \quad \X\overset{\alpha^{\mathrm{op}}}{\curvearrowleft} \G^{\mathrm{cop}},\] where $\alpha^{\mathrm{op}} = \varsigma\circ \alpha$.

We now give several examples.

\begin{Exa} Let $\G$ be a compact quantum group. Then $\G\overset{\Delta}{\curvearrowleft} \G$ by \[\Delta: C(\G)\rightarrow C(\G)\otimes C(\G).\]
\end{Exa}

The following lemma shows that compact quantum group actions reduce to the ordinary notion of `continuous group action on a C$^*$-algebra' in the case of $\G$ classical.

\begin{Lem} Let $G$ be a compact Hausdorff group with a continuous (left) action $\alpha$ on a C$^*$-algebra $C_0(\mathbb{X})$, that is,
\begin{itemize}
\item the map \[G\times A \rightarrow A, \quad (g,a)\mapsto \alpha_g(a)\] is continuous,
\item for all $g,h\in G$, $\alpha_{gh}= \alpha_g\circ \alpha_h$,
\item each $\alpha_g$ is a $^*$-automorphism.
\end{itemize}
Then $\mathbb{X}\curvearrowleft G$ is a continuous right action (in the sense of Definition \ref{DefActPod}) by \[\alpha: C_0(\mathbb{X})\rightarrow C_0(\mathbb{X})\otimes C(G) \cong C(G,C_0(\mathbb{X})),\quad a\mapsto \left(\alpha(a): g\mapsto \alpha_g(a)\right).\] 

Conversely, all actions $\X \curvearrowleft G$ arise in this way.
\end{Lem}

\begin{proof} Assume first that $G$ is just a compact Hausdorff space, without any group structure. Then the collection of continuous maps $G\rightarrow C_0(\X)$ forms a C$^*$-algebra $C(G,C_0(\X))$ by pointwise multiplication and $^*$-structure, with the uniform norm. By basic functional analysis, one then has a one-to-one correspondence between 
\begin{itemize}
\item continuous maps \[G\times C_0(\X)\rightarrow C_0(\X), \quad(g,a)\mapsto \alpha_g(a)\] for which each $\alpha_g$ is a linear endomorphism, and
\item continuous linear maps \[\alpha:C_0(\X)\rightarrow C(G,C_0(\X)),\quad a\mapsto (g\mapsto \alpha_g(a)).\]
\end{itemize}
Moreover, one easily sees that $\alpha$ is a $^*$-homomorphism if and only if each $\alpha_g$ is a $^*$-endomorphism. 

Using a partition of unity for $G$ and the definition of the minimal tensor product, we furthermore get a $^*$-isomorphism of C$^*$-algebras 
\[C_0(\mathbb{X})\otimes C(G) \overset{\cong}{\rightarrow} C(G,C_0(\mathbb{X})), \quad a\otimes f \mapsto (g\mapsto f(g)a).\] 

Assume now that $G$ is a compact Hausdorff group. Since we also have a $^*$-isomorphism \[C(\mathbb{X})\otimes C(G)\otimes C(G)\cong C(G\times G,C(\mathbb{X})),\]
it follows from the above and an easy verification that we get a one-to-one correspondence between 
\begin{itemize}
\item continuous maps $\alpha:G\times C_0(\X) \rightarrow C_0(\X)$ for which each $\alpha_g$ is a $^*$-endomorphism, and $\alpha_g\alpha_h = \alpha_{gh}$ for all $g,h\in G$, and
\item $^*$-homomorphisms $\alpha: C_0(\mathbb{X})\rightarrow C_0(\mathbb{X})\otimes C(\G)$ such that the coaction property holds. 
\end{itemize}

What remains to be done is to relate the Podle\'{s} condition to $G$ acting by $^*$-automorphisms. The latter is easily seen to be equivalent to $\alpha_e$ acting by the identity (where $e$ is the unit of $G$). But assume that $(g,a)\rightarrow \alpha_g(a)$ is a continuous action by $^*$-endomorphisms. Then we have a $^*$-homomorphism \[\widetilde{\alpha}:C_0(\mathbb{X})\otimes C(G)\rightarrow C_0(\mathbb{X})\otimes C(G),\quad a\otimes f\mapsto \alpha(a)(1\otimes f),\] and we see that the Podle\'{s} condition is satisfied if and only if $\widetilde{\alpha}$ is surjective.

Now on the level of $C_0(\X)\otimes C(G)\cong C(G,C_0(\mathbb{X}))$, we have \[\forall F\in C(G,C(\X)),\quad \widetilde{\alpha}(F)(g) = \alpha_g(F(g)).\] So assume first that $\alpha_e= \id_{\mathbb{X}}$. Then $\widetilde{\alpha}$ has the inverse $\widetilde{\beta}$, \[\widetilde{\beta}(F)(g) = \alpha_{g^{-1}}(F(g)).\] Hence $\widetilde{\alpha}$ is surjective. 

Conversely, assume $\alpha_e\neq  \id_{\mathbb{X}}$. This implies $\alpha_e$ is a non-trivial idempotent $^*$-endomorphism. Put $C_0(\mathbb{X}_e) = \alpha_e(C_0(\mathbb{X}))\neq C_0(\mathbb{X})$. Then \[\forall g\in G,\quad \alpha_g(C_0(\mathbb{X})) = \alpha_e(\alpha_g(C_0(\mathbb{X})))\subseteq C_0(\mathbb{X}_e).\] But for $a\notin C_0(\mathbb{X}_e)$, we then have that the constant map $g\mapsto a$ is not in the range of $\widetilde{\alpha}$, and hence $\widetilde{\alpha}$ is not surjective.
\end{proof}

\begin{Exa} Assume $G$ is a compact Hausdorff group, and $X$ a locally compact Hausdorff space. Then \[X\curvearrowleft G \textnormal{ continuously} \quad \Leftrightarrow \quad G\curvearrowright C_0(X),\quad \alpha_g(f)(x) = f(x\cdot g).\]
\end{Exa}

\begin{Exa} Consider the sphere \[S^{N-1} = \{z=(z_1,\ldots,z_N)\in \R^N\mid \sum_i z_i^2 = 1\},\] and let $O(N)$ be the orthogonal group, \[O(N) = \{u \in M_N(\R)\mid u^tu = I_N = uu^t\}.\] Then $S^{N-1} \curvearrowleft O(N)$ by \[(z,u) \mapsto zu.\]
 \end{Exa}  

 \begin{Exa}\label{ExaCuntz} Consider the \emph{Cuntz algebra}, \[\mathcal{O}_N = C^*(V_1,\ldots,V_N\mid V_i^*V_j = \delta_{ij}, \sum_i V_iV_i^* = 1),\] where $C^*(\cdot)$ means `the universal C$^*$-algebra generated by'.
 Let $U(N)$ be the unitary group, \[U(N) = \{u \in M_N(\C)\mid u^*u = I_N = uu^*\}.\]
Then $U(N) \curvearrowright \mathcal{O}_N$ by \[\alpha_u(V_i) = \sum_{j} u_{ji} V_j.\]
In particular, by restriction $S^1 \curvearrowright \mathcal{O}_N$, \[\alpha_z(V_i) = zV_i.\]
\end{Exa}

\begin{Exa}[\cite{BaG10}]\label{ExaFreeSphere} Consider the \emph{free sphere} \[C(S^{N-1}_+) = C^*(V_1,\ldots, V_N\mid V_i = V_i^*, \sum_i V_i^2 = 1).\] Then $ S^{N-1}_+\curvearrowleft O(N)$ by \[\alpha_u(V_i) = \sum_{j} u_{ji} V_j.\]
\end{Exa}

\begin{Exa} Let $\G$ be a compact quantum group, and $\pi$ a $\G$-representation. Then $\G\curvearrowright B(\Hsp_{\pi})$ by the \emph{adjoint action} \[\mathrm{Ad}_{\pi}: B(\Hsp_{\pi})\rightarrow B(\Hsp_{\pi})\otimes C(\G),\]\[ \xi\eta^* \mapsto \delta_{\pi}(\xi)\delta_{\pi}(\eta)^* = U_{\pi}(\xi\eta^*\otimes 1_{\G})U_{\pi}^*.\]
\end{Exa}

Note that, for $G$ a compact Hausdorff group with representation $\pi$, \[(\mathrm{Ad}_{\pi})_g(x) = \pi_gx\pi_g^*,\quad x\in B(\Hsp_{\pi}).\]

\begin{Exa} If $\X\overset{\alpha}{\curvearrowleft} \G$, then we can let $\G$ act on the `Alexandroff compactification' $C(\X^{\bullet}) = C_0(\X)\oplus \C$ by extending the coaction unitally.
\end{Exa}

A general method to construct examples of compact quantum group actions is to complete, in a C$^*$-algebraic sense, purely algebraic examples. We have already seen instances of this procedure in Example \ref{ExaCuntz} and Example \ref{ExaFreeSphere}.

Let us first recall in more detail the notion of \emph{universal C$^*$-envelope}.

 \begin{Def} For $\mathscr{O}(\X)$ a $^*$-algebra, we say that $\mathscr{O}(\X)$ \emph{admits a universal C$^*$-envelope} if for each $a\in \mathscr{O}(\X)$ there exists $C_a \geq 0$ such that \[\|\pi(a)\|\leq C_a\] for all $^*$-representations of $\mathscr{O}(\X)$ as (bounded) operators on some Hilbert space. 
 \end{Def}
 
 We then write $\|a\|_u$ for the infimum of all possible $C_a$. One can show that $\|\cdot\|_u$ is a submultiplicative norm on $\mathscr{O}(\X)/I$, where $I$ is the ideal of all elements $a$ with $\|a\|_u= 0$, and that $\|\cdot\|_u$ satisfies the C$^*$-identity.

\begin{Def}  We denote by $C_0(\X_u)$ the completion of $\mathscr{O}(\X)/I$ with respect to $\|\cdot\|_u$, and call it the \emph{universal C$^*$-envelope} of $\mathscr{O}(\X)$. \end{Def}

\begin{Exa} Let $\G$ be a compact quantum group. Since \[\|\rho\left(U(\xi,\eta)\right)\| = \|(\xi^*\otimes 1)(\id\otimes \pi)(U_{\pi})(\eta\otimes 1)\|\leq \|\xi\|\|\eta\|\] for any $^*$-representation $\rho$ of $\mathscr{O}(\G)$ and any representation $\pi$ of $\G$, it follows that $\mathscr{O}(\G)$ admits a universal C$^*$-algebraic completion, which we write $C_u(\G) = C(\G_u)$. 

As $\mathscr{O}(\G)$ embeds into $C(\G)$, it also  
embeds into $C(\G_u)$. Moreover, by its universal property, $C(\G_u)$ inherits a coproduct from $\mathscr{O}(\G)$, and $\G_u$ becomes a compact quantum group in its own right. We then have canonical, $\Delta$-preserving, surjective $^*$-homomorphisms \[C(\G_u)\overset{\pi_u}{\twoheadrightarrow} C(\G)\overset{\pi_{\red}}{\twoheadrightarrow} C(\G_{\red}).\]
Note that we still have $\mathscr{O}(\G_u) = \mathscr{O}(\G)$.
\end{Exa}

\begin{Lem}\label{LemAlgtoAn} Let $\mathscr{O}(\X)$ be a $^*$-algebra with a Hopf $^*$-algebraic coaction \[\alpha: \mathscr{O}(\X) \rightarrow \mathscr{O}(\X)\underset{\alg}{\otimes} \mathscr{O}(\G).\] Assume $\mathscr{O}(\X)$ admits a universal C$^*$-envelope. \vspace{0.2cm}

Then $\alpha$ extends to coaction \[\alpha_u: C_0(\X_u)\rightarrow C_0(\X_u)\otimes C(\G_u)\] satisfying the Podle\'{s} condition.
\end{Lem}

\begin{proof} The existence of $\alpha_u$ as a $^*$-homomorphism is clear by universality. The fact that $\alpha_u$ satisfies the coaction property is then clear by continuity. 

To see that $\alpha_u$ satisfies the Podle\'{s} condition we compute for $a\in \mathscr{O}(\X)$\begin{eqnarray*} \alpha(a_{(0)})(1_{\X}\otimes S(a_{(1)})) &=& a_{(0)}\otimes a_{(1)}S(a_{(2)})\\ &=& a_{(0)}\otimes \varepsilon(a_{(1)})1_{\G}\\  &=& a\otimes 1_{\G}.\end{eqnarray*} Hence \[\alpha(\mathscr{O}(\X))(1_{\X}\otimes \mathscr{O}(\G)) =\mathscr{O}(\X)\aotimes \mathscr{O}(\G),\] and \[\lbrack \alpha_u(C_0(\X_u))(1_{\X_u}\otimes C(\G_u))\rbrack =C_0(\X_u)\otimes C(\G_u).\]
\end{proof}

We now construct several further actions of compact quantum groups.

Let us first return to actions on the Cuntz algebras from a more coordinate-free perspective. 

\begin{Def} Let $\Hsp$ be a finite-dimensional Hilbert space. The Cuntz C$^*$-algebra $\mathcal{O}(\Hsp)$ is defined by the following universal properties:
\begin{itemize}
\item $\Hsp\subseteq \mathcal{O}(\Hsp)$ linearly,
\item $\mathcal{O}(\Hsp)$ is generated by $\Hsp$ as a C$^*$-algebra,
\item $\xi^*\eta = \langle \xi,\eta\rangle1_{\mathcal{O}(\Hsp)}$ for $\xi,\eta\in \Hsp$,
\item $\sum_i e_ie_i^*  = 1$ for $\{e_i\}$ an orthonormal basis.
\end{itemize}
\end{Def}

For example, we then have $\mathcal{O}_N = \mathcal{O}(\C^N)$.  The next example was introduced in \cite{KNW92}, see also \cite{Gab14}.

\begin{Exa} Let $\G$ be a compact quantum group, and $\pi$ a $\G$-representation. Then we have an action $\G\curvearrowright \mathcal{O}(\mathcal{H}_{\pi})$ by \[\alpha_{\pi}: \mathcal{O}(\mathcal{H}_{\pi}) \rightarrow \mathcal{O}(\mathcal{H}_{\pi})\otimes C(\G),\quad \xi\mapsto \delta_{\pi}(\xi).\]
\end{Exa}

\begin{Def}[\cite{VDW96}] The \emph{free orthogonal quantum group} $O^+_N$ is defined as the universal C$^*$-algebra \[C(O^+_N)=C^*(U_{ij}\mid 1\leq i,j\leq N, U_{ij}^* = U_{ij} \textnormal{\textit{ and }} U=(U_{ij})_{i,j}\textnormal{\textit{ unitary}})\] with the coproduct $\Delta$ characterised by 
\[\Delta(U_{ij}) = \sum_k U_{ik}\otimes U_{kj}.\]
\end{Def}

It is easy to show that $O^+_N$ is indeed a compact quantum group. 

\begin{Exa} The compact quantum group $O^+_N$ acts on the free quantum sphere $S_+^{N-1}$ by \[\alpha(V_i) = \sum_{j} V_j\otimes U_{ji}.\]
\end{Exa}

Non-classical quantum groups can also act on classical spaces.

\begin{Def}[\cite{Wan98}] The \emph{free permutation group} $\mathrm{Sym}_N^+$ is defined as the universal C$^*$-algebra \[C^*(U_{ij}\mid 1\leq i,j\leq N, U_{ij}^* = U_{ij}=U_{ij}^2 \textnormal{\textit{ and }} U=(U_{ij})_{i,j}\textnormal{\textit{ unitary}})\]equipped with the coproduct \[\Delta(U_{ij}) = \sum_k U_{ik}\otimes U_{kj}.\]
\end{Def}

It is again easy to show that $\mathrm{Sym}_N^+$ is a compact quantum group.

\begin{Exa} Let $X_N= \{1,2,\ldots, N\}$. Then $X_N \curvearrowleft \mathrm{Sym}_N^+$ by \[\alpha: C(X_N)\rightarrow C(X_N)\otimes C(\mathrm{Sym}_N^+),\quad \delta_i \mapsto \sum_j \delta_j\otimes U_{ji}.\] Here the $\delta_j$ denote the Dirac functions $\delta_j(i) = \delta_{j,i}$. 
\end{Exa}

The above phenomenon of quantum groups acting on classical spaces is however much rarer than that of classical groups acting on quantum spaces, as shown by the work of D. Goswami and collaborators, see in particular \cite{GoJ13}. Some other examples can be found in \cite{BBC07,Hua13}. 

As a final example, let us consider \emph{duals} of discrete groups. 

\begin{Def} Let $\Gamma$ be a discrete group. We define the compact quantum group $\widehat{\Gamma}_u$ as the universal group C$^*$-algebra $C(\widehat{\Gamma}_u) = C^*_u(\Gamma)$ equipped with the coproduct \[\Delta(\lambda_g)= \lambda_g\otimes \lambda_g,\] where $\lambda_g$ for $g\in \Gamma$ denote the generators of $C^*_u(\Gamma)$. 
\end{Def}

\begin{Exa}[C$^*$-algebraic bundles and $\Gamma$-graded C$^*$-algebras]
Assume that we have the following data:
\begin{itemize}
\item a discrete group $\Gamma$,
\item Banach spaces $A_g=\{a_g\}$ with associative contractive multiplications \[A_g\times A_h\rightarrow A_{gh},\]
\item antilinear, involutive, isometric maps $*: A_g \rightarrow A_{g^{-1}}$
\end{itemize}
such that 
\begin{itemize} 
\item $\|b^*b\| = \|b\|^2$ for $b\in A_g$, 
\item $b^*b\geq 0$ in (the C$^*$-algebra) $A_e$ for $b\in A_g$.
\end{itemize}
Then $\widehat{\Gamma}_u\curvearrowright A$, the universal C$^*$-envelope of $\oplus_{g} A_g$ with its obvious $^*$-algebra structure, by \[\alpha: A \rightarrow A\otimes C(\widehat{\Gamma}_u),\quad a_g \mapsto a_g\otimes \lambda_g.\]
\end{Exa}

Still further important examples will be introduced throughout the remainder of this article.

\section{Isotypical components and algebraic core}

The basic results and notions in this section can be found in \cite{Pod95}. 

\begin{Def} Let $\mathbb{X}\overset{\alpha}{\curvearrowleft} \mathbb{G}$. We define the \emph{quantum orbit space} \[\mathbb{Y} = \mathbb{X}/\mathbb{G}\] by the C$^*$-algebra \[C_0(\mathbb{Y}) = C_0(\mathbb{X})^{\mathbb{G}}= \{a\in C_0(\mathbb{X})\mid \alpha(a) = a\otimes 1_{\mathbb{G}}\}.\]
\end{Def}

\begin{Exa}
If $G\overset{\alpha}{\curvearrowright} C_0(\mathbb{X})$, then \[C_0(\mathbb{Y}) = C_0(\mathbb{X})^{G} =  \{a \in C_0(\mathbb{X})\mid \alpha_g(a) = a \textrm{ for all }g\in G\}.\]
\end{Exa}
\begin{Exa}
 If $X\overset{\alpha}{\curvearrowleft} G$, then \begin{eqnarray*} C_0(X)^G &=& \{G\textrm{-constant continuous functions on }X\textrm{ vanishing at infinity}\}\\ &\cong& \{\textrm{continuous functions on }Y=X/G\textrm{ vanishing at infinity}\}.\end{eqnarray*}
\end{Exa}

Other examples of quantum orbit spaces can be constructed from representation theory.

\begin{Exa} Let $\pi$ be a $\G$-representation, and let $\Ad_{\pi}$ be the adjoint action on $B(\Hsp_{\pi})$. Then
\[B(\Hsp_{\pi})^{\Ad_{\pi}} = \Mor(\pi,\pi).\]
Indeed, this follows immediately from the formula \[\Ad_{\pi}(x) = U_{\pi}(x\otimes 1)U_{\pi}^*.\]
\end{Exa}

Since the original group action has been quotiented out, quantum orbit spaces will not have an action anymore by the original quantum group. However, if there was more symmetry to begin with, taking into account the quantum group action, the extra symmetry will pass to the quotient.

\begin{Def} Let $\X\overset{\alpha}{\curvearrowleft}\G$ and $\mathbb{H}\overset{\beta}{\curvearrowright} \X$. We say that the actions $\alpha$ and $\beta$ \emph{commute} if
\[(\beta\otimes \id_{\G})\alpha = (\id_{\mathbb{H}}\otimes \alpha)\beta.\]
\end{Def}

\begin{Exa}\label{ExaHom} Assume that $\X\overset{\alpha}{\curvearrowleft}\G$ and $\mathbb{H}\overset{\beta}{\curvearrowright} \X$ commute. Then we have a continuous action $\mathbb{H}\curvearrowright \X/\G$ by \[\beta_{\mid C_0(\X/\G)}: C_0(\X/\G)\rightarrow C(\mathbb{H})\otimes C_0(\X/\G).\]
\end{Exa} 

To prove the latter statement, we have to make use of the natural \emph{conditional expectation} from $C_0(\X)$ onto $C_0(\X/\G)$, whereby one `integrates the action out on the fibers over the quotient space'. 

\begin{Def} Let $\mathbb{X}\overset{\alpha}{\curvearrowleft} \mathbb{G}$ and $\mathbb{Y} = \mathbb{X}/\mathbb{G}$. The \emph{natural conditional expectation} onto $C_0(\mathbb{Y})$ is the map \[E_{\mathbb{Y}}: C_0(\mathbb{X})\rightarrow C_0(\mathbb{X}),\quad a \mapsto (\id_{\X}\otimes \varphi_{\G})\alpha(a).\] 
\end{Def}

\begin{Lem}\label{LemCondExp} The map $E_{\mathbb{Y}}:C_0(\X)\rightarrow C_0(\X)$ has range $C_0(\mathbb{Y})$ and is
\begin{itemize}
\item idempotent,
\item completely positive,
\item bimodular:  \[E_{\mathbb{Y}}(bac) = bE_{\mathbb{Y}}(a)c,\qquad a\in C_0(\mathbb{X}), b,c\in C_0(\mathbb{Y}),\]
\item non-degenerate: $\lbrack C_0(\mathbb{X})C_0(\mathbb{Y})\rbrack = C_0(\mathbb{X})$.
\end{itemize}
\end{Lem} 
The nondegeneracy can be interpreted as the condition that `Every point of $\mathbb{X}$ is in an orbit (that is, lies over a point of $\mathbb{Y}$)'.  
\begin{proof} 
\begin{itemize}
\item To see that $E_{\mathbb{Y}}(C_0(\mathbb{X})) \subseteq C_0(\mathbb{Y})$, we compute \begin{eqnarray*} \alpha(E_{\mathbb{Y}}(a)) &=& \alpha\big{(} (\id_{\X}\otimes \varphi_{\G})\alpha(a)\big{)} \\
&=&(\id_{\X}\otimes \id_{\X}\otimes \varphi_{\G})((\alpha\otimes \id_{\G})\alpha(a)) \\&=& (\id_{\X}\otimes \id_{\X}\otimes \varphi_{\G})((\id_{\X}\otimes \Delta)\alpha(a)) \\ & = & (\id_{\X}\otimes \varphi_{\G})(\alpha(a))\otimes 1_{\mathbb{G}} \\ &=& E_{\mathbb{Y}}(a)\otimes 1_{\mathbb{G}}.\end{eqnarray*}
\item Trivially, $E_{\mathbb{Y}}(b) = b$ for $b\in C_0(\mathbb{Y})$, so in particular $E_{\mathbb{Y}}$ idempotent.
\item The map $E_{\mathbb{Y}}$ is completely positive since states and $^*$-homomorphisms are completely positive.
\item Trivially, $E_{\mathbb{Y}}$ is $C_0(\Y)$-bimodular.
\item Non-degeneracy can be shown as follows: if $(u_{i})_{i}$ is a bounded approximate unit for $C_0(\X)$, then
\[\forall b\in C_0(\X), \quad E_{\Y}(u_{i})b = (\id_{\X}\otimes \varphi_{\G})(\alpha(u_{i})(b\otimes 1_{\G})) \underset{i\rightarrow\infty}{\rightarrow}  b,\]
since $b\otimes 1_{\G} \in \lbrack \alpha(C_0(\X))(1_{\X}\otimes C(\G))\rbrack$.
\end{itemize}
\end{proof}

\begin{Rem} Any map $E_{\Y}: C_0(\X)\rightarrow C_0(\Y)\subseteq C_0(\X)$ of a C$^*$-algebra onto a C$^*$-subalgebra satisfying all conditions in Lemma \ref{LemCondExp} will be called a \emph{conditional expectation}.
\end{Rem} 

We can now prove the claim in Example \ref{ExaHom}: if $a\in C_0(\X/\G)$, then \begin{eqnarray*} \beta(a) &=& \beta(E_{\X/\G}(a)) \\ &=& \beta((\id_{\mathbb{X}}\otimes \varphi_{\G})\alpha(a))\\ &=& (\id_{\mathbb{H}}\otimes \id_{\X}\otimes \varphi_{\G})((\beta\otimes \id_{\G})\alpha(a))\\ &=& (\id_{\mathbb{H}}\otimes \id_{\X}\otimes \varphi_{\G})((\id_{\mathbb{H}}\otimes \alpha)\beta(a)) \\ &=& (\id_{\mathbb{H}}\otimes E_{\X/\G})\beta(a),\end{eqnarray*} which lies in $C(\mathbb{H})\otimes C_0(\X/\G)$. Obviously, $\beta_{\mid C_0(\X/\G)}$ satisfies the coaction property, and it satisfies the Podle\'{s} condition since, by the above calculation \begin{multline*} \left[ (C(\mathbb{H})\otimes 1)\beta(C_0(\X/\G))\right] = (\id_{\mathbb{H}}\otimes E_{\X/\G})\left[(C(\mathbb{H})\otimes 1_{\X})\beta(C_0(\X))\right]  \\ = (\id_{\mathbb{H}}\otimes E_{\X/\G})(C(\mathbb{H})\otimes C_0(\X)) = C(\mathbb{H})\otimes C_0(\X/\G).\end{multline*}

Let us now look at some more examples of actions and their associated quantum orbit spaces. 

\begin{Exa} If $G$ is a compact Hausdorff group and $G\overset{\alpha}{\curvearrowright} C_0(\mathbb{X})$, then \[E_{\mathbb{X}/G}(a) = \int_G \alpha_g(a)\rd \mu(g),\] where $\mu$ is the normalized Haar measure of $G$.
\end{Exa}

\begin{Exa} If $G$ is a compact group, $X$ a locally compact space with $X\overset{\alpha}{\curvearrowleft} G$, then $E_{X/G}$ is indeed integration over orbits, \[E_{X/G}(f)(xG) = \int_G f(xg)\rd \mu(g).\]
\end{Exa}

\begin{Exa} For the action $S^1 \curvearrowright \mathcal{O}(\Hsp)$ determined by $\alpha_z(\xi) = z\xi$, we have  \begin{eqnarray*} E_{\Y}(\xi_1\ldots \xi_N\eta_1^*\ldots \eta_M^*) &=& \int_{S^1} z^{N-M} (\xi_1\ldots \xi_N\eta_1^*\ldots \eta_M^*) \rd z 
\\ &=& \delta_{M,N}\xi_1\ldots \xi_N\eta_1^*\ldots \eta_M^* .\end{eqnarray*}
\end{Exa}

We now define, for an action of a compact quantum group, the notion of an \emph{isotypical component}.

\begin{Def} Let $\mathbb{X}\overset{\alpha}{\curvearrowleft} \mathbb{G}$, and let $\pi$ be a $\mathbb{G}$-representation.

The \emph{intertwiner space} between $\pi$ and $\alpha$ is defined as \[\Mor(\pi,\alpha) = \{T: \Hsp_{\pi} \rightarrow C_0(\mathbb{X})\mid \alpha(T\xi) = (T\otimes \id_{\G})\delta_{\pi}(\xi)\}.\]

When $\pi$ is irreducible, we call \emph{$\pi$-isotypical component} (or \emph{$\pi$-spectral subspace}) the subspace \[C_0(\mathbb{X})_{\pi} = \textrm{lin.~ span } \{T\xi \mid \xi\in \Hsp_{\pi}, T\in \Mor(\pi,\alpha)\}\subseteq C_0(\mathbb{X}).\]
\end{Def} 

Note that by its definition each $C_0(\mathbb{X})_{\pi}$ is a $C_0(\mathbb{Y})$-bimodule with \[\alpha(C_0(\X)_{\pi})\subseteq C_0(\X)_{\pi}\otimes C(\G)_{\pi}.\]

Note further that, when $\X = \G$, we have indeed that $C(\G)_{\pi}$ is the same space as was introduced in Section \ref{SecPrel}. Namely, for each $\xi\in \Hsp_{\pi}$, we have that \[T_{\xi}: \Hsp_{\pi}\mapsto C(\G)_{\pi},\quad \eta\mapsto (\xi^*\otimes \id_{\G})\delta_{\pi}(\eta)=U_{\pi}(\xi,\eta)\] is in $\Mor(\pi,\Delta)$. Conversely, put \begin{equation}\label{EqElChi} \chi_{\pi} = \sum_i \mathrm{Tr}(Q_{\pi})U_{\pi}(e_i,Q_{\pi}^{-1}e_i),\end{equation} with $e_i$ an orthonormal basis for $\Hsp_{\pi}$. Then for $\xi,\eta\in \Hsp_{\pi}$, \[\varphi_{\G}(U_{\pi}(\xi,\eta)\chi_{\pi}^*) = \langle \xi,\eta\rangle,\] while $\varphi_{\G}(h\chi_{\pi}^*)= 0$ for $h\in C(\G)_{\pi'}$ with $\pi'$ inequivalent with $\pi$. Hence, for $T\in \Mor(\pi,\Delta)$, \[T\xi = (\id_{\G} \otimes \varphi_{\G})(\Delta(T\xi)(1_{\G}\otimes \chi_{\pi}^*)),\]which lies in $C(\G)_{\pi}$ by the orthogonality and finite-dimensionality of the $C(\G)_{\pi'}$, and the density of $\mathscr{O}(\G)$ in $C(\G)$.

\begin{Lem}\label{LemClosSpec} Each $C_0(\X)_{\pi}$ is closed in $C_0(\X)$ for the C$^*$-algebra norm.
\end{Lem}
\begin{proof} Define $\chi_{\pi}\in C(\G)_{\pi}$ as in \eqref{EqElChi}, and write \begin{equation}\label{DefEpi}E_{\pi}(a) =   (\id_{\X}\otimes \varphi_{\G})(\alpha(a)(1_{\X}\otimes \chi_{\pi}^*)).\end{equation} We claim that \[C_0(\X)_{\pi} = \{a\in C_0(\X)\mid E_{\pi}(a) = a\},\] from which the closedness immediately follows. 

Indeed, if $T\in \Mor(\pi,\alpha)$, then $E_{\pi}(T\xi) = T\xi$ is immediate. 

Conversely, for $a\in C_0(\X)$ and $\eta\in \Hsp_{\pi}$, consider the (linear!) map \[T: \Hsp_{\pi} \rightarrow C_0(\X),\quad \xi \mapsto  (\id_{\X}\otimes \varphi_{\G})(\alpha(a)(1_{\X}\otimes U_{\pi}(\xi,\eta)^*)).\] Let $\{e_i\}$ be an orthonormal basis of $\Hsp_{\pi}$. By strong left invariance of $\varphi_{\G}$, \begin{eqnarray*} \alpha(T\xi) &=& (\id_{\X}\otimes \id_{\G}\otimes \varphi_{\G})((\id_{\X}\otimes \Delta)\alpha(a) (1_{\X}\otimes 1_{\X} \otimes U_{\pi}(\xi,\eta)^*)) \\ &=& \sum_i (\id_{\X}\otimes \id_{\G}\otimes \varphi_{\G})(\alpha(a)_{13}(1_{\X}\otimes U_{\pi}(e_i,\xi) \otimes U_{\pi}(e_i,\eta)^*))\\ &=& \sum_i T(e_i)\otimes U_{\pi}(e_i,\xi) \\ &=& (T\otimes \id_{\G})\delta_{\pi}(\xi),\end{eqnarray*}  hence $T\in \Mor(\pi,\alpha)$. Since $\chi_{\pi}$ is a linear combination of $U_{\pi}(\xi,\eta)$'s, it follows that $E_{\pi}(C_0(\X)) \subseteq C_0(\X)_{\pi}$.
\end{proof} 

\begin{Def} Let $\mathbb{X}\overset{\alpha}{\curvearrowleft} \mathbb{G}$. Then we define the \emph{Podle\'{s} subalgebra} or \emph{algebraic core} of $C_0(\X)$ to be the set \[\mathscr{O}_{\mathbb{G}}(\mathbb{X}) = \textrm{linear span }\{C_0(\mathbb{X})_{\pi}\mid \pi \textrm{ irreducible}\}\subseteq C_0(\mathbb{X}).\]  
\end{Def} 

\begin{Theorem} The linear space $\mathscr{O}_{\G}(\X)$ is a $^*$-algebra, unital if $\X$ compact. 
\end{Theorem} 
\begin{proof}
If $\X$ is compact, then clearly $\mathscr{O}_{\G}(\X)$ contains the unit since with $\pi_{\epsilon}$ denoting the trivial representation \[\delta_{\epsilon}: \C\rightarrow \C\otimes C(\G),\quad 1\mapsto 1\otimes 1_{\G}\] we have that \[\eta: \C\rightarrow C(\G),\quad 1\mapsto 1_{\G}\] lies in $\Mor(\pi_{\epsilon},\alpha)$. 

In general, we have by linearity and semisimplicity that \[\mathscr{O}_{\mathbb{G}}(\mathbb{X}) = \{T\xi \mid \pi\textrm{ a }\G\textrm{-representation}, \xi\in \Hsp_{\pi}, T\in \Mor(\pi,\alpha)\}\subseteq C_0(\mathbb{X}).\] 
If then $a = T\xi$ and $b = T'\eta$, and $m$ the multiplication map from $\mathscr{O}_{\G}(\X)\atimes \mathscr{O}_{\G}(\X)$ to $\mathscr{O}_{\G}(\X)$, we have \[ab = m(T\xi\otimes T'\eta),\] where $m\circ (T\otimes T')$ in $\Mor(\pi_1\otimes \pi_2,\alpha)$ since $\alpha$ is a homomorphism.

Finally,  if $a = T\xi$, then \[a^* = T^{\dag}(\xi^*),\] where $T^{\dag}:\eta^*\mapsto (T\eta)^*$ is in $\Mor(\bar{\pi},\alpha)$ since $\alpha$ is $^*$-preserving.
\end{proof} 

For example, by the remark above Lemma \ref{LemClosSpec}, we have $\mathscr{O}(\G) = \mathscr{O}_{\G}(\G)$. 

\begin{Prop} Let $\mathbb{X}\overset{\alpha}{\curvearrowleft} \mathbb{G}$. Then $\alpha$ restricts to a Hopf $^*$-algebraic right coaction \[\alpha_{\alg}: \mathscr{O}_{\mathbb{G}}(\mathbb{X})\rightarrow \mathscr{O}_{\mathbb{G}}(\mathbb{X})\underset{\alg}{\otimes} \mathscr{O}(\mathbb{G}).\]
\end{Prop}

 \begin{proof}
 For $a = T\xi$ with $\xi\in \Hsp_{\pi}$ and $T\in \Mor(\pi,\alpha)$, we have 
 $a\in C_0(\mathbb{X})_{\pi}$ and \[\alpha(a) = \alpha(T\xi) = (T\otimes \id_{\G})\delta_{\pi}(\xi)\in C_0(\mathbb{X})_{\pi}\underset{\alg}{\otimes} C(\mathbb{G})_{\pi} \subseteq  \mathscr{O}_{\mathbb{G}}(\mathbb{X})\underset{\alg}{\otimes} \mathscr{O}(\mathbb{G}).\]
The fact that $\alpha_{\alg}$ satisfies the coaction property is immediate. To see that
$\alpha_{\alg}$ is counital, take $a = T\xi$ with $\xi\in \Hsp_{\pi}$ and $T\in \Mor(\pi,\alpha)$. Then 
\[(\id_{\X}\otimes \varepsilon)\alpha(T\xi) = T((\id_{\Hsp_{\pi}}\otimes \varepsilon)\delta_{\pi}(\xi)) = T\xi.\]
  \end{proof}

The following theorem is part of \cite[Theorem 1.5]{Pod95}.

\begin{Theorem}\label{TheoDensPod} Let  $\mathbb{X}\overset{\alpha}{\curvearrowleft} \mathbb{G}$. Then $\mathscr{O}_{\mathbb{G}}(\mathbb{X})$ is dense in $C_0(\mathbb{X})$. 
\end{Theorem}

\begin{proof} Since 
\[\lbrack \alpha(C_0(\mathbb{X}))(1_{\X}\otimes C(\mathbb{G}))\rbrack = C_0(\mathbb{X})\otimes C(\mathbb{G})\supseteq C_0(\mathbb{X})\otimes \C,\] and since $\mathscr{O}(\G)$ is dense in $C(\G)$, we have that \[C_0(\mathbb{X}) =\lbrack \{(\id_{\X}\otimes \varphi_{\G})(\alpha(a)(1_{\X}\otimes h))\mid a\in C_0(\mathbb{X}), h\in \mathscr{O}(\mathbb{G})\}\rbrack.\] But, as follows from the proof of Lemma \ref{LemClosSpec}, we have \[\{(\id_{\X}\otimes \varphi_{\G})(\alpha(a)(1_{\X}\otimes h))\mid a\in C_0(\mathbb{X}), h\in \mathscr{O}(\mathbb{G})\} \subseteq \mathscr{O}_{\G}(\X).\]
\end{proof}

\begin{Lem} Let $\X\overset{\alpha}{\curvearrowleft} \G$. Then $E_{\X/\G}$ is faithful on $\mathscr{O}_{\G}(\X)$: \[\forall a\in \mathscr{O}_{\G}(\X), \quad E_{\X/\G}(a^*a) = 0 \quad \Rightarrow \quad a = 0.\]
\end{Lem} 
\begin{proof} Assume that $a\in \mathscr{O}_{\G}(\X)$ with  $E_{\X/\G}(a^*a) = 0$. For $\omega$ a positive functional on $C_0(\X)$, we then have \[0 = \omega (E_{\X/\G}(a^*a)) = \varphi_{\G}((\omega\otimes \id_{\G})\alpha(a^*a)).\]
Since $(\omega\otimes \id_{\G})\alpha(a^*a)\in \mathscr{O}(\G)$ is positive in $C(\G)$, it follows from Lemma \ref{LemPhiFaith} that \[(\omega\otimes \id_{\G})\alpha(a^*a)=0.\]
Hence, as we are working with the spatial tensor product, $\alpha(a^*a)= 0$. Applying the counit to the second leg, we get $a^*a=0$, and so $a=0$.
\end{proof} 

From the fact that $\alpha(C_0(\X)_{\pi})\subseteq C_0(\X)_{\pi}\otimes C(\G)_{\pi}$, we have for $\pi\ncong \pi'$ that \[E_{\X/\G}(a^*b)=0,\quad a\in C_0(\X)_{\pi}, b\in C_0(\X)_{\pi'}.\] Hence,  \[\mathscr{O}_{\G}(\X) = \underset{\pi\in  \textnormal{\textit{Irrep}}(\G)}\sum^{\!\!\oplus}\,\, C_0(\X)_{\pi},\] where the direct sum is over a maximal family of non-equivalent irreducible representations of $\G$. 

In general the coaction map $\alpha$ associated to an action of a compact quantum group need not be faithful. The following results are taken from \cite{Sol11}. 

\begin{Exa}\label{ExSolNI} Let $\Gamma$ be a non-amenable discrete group, so that $C^*_u(\Gamma)\ncong C^*_{\red}(\Gamma)$. Then \[\Delta: C^*_u(\Gamma)\rightarrow C^*_u(\Gamma)\otimes C^*_{\red}(\Gamma),\quad \lambda_g\mapsto \lambda_g\otimes \lambda_g\] defines a continuous action $\widehat{\Gamma}_u\curvearrowleft \widehat{\Gamma}_{\red}$, but is non-injective by Fell's absorption principle.
\end{Exa}

However, one can always get rid of this nuissance in a canonical way.

\begin{Lem} Let $\X\overset{\alpha}{\curvearrowleft} \G$. Define $C_0(\X') = C_0(\X)/\Ker(\alpha)$. Then $\X\overset{\alpha}{\curvearrowleft} \G$ descends to a continuous action $\X'\overset{\alpha'}{\curvearrowleft} \G$, with $\alpha'$ is injective.
\end{Lem} 

The proof will use in an essential way that $\Delta$ is assumed to be injective!

\begin{proof}It is immediate that $\alpha'$ is a well-defined coaction satisfying the Podle\'{s} condition. To prove injectivity, write $\rho: C_0(\X)\rightarrow C_0(\X)/\Ker(\alpha)$ for the canonical projection map, so that by definition \[\alpha'(\rho(a)) = (\rho\otimes \id_{\G})\alpha(a).\] Assume that $\alpha'(\rho(a))=(\rho\otimes \id_{\G})\alpha(a)=0$. Then for any $\omega \in C(\G)^*$, we have $\rho((\id_{\X}\otimes \omega)\alpha(a))=0$, hence \[ 0 =  \alpha((\id_{\X}\otimes \omega)\alpha(a))\\ = (\id_{\X}\otimes \id_{\G}\otimes \omega)((\id_{\X}\otimes \Delta)\alpha(a)).\] Since $\omega$ was arbitrary, and since $\Delta$ is injective by assumption, we have $\alpha(a)=0$, and hence $\rho(a)=0$.
\end{proof} 

It is easy to see that the natural map $\mathscr{O}_{\G}(\X) \rightarrow C_0(\X')$ is injective. The following proposition shows that we in fact have an isomorphism $\mathscr{O}_{\G}(\X) \cong \mathscr{O}_{\G}(\X')$.

\begin{Prop}  With $C_0(\X') = C_0(\X)/\Ker(\alpha)$, one has $\mathscr{O}_{\G}(\X) = \mathscr{O}_{\G}(\X')$. 
\end{Prop} 
\begin{proof} The natural injection $\mathscr{O}_{\G}(\X) \rightarrow C_0(\X')$ clearly has range in $\mathscr{O}_{\G}(\X')$, as the $^*$-homomorphism $\rho:C_0(\X) \rightarrow C_0(\X')$ intertwines the coactions. If however $b\in C_0(\X')_{\pi}$ for an irreducible representation $\pi$, pick $a\in C_0(\X)$ with $\rho(a) =b$. Then using the map $E_{\pi}$ of \eqref{DefEpi}, we have $\rho(E_{\pi}(a)) =b$ by equivariance of $\rho$. It follows that we may assume $a\in C_0(\X)_{\pi} \subseteq \mathscr{O}_{\G}(\X)$. Hence the map $\mathscr{O}_{\G}(\X) \rightarrow \mathscr{O}_{\G}(\X')$ is also surjective.  
\end{proof}

\section{Universal and reduced actions}

In this section, we show that to any action of a compact quantum group can be associated canonically an action of its universal completion and its reduced companion. The results in this section are taken from \cite{Li09}. 

\begin{Prop}\label{PropUniLi} Let $\mathbb{X}\overset{\alpha}{\curvearrowleft} \mathbb{G}$. Then the Podle\'{s} $^*$-algebra $\mathscr{O}_{\mathbb{G}}(\mathbb{X})$ admits a universal C$^*$-algebra $C_0(\mathbb{X}_u)$.
\end{Prop} 

\begin{proof} Choose an irreducible representation $\pi$ and a morphism $T\in \Mor(\pi,\alpha)$. Then for $\{e_i\}$ an orthonormal basis of $\Hsp_{\pi}$, we have \[\delta_{\pi}(e_i) = \sum_{j} e_j \otimes U_{\pi}(e_j,e_i),\] and so, with $x_T = \sum_iT(e_i)T(e_i)^*$, we have \begin{eqnarray*} \alpha(x_T) &=& \sum_{i,j,k} T(e_j)T(e_k)^*\otimes U_{\pi}(e_j,e_i)U_{\pi}(e_k,e_i)^* \\&=& \sum_j T(e_j)T(e_j)^* \otimes 1_{\G} \\ &=& x_T\otimes 1_{\G}. \end{eqnarray*} It follows that
\[x_{T} \in C_0(\mathbb{X}/\mathbb{G}),\] and hence \[\|\lambda(T(e_i))\|^2\leq \|x_T\|, \quad \textrm{for all } ^*\textrm{-representations }\lambda\textrm{ of }\mathscr{O}_{\G}(\X).\]
Since $\mathscr{O}_{\mathbb{G}}(\mathbb{X})$ is the linear span of $\{T\xi\mid \pi,T\in \Mor(\pi,\alpha), \xi \in \Hsp_{\pi}\}$, we obtain \[\forall a\in\mathscr{O}_{\mathbb{G}}(\X),\quad \|a\|_u = \sup\{\|\lambda(a)\|\mid \lambda \;^*\textrm{-representation of }\mathscr{O}_{\G}(\X)\} < \infty.\]
\end{proof} 

\begin{Theorem} Let $\mathbb{X}\overset{\alpha}{\curvearrowleft}\mathbb{G}$. Then $\alpha_{\alg}$ extends to an action \[\alpha_u: C_0(\mathbb{X}_u) \rightarrow C_0(\mathbb{X}_u)\otimes C(\mathbb{G}_u)\] with $\alpha_u$ injective. Moreover, we have
\begin{itemize}
\item $C_0(\Y_u) = C_0(\Y)$, where $\Y_u = \X_u/\G_u$ and $\Y= \X/\G$,
\item $\mathscr{O}_{\G_u}(\X_u) = \mathscr{O}_{\G}(\X)$.
\end{itemize}
\end{Theorem}

\begin{proof}We have already proven that the action $\X_u\overset{\alpha_u}{\curvearrowleft}\G_u$ is a well-defined action of $\G_u$ in Lemma \ref{LemAlgtoAn}.

Note that the counit on $\mathscr{O}(\G)$ extends to $C_0(\G_u)$. As $\alpha_{\alg}$ is a Hopf algebra coaction, we obtain \[(\id_{\X}\otimes \varepsilon)\alpha_u(a) = a,\quad \forall a\in C_0(\X_u),\]  and in particular $\alpha_u$ injective. 

Let us now show that $C_0(\X_u)$ has the same spectral subspaces as $C_0(\X)$. Write 
\[\pi_u: C_0(\X_u)\rightarrow C_0(\X),\quad \pi_u: C(\G_u)\rightarrow C(\G).\]  Then by construction, one has \[(\pi_u\otimes \pi_u)\circ \alpha_u = \alpha\circ \pi_u.\]
From this it follows immediately that for all irreducible representations $\pi$ of $\G$, one has \[\pi_u(C_0(\X_u)_{\pi})= C_0(\X)_{\pi}.\] 

What remains to show is that $\pi_u$ is injective on each $C_0(\X_u)_{\pi}$. Now if $a_n \in \mathscr{O}_{\G}(\X)$ and \[a_n\underset{n\rightarrow \infty}{\rightarrow} b\in C_0(\Y_u),\] we get that \[b_n := E_{\Y}(a_n) = (\id_{\X}\otimes \varphi_{\G})\alpha(a_n)\underset{n\rightarrow \infty}{\rightarrow} (\id_{\X_u}\otimes \varphi_{\G_u})\alpha_u(b) =b.\]  As $C_0(\Y)$ is a C$^*$-algebra and $b_n \in C_0(\Y)$, we deduce $b\in C_0(\Y)$. This already shows $C_0(\Y_u) = C_0(\Y)$. 

Assume now that $a\in C_0(\X_u)_{\pi}$ with $\pi_u(a) =0$. Then \[0=\alpha(\pi_u(a^*a)) = (\pi_u\otimes \pi_u)\alpha_u(a^*a) \in \mathscr{O}_{\G}(\X)\underset{\mathrm{alg}}{\otimes} \mathscr{O}(\G).\] 

Applying $(\id_{\X}\otimes \varphi_{\G})$, we see that \[\pi_u \big{(}E_{\Y_u}(a^*a)\big{)}= 0\quad \Rightarrow \quad E_{\Y_u}(a^*a)=0.\]
But as $E_{\Y_u}$ is faithful on $\mathscr{O}_{\G_u}(\X_u)$, we obtain $a=0$.
\end{proof}

We now turn to the construction of the \emph{reduced} C$^*$-algebra associated to an action.

\begin{Def}\label{DefRightYMod} For $C_0(\Y) \subseteq C_0(\X)$ a non-degenerate C$^*$-subalgebra and \[E_{\mathbb{Y}}: C_0(\X)\rightarrow C_0(\Y)\] a faithful conditional expectation, we endow $C_0(\X)$ with the pre-Hilbert $C_0(\mathbb{Y})$-module structure \[\langle a,b\rangle_{\mathbb{Y}} = E_{\mathbb{Y}}(a^*b).\] We denote $L^2_{\Y}(\X)$ for the Hilbert $C_0(\Y)$-module obtained as the completion of $C_0(\X)$ with respect to the norm \[\|a\|_{\Y} = \sqrt{\|\langle a,a\rangle_{\Y}\|}.\]  
\end{Def} 

\begin{Lem}\label{LemIntroImplUn} Let $\mathbb{X}\overset{\alpha}{\curvearrowleft} \mathbb{G}$ with $\mathbb{Y} = \mathbb{X}/\mathbb{G}$. 
Then \begin{equation}\label{EqDefU}  \mathscr{O}_{\mathbb{G}}(\mathbb{X}) \underset{\alg}{\otimes} \mathscr{O}(\mathbb{G})\rightarrow  \mathscr{O}_{\G}(\mathbb{X})   \underset{\alg}{\otimes} \mathscr{O}(\mathbb{G}), \quad a\otimes g\mapsto \alpha(a)(1_{\X}\otimes g)\end{equation} completes to a unitary map \[U_{\alpha}: L^2_{\mathbb{Y}}(\mathbb{X}) \otimes L^2(\mathbb{G}) \rightarrow  L^2_{\mathbb{Y}}(\mathbb{X})\otimes L^2(\mathbb{G}).\]
\end{Lem}

\begin{proof} Note that \[(E_{\mathbb{Y}}\otimes \id_{\G})\alpha(a) = E_{\mathbb{Y}}(a)\otimes 1_{\G},\quad a\in C_0(\mathbb{X}).\] Then the map in the lemma is isometric by the following computation: for $a,b\in \mathscr{O}_{\G}(\X)$ and $g,h\in \mathscr{O}(\G)$, we have \begin{eqnarray*}  \langle \alpha(a)(1_{\X}\otimes g),\alpha(b)(1_{\X}\otimes h)\rangle_{\Y}  &=& (E_{\mathbb{Y}}\otimes \varphi_{\G})((1_{\X}\otimes g^*)\alpha(a^*b)(1_{\X}\otimes h)) \\&=& \varphi_{\G}(g^*h)E_{\mathbb{Y}}(a^*b)   \\ &=& \langle a\otimes g,b\otimes h\rangle_{\Y}.
\end{eqnarray*} 
The surjectivity follows from the surjectivity of the map in \eqref{EqDefU}, which is a consequence of $\alpha_{\alg}$ being a coaction by a Hopf algebra: \[b \otimes h = \varepsilon(b_{(1)})b_{(0)} \otimes h = b_{(0)} \otimes b_{(1)}(S(b_{(2)})h),\quad b\in \mathscr{O}_{\G}(\X),h\in \mathscr{O}(\G).\]
\end{proof} 

\begin{Lem} The  non-degenerate $^*$-homomorphism \[\pi_{\red}:C_0(\mathbb{X})\rightarrow \mathcal{L}(L^2_{\mathbb{Y}}(\mathbb{X}))\] obtained as the closure of the left multiplication map for $C_0(\X)$ satisfies \[(\pi_{\red}\otimes \pi_{\red})\alpha(a) = U_{\alpha}(\pi_{\red}(a)\otimes 1_{L^2(\G)})U_{\alpha}^*,\quad a\in C_0(\X)\] Moreover, $\pi_{\red}$ is injective on $\mathscr{O}_{\G}(\X)$. 
\end{Lem}

\begin{proof}  As $E_{\Y}$ is positive, it follows that \[E_{\Y}(y^*x^*xy)\leq \|x\|^2 E_{\Y}(y^*y),\quad \forall x,y\in C_0(\X).\] This implies that $\pi_{\red}$ is well-defined on $C_0(\X)$.

To see that $U_{\alpha}$ implements $\alpha$, we check that, for $a\in \mathscr{O}_{\G}(\X)$,  \[U_{\alpha}(\pi_{\red}(a)\otimes 1_{\G}) = (\pi_{\red}\otimes \pi_{\red})(\alpha_{\alg}(a))U_{\alpha}\quad \textrm{on }\mathscr{O}_{\G}(\X)\aotimes \mathscr{O}(\G).\]

Finally, as $E_{\Y}$ is faithful on $\mathscr{O}_{\G}(\mathbb{X})$, it follows straightforwardly that $\pi_{\red}$ is injective on $\mathscr{O}_{\G}(\X)$.
\end{proof}

\begin{Theorem} Let $\X\overset{\alpha}{\curvearrowleft}\G$ and write \[C_0(\mathbb{X}_{\red}) = \pi_{\red}(C_0(\mathbb{X})).\] Then \[\alpha_{\red}: C_0(\mathbb{X}_{\red})\rightarrow  C_0(\mathbb{X}_{\red})\otimes  C(\mathbb{G}_{\red})\subseteq \mathcal{L}(L^2_{\mathbb{Y}}(\mathbb{X})\otimes L^2(\mathbb{G})),\]\[ a \mapsto U_{\alpha}(\pi_{\red}(a)\otimes 1)U_{\alpha}^*\] defines an injective right coaction $\mathbb{X}_{\red}\overset{\alpha_{\red}}{\curvearrowleft}\mathbb{G}_{\red}$.

Moreover,
\begin{itemize}
\item $C_0(\Y_{\red}) = C_0(\Y)$,
\item  $\mathscr{O}_{\G_{\red}}(\mathbb{X}_{\red}) = \mathscr{O}_{\mathbb{G}}(\mathbb{X})$.
\end{itemize}
\end{Theorem}

\begin{proof} As $U_{\alpha}$ implements the coaction $\alpha_{\alg}$ on $\mathscr{O}_{\G}(\X)$, it follows immediately that $\alpha_{\red}$ is a well-defined coaction satisfying the Podle\'{s} condition. Also its injectivity is immediate from the defining formula.

Write \[\pi_{\red}: C_0(\X)\rightarrow C_0(\X_{\red})\] for the canonical quotient map. Then we have by construction that \[(\pi_{\red}\otimes \pi_{\red})\circ\alpha = \alpha_{\red}\circ \pi_{\red}.\]

As $\pi_{\red}$ is faithful on $\mathscr{O}_{\mathbb{G}}(\X)$, we have $ \mathscr{O}_{\mathbb{G}}(\mathbb{X})\subseteq \mathscr{O}_{\G_{\red}}(\mathbb{X}_{\red})$. On the other hand, if $\pi$ is an irreducible representation and $T\in \Mor(\pi,\alpha_{\red})$, $\xi\in \Hsp_{\pi}$, choose $a\in C_0(\X)$ with $\pi_{\red}(a) = T\xi$. Then with $\chi_{\pi}$ the element defined by \eqref{EqElChi}, and \[b = (\id_{\X}\otimes \varphi_{\G})(\alpha(a)(1_{\X}\otimes \chi_{\pi}^*)),\] we have $b\in C_0(\X)_{\pi}$ and \[\pi_{\red}(b) = (\id_{\X_{\red}}\otimes \varphi_{\G_{\red}})(\alpha_{\red}(\pi_{\red}(a))(1_{\X_{\red}}\otimes \chi_{\pi}^*)) = T\xi.\] This shows that we have $C_0(\X)_{\pi} = C_0(\X_{\red})_{\pi}$, and in particular $\mathscr{O}_{\mathbb{G}}(\mathbb{X})= \mathscr{O}_{\G_{\red}}(\mathbb{X}_{\red})$ and $C_0(\Y_{\red}) = C_0(\Y)$.
\end{proof}

To end, let us consider the case where there is only one completion of $\mathscr{O}_{\G}(\X)$, which must hence necessarily coincide with the original $C_0(\X)$. Recall that $\G$ is called \emph{coamenable} if the canonical surjection map $C(\G_u) \rightarrow C(\G_{\red})$ is an isomorphism. This is equivalent  to $C(\G_u)$ having a faithful Haar state or $C(\G_{\red})$ having a bounded counit, see \cite{BMT01}.

\begin{Prop} Let $\X\curvearrowleft \G$. If $\G$ is coamenable, then $C_0(\X) = C_0(\X_{\red}) = C_0(\X_u)$. 
\end{Prop}  
\begin{proof} In fact, the faithfulness of the Haar state on $C(\G_u)$ implies that the conditional expectation $C_0(\X_u) \rightarrow C_0(\X/\G)$ is faithful, which immediately entails the proposition.
\end{proof}

If the compact quantum group $\G$ is coamenable, we also have the following preservation of nuclearity.

\begin{Theorem}[\cite{DLRZ02}] Let $\X\curvearrowleft \G$. If $\G$ is coamenable, then $C_0(\X)$ is a nuclear C$^*$-algebra if and only if $C_0(\X/\G)$ is a nuclear C$^*$-algebra.
\end{Theorem} 

For the proof we refer to \cite{DLRZ02}. Note that by replacing $\X$ with its one-point compactification, we may assume that $\X$ is compact.

\section{Actions and coactions, crossed and smash products}

In this section, we will consider the notion \emph{dual} to that of action, called \emph{$\G$-coaction} or \emph{$\widehat{\G}$-action}. We will also consider the algebras associated to actions and coactions, called respectively the \emph{crossed} and \emph{smash} products. The results in this section are well-known in various contexts, such as algebraic (see e.g.\ \cite[Section 1.6]{Maj95}), C$^*$-algebraic (see e.g.\ \cite[Section 9]{Tim08}) and von Neumann algebraic (see e.g.\ \cite[Section 2]{Vae01}), but we will give a slightly idiosynchratic treatment which blends the algebraic and operator algebraic approaches. 

\begin{Def} Let $\G$ be a compact quantum group. A left $\mathscr{O}(\G)$-module $^*$-algebra $(\mathscr{O}(\X),\rhd)$ consists of a $^*$-algebra $\mathscr{O}(\X)$, endowed with a unital $\mathscr{O}(\G)$-module structure $\rhd$ such that \[h\rhd (ab) = (h_{(1)}\rhd a)(h_{(2)}\rhd b)\quad \textrm{and}\quad (h\rhd a)^* = S(h)^*\rhd a^*,\]for all $a,b\in \mathscr{O}(\X)$ and $h\in \mathscr{O}(\G)$. 
\end{Def}

We will use the following terminology. 

\begin{Term}We say that a left $\mathscr{O}(\G)$-module $^*$-algebra $\mathscr{O}(\X)$ corresponds to a right $\widehat{\G}$-action $\X\curvearrowleft \widehat{\G}$. We also refer to the latter as a \emph{left $\G$-coaction}.
\end{Term}

We will explain some of the terminology in the next examples. 

\begin{Exa} Let $\Gamma$ be a discrete group. Write $C(\widehat{\Gamma}_u) = C^*_u(\Gamma)$, so that $\mathscr{O}(\widehat{\Gamma}_u) = \mathscr{O}(\widehat{\Gamma}) = \C\lbrack \Gamma\rbrack$ is the group algebra of $\Gamma$. Then a left $\mathscr{O}(\widehat{\Gamma})$-module $^*$-algebra structure on some $^*$-algebra $\mathscr{O}(\X)$ corresponds precisely to a left action of $\Gamma$ on the $^*$-algebra $\mathscr{O}(\X)$. If then moreover $X$ is a locally compact space, a left action of $\Gamma$ on $C_0(X)$ is nothing but a \emph{right} action $X\curvearrowleft \Gamma$.  
\end{Exa}

\begin{Exa} Let $\G$ be a \emph{finite} quantum group, that is, $\G$ is a compact quantum group with $\mathscr{O}(\G)$ finite-dimensional. Then the dual $\mathscr{O}(\widehat{\G}) = \mathscr{O}(\G)^*$ is a again a compact quantum group with respect to the Hopf $^*$-algebra structure \[(\omega\eta)(a) = (\omega\otimes \eta)\Delta(a),\quad \omega^*(a) = \overline{\omega(S(a)^*)},\quad \widehat{\Delta}(\omega)(a\otimes b) = \omega(ab).\] It is easy to check that, for $\mathscr{O}(\X)$ a $^*$-algebra, there is a one-to-one correspondence between $\widehat{\G}$-actions in the sense of Section \ref{SecAc} and $\widehat{\G}$-actions in the above sense, by \[\alpha: \mathscr{O}(\X)\rightarrow \mathscr{O}(\X)\otimes \mathscr{O}(\widehat{\G})\] \[\Updownarrow\]\[h \rhd a = (\id\otimes h)\alpha(a),\quad a\in \mathscr{O}(\X),h\in \mathscr{O}(\G),\] where we identified $\mathscr{O}(\G) = \mathscr{O}(\G)^{**}$.  
\end{Exa}

\begin{Exa} Let $\G$ be a compact quantum group. Then we have the canonical right $\widehat{\G}$-action $\G\curvearrowleft \widehat{\G}$, given by the $\mathscr{O}(\G)$-module $^*$-algebra structure \[h\rhd f = h_{(1)}fS(h_{(2)}),\quad h,f\in \mathscr{O}(\G).\]We refer to this as the (right) \emph{conjugate action} of $\widehat{\G}$. 
\end{Exa}

There is no need to restrict $\mathscr{O}(\X)$ to be a purely algebraic object, as will follow from the proof of the following lemma. 

\begin{Lem} Let $\G$ be a compact quantum group and $C_0(\X)$ a C$^*$-algebra. If $C_0(\X)$ is a left $\mathscr{O}(\G)$-module $^*$-algebra, we have \[l_h: C_0(\X)\rightarrow C_0(\X),\; a\mapsto h\rhd a\] bounded, for all $h\in \mathscr{O}(\G)$. 
\end{Lem} 

\begin{proof} Let $\pi$ be a unitary representation of $\G$, and let $\{e_i\}$ be an orthonormal basis of $\Hsp_{\pi}$. Write $e_{ij}$ for the standard matrix units in $B(\Hsp_{\pi})$ with respect to this basis. Then from the corepresentation property of $U_{\pi}$ and the definining properties of a module $^*$-algebra, we obtain that the map \[\gamma_{\pi}: C_0(\X) \rightarrow B(\Hsp_{\pi})\otimes C_0(\X),\quad a\mapsto U_{\pi}\rhd a  =\sum_{i,j} e_{ij}\otimes \left(U_{\pi}(e_i,e_j)\rhd a\right)\] is a unital $^*$-homomorphism. Hence in particular \[\|U_{\pi}\rhd a\|\leq \|a\|,\quad \textrm{for all }a\in C_0(\X).\] But this implies that $\|U_{\pi}(e_i,e_j)\rhd a\|\leq \|a\|$ for all $i,j$ and all $a\in C_0(\X)$. As the $U_{\pi}(e_i,e_j)$ span $\mathscr{O}(\G)$ as a vector space, the lemma is proven.
\end{proof} 

The same argument allows to prove the next corollary.

\begin{Cor}\label{CorExtModUn} Assume that the $^*$-algebra $\mathscr{O}(\mathbb{X})$ admits a universal C$^*$-envelope $C_0(\mathbb{X}_u)$. Then any $\mathscr{O}(\G)$-module $^*$-algebra structure on $\mathscr{O}(\X)$ extends to an $\mathscr{O}(\G)$-module $^*$-algebra structure on $C_0(\mathbb{X}_u)$.
\end{Cor} 

We now define the notion of \emph{smash product} with respect to a right $\widehat{\G}$-action.

\begin{Def} Assume that we have a left $\mathscr{O}(\G)$-module $^*$-algebra $\mathscr{O}(\X)$. The \emph{algebraic smash product} $^*$-algebra \[\mathscr{O}(\X\rtimes \widehat{\G}) = \widehat{\G}\ltimes \mathscr{O}(\X) = \mathscr{O}(\X)\# \mathscr{O}(\G)\] is defined as the tensor product vector space $\mathscr{O}(\X)\aotimes \mathscr{O}(\G)$, equipped with the product \[(a\otimes h)(b\otimes g) = a(h_{(1)}\rhd b)\otimes h_{(2)}g\] and the $^*$-structure \[(a\otimes h)^* = (h_{(1)}^*\rhd a^*)\otimes h_{(2)}^*.\]
\end{Def} 

\begin{Exe} Show that $\mathscr{O}(\X)\#\mathscr{O}(\G)$ is an associative $^*$-algebra.
\end{Exe}

\begin{Not} For $\mathscr{O}(\X)$ a left $\mathscr{O}(\G)$-module $^*$-algebra, $a\in \mathscr{O}(\X)$ and $h\in \mathscr{O}(\G)$, we write \[ah = a\otimes h = `(a\otimes 1_{\G})(1_{\X}\otimes h)\textrm{'}\quad \textrm{in }\mathscr{O}(\X)\#\mathscr{O}(\G).\]
We further write \[ha = (h_{(1)}\rhd a)\otimes h_{(2)} = `(1_{\X}\otimes h)(a\otimes 1_{\G})\textrm{'}\quad \textrm{in }\mathscr{O}(\X)\#\mathscr{O}(\G).\]\end{Not}
\begin{Lem}Let $\X\curvearrowleft\widehat{\G}$. For all $a\in \mathscr{O}(\X)$ and $h\in \mathscr{O}(\G)$, one has
\begin{enumerate}
\item $(ah)^* = h^*a^*$,
\item $ah = h_{(2)}(S^{-1}(h_{(1)})\rhd a)$. 
\end{enumerate}
\end{Lem} 
\begin{proof} Exercise. 
\end{proof}

\begin{Lem} Assume that $\X\curvearrowleft \widehat{\G}$ with $\X$ a locally compact quantum space. Then $\mathscr{O}(\X\rtimes \widehat{\G})$ admits a universal C$^*$-envelope.
\end{Lem} 

We will denote this universal C$^*$-algebra as $C_0(\X\rtimes_u \widehat{\G})$.

\begin{proof} If $\pi$ is a non-degenerate $^*$-representation of $\mathscr{O}(\X\rtimes \widehat{\G})$ on a Hilbert space $\Hsp$, we have a $^*$-representation $\pi$ of $\mathscr{O}(\G)$ on the pre-Hilbert space $\pi(\mathscr{O}(\X\rtimes \widehat{\G}))\Hsp$ by \[\pi(h) \pi(bg)\xi = \pi(hbg)\xi,\]  and then \[\pi(ah)\xi= \pi(a)\pi(h)\xi,\quad \forall a\in C_0(\X),h\in \mathscr{O}(\G),\xi \in \pi(\mathscr{O}(\X\rtimes \widehat{\G}))\Hsp.\] As any $^*$-representation of $\mathscr{O}(\G)$ on a pre-Hilbert space is automatically bounded, we obtain \[\|\pi(ah)\|\leq \|a\|\|h\|_u, \quad a\in C_0(\X),h\in \mathscr{O}(\G).\]
\end{proof}

However, it is not clear if $\mathscr{O}(\X\rtimes \widehat{\G})$ is actually a \emph{good} $^*$-algebra, that is, if \[ \mathscr{O}(\X\rtimes \widehat{\G})\subseteq C_0(\X\rtimes_u \widehat{\G}).\] This is the question we deal with next, see Corollary \ref{CorGoodAlg}. 

\begin{Prop}\label{PropDualAct} Assume $\mathscr{O}(\X)$ is a left $\mathscr{O}(\G)$-module $^*$-algebra. Then $\X\rtimes \widehat{\G}\curvearrowleft \G$ by the Hopf $^*$-algebra coaction
\[\alpha: \mathscr{O}(\X\rtimes \widehat{\G}) \rightarrow \mathscr{O}(\X\rtimes \widehat{\G})\otimes \mathscr{O}(\G),\quad ah\mapsto ah_{(1)}\otimes h_{(2)}.\] 
\end{Prop}
\begin{proof}
Exercise.
\end{proof}

We refer to $\alpha$ as the \emph{dual action} of $\G$ on $\X\rtimes \widehat{\G}$. The following corollary follows from Lemma \ref{LemAlgtoAn}. 

\begin{Cor} If $\X\curvearrowleft \widehat{\G}$ with $\X$ a locally compact quantum space, then $\X\rtimes_u \widehat{\G}\curvearrowleft \G$.
\end{Cor}

\begin{Prop}\label{PropCrossImp}Let $\X$ be a locally compact quantum space, and $\X\curvearrowleft \widehat{\G}$. Put \[E_{\X}: \mathscr{O}(\X\rtimes \widehat{\G}) \rightarrow C_0(\X),\quad x\mapsto (\id_{\X}\otimes \varphi_{\G})\alpha(x),\] with $X\rtimes \widehat{\G}\overset{\alpha}{\curvearrowleft} \G$. Then \[\langle x,y\rangle_{\X} = E_{\X}(x^*y)\] defines a pre-Hilbert $C_0(\X)$-module structure on $\mathscr{O}(\X\rtimes\widehat{\G})$. Moreover, its completion $L^2_{\X}(\X\rtimes \widehat{\G})$ satisfies \[L^2_{\X}(\X\rtimes \widehat{\G})\cong L^2(\G)\otimes C_0(\X),\] isometrically as right Hilbert $C_0(\X)$-modules, by means of the map \[V: \mathscr{O}(\X\rtimes\widehat{\G}) \rightarrow L^2(\G)\otimes C_0(\X),\quad ha \mapsto h\otimes a.\]
\end{Prop} 

\begin{proof}
Note that the range of $E_{\X}$ indeed lies in $C_0(\X)$, since \begin{eqnarray*} (\id_{\X}\otimes \varphi_{\G})(\alpha(ha)) &=& (\id_{\X}\otimes \varphi_{\G})(\Delta(h)(a\otimes 1_{\G})) \\ &=& \varphi_{\G}(h)a.\end{eqnarray*}

Now we compute \begin{eqnarray*} \langle ha,gb\rangle_{\X} &=& (\id_{\X}\otimes \varphi_{\G})((a^*\otimes 1_{\G})\Delta(h^*g)(b\otimes 1_{\G}))\\ &=& \varphi_{\G}(h^*g)a^*b \\ &=& \langle h\otimes a,g\otimes b \rangle \\ &=& \langle V(ha),V(gb)\rangle.\end{eqnarray*} It follows that indeed $\mathscr{O}(\X\rtimes \widehat{\G})$ is a pre-Hilbert $C_0(\X)$-module, and that $V$ induces a unitary map  \[L^2_{\X}(\X\rtimes \widehat{\G})\cong L^2(\G)\otimes C_0(\X)\] which is clearly right $C_0(\X)$-linear.
\end{proof} 

\begin{Theorem} Let $\G$ be a compact quantum group, $\X$ a locally compact quantum space, and let $\X\curvearrowleft\widehat{\G}$.  Then the left multiplication map of $\mathscr{O}(\X\rtimes \widehat{\G})$ on itself extends to an injective $^*$-homomorphism \[\pi_{\red}:\mathscr{O}(\X\rtimes \widehat{\G})\rightarrow \mathcal{L}(L^2_{\X}(\X\rtimes \widehat{\G})).\]
\end{Theorem}

\begin{proof}We will use the notation from Proposition \ref{PropCrossImp}.

 For $a\in C_0(\X)$ and $x\in \mathscr{O}(\X\rtimes \widehat{\G})$, we have that \[E_{\X}(x^*a^*ax)\leq \|a^*a\|E_{\X}(x^*x),\] by positivity of $E_{\X}$.
 Hence, writing $L_ax = ax$, we have \[\|L_ax\|_{\X} \leq \|a\|\|x\|_{\X}.\] It is furthermore easy to see that \[\langle y,ax\rangle_{\X} = \langle a^*y,x\rangle_{\X},\] hence $L_a$ is adjointable with  $L_a^* = L_{a^*}$. Thus $L_a \in \mathcal{L}(L^2_{\X}(\X\rtimes \widehat{\G}))$.
 
On the other hand, for $h\in \mathscr{O}(\G)$ we see that \[VL_h = (L_h\otimes 1)V\quad \textrm{on }\mathscr{O}(\X\rtimes \widehat{\G}),\] where on the right $L_h$ denotes the operation of left multiplication with $h$ on $L^2(\G)$.  This proves the existence of $L_h\in  \mathcal{L}(L^2_{\X}(\X\rtimes \widehat{\G}))$. It is then easily seen that we obtain a (non-degenerate) $^*$-homomorphism \[\pi_{\red}: \mathscr{O}(\X\rtimes \widehat{\G})\rightarrow \mathcal{L}(L^2_{\X}(\X\rtimes \widehat{\G})),\quad ah \mapsto L_aL_h.\]

To see that it is injective, assume that $\pi_{\red}(x) = 0$.
Then $\langle xx^*,xx^*\rangle_{\X}= 0$, and hence $xx^*=0$. Applying $E_{\X}$, we obtain $\langle x^*,x^*\rangle_{\X}= 0$, hence $x=0$.
\end{proof}

\begin{Cor}\label{CorGoodAlg} Let $\X$ be a locally compact quantum space, and let $\X\curvearrowleft\widehat{\G}$.  Then \[\mathscr{O}(\X\rtimes \widehat{\G})\subseteq C_0(\X\rtimes_u\widehat{\G}).\] 
\end{Cor}

The preceding discussion also allows us to define a reduced smash product, in the following way.

\begin{Def} Let $\X$ be a locally compact quantum space, and let $\X\curvearrowleft\widehat{\G}$.  We define \[C_0(\X \rtimes_{\red}\widehat{\G}) = \Big{\lbrack} \pi_{\red}\big{(}\mathscr{O}(\X\rtimes \widehat{\G})\big{)}\Big{\rbrack}\subseteq \mathcal{L}(L^2_{\X}(\X\rtimes \widehat{\G})).\]
\end{Def}

\begin{Exe} Show that, with \[V: L^2_{\X}(\X\rtimes\widehat{\G}) \rightarrow L^2(\G)\otimes C_0(\X),\quad ha \mapsto h\otimes a,\] we have for $h,g\in \mathscr{O}(\G)$ and $a,b\in C_0(\X)$ that 
\[V\pi_{\red}(ha)V^*(g\otimes b) = hg_{(2)}\otimes \Big{(}S^{-1}(g_{(1)})\rhd a\Big{)}b.\]
\end{Exe}

\begin{Lem}\label{LemFixSmash} Assume $\X$ is a locally compact quantum space and $\X\curvearrowleft \widehat{\G}$. Then 
\begin{itemize}
\item  $(\X\rtimes_u \widehat{\G})/\G = \X$, 
\item $(\X\rtimes_{\red} \widehat{\G})/\G = \X$,
\item $\mathscr{O}_{\G}(\X\rtimes_u \widehat{\G}) = \mathscr{O}(\X\rtimes \widehat{\G})$,
\item $\mathscr{O}_{\G}(\X\rtimes_{\red} \widehat{\G}) = \mathscr{O}(\X\rtimes \widehat{\G})$.
\end{itemize}
\end{Lem}
\begin{proof}
Exercise.
\end{proof}

Our next goal is to define the crossed product with respect to an action of a compact quantum group. We first recall the definition and basic structure theory of the dual $^*$-algebra.

\begin{Def} Let $\G$ be a compact quantum group. Then we define \[\mathscr{O}(\widehat{\G}) = \{\varphi_{\G}(\,\cdot\, h)\mid h\in \mathscr{O}(\G)\}\subseteq \mathscr{O}(\G)^*,\] with $\mathscr{O}(\G)^*$ the vector space dual to $\mathscr{O}(\G)$. 
\end{Def}

Note that we have as well $\mathscr{O}(\widehat{\G}) = \{\varphi_{\G}(h\,\cdot\,)\mid h\in \mathscr{O}(\G)\}$, by making use of the modular automorphism $\sigma$.

\begin{Lem}The vector space $\mathscr{O}(\widehat{\G})$ is a $^*$-algebra for \[(\omega\cdot \theta)(h) = (\omega\otimes \theta)\Delta(h),\quad \omega^*(h) = \overline{\omega(S(h)^*)}.\] In fact, with $\{\pi\}$ a maximal collection of mutually inequivalent unitary representations of $\G$, we have 
 \[ \mathscr{O}(\widehat{\G})\cong \oplus_{\pi} B(\Hsp_{\pi}),\quad \omega \mapsto \sum_{\pi} (\id_{\Hsp_{\pi}}\otimes \omega)\delta_{\pi}.\]
\end{Lem} 

Note that $\mathscr{O}(\widehat{\G})$ is non-unital, unless $\G$ is finite! 

\begin{Lem} Let $\X$ be a locally compact quantum space, and $\X\overset{\alpha}{\curvearrowleft}\G$. Let $\Y = \X/\G$. There is a unique non-degenerate $^*$-representation \[l:\mathscr{O}(\widehat{\G}) \rightarrow \mathcal{L}(L^2_{\Y}(\X))\]such that \[ l_{\omega}(a) =  (\id_{\X}\otimes \omega)\alpha(a),\quad a\in \mathscr{O}_{\G}(\X),\omega \in \mathscr{O}(\widehat{\G}).\]
\end{Lem} 
Recall that $L^2_{\Y}(\X)$ is the completion of $\mathscr{O}_{\G}(\X)$ with respect to \[\|a\|_{\Y}= \|E_{\Y}(a^*a)\|^{1/2}.\]
\begin{proof}

We first check that $l_{\omega}$ is well-defined and bounded: for $a\in C_0(\X)$ and $\omega \in \mathscr{O}(\widehat{\G})$, we compute \begin{eqnarray*} \|l_{\omega}(a)\|_{\Y}^2 &=&\| E_{\Y}((\id_{\X}\otimes \omega)(\alpha(a))^*(\id_{\X}\otimes \omega)(\alpha(a)))\|\\ &\leq & \|\omega\|\| E_{\Y}\left( (\id_{\X}\otimes |\omega|)\alpha(a^*a)\right) \|\\ &=& \|\omega\|^2 \|E_{\Y}(a^*a)\| \\ &=& \|\omega\|^2\|a\|_{\Y}^2.
\end{eqnarray*}
It is then clear that $l$ is a representation. To see that it is $^*$-preserving, we first observe that, with $g,h\in \mathscr{O}(\G)$, we have in the case $\X = \G$ that 
\[ \langle g,l_{\omega}h\rangle = \langle l_{\omega^*}g,h\rangle\] in $L^2(\G)$, by strong left invariance. In general, it then follows that $l_{\omega}^* = l_{\omega^*}$ by using that \[\alpha(l_{\omega}a) = (\id_{\X}\otimes l_{\omega})\alpha(a),\quad a\in \mathscr{O}_{\G}(\X).\]  

Finally, to see that the representation is non-degenerate, we use that for any $a\in \mathscr{O}_{\G}(\X)$, there exists an $\omega \in \mathscr{O}(\widehat{\G})$ such that $l_{\omega}a = a$, using for example the elements defined by \eqref{EqElChi}. 
\end{proof} 

In the following, we write $L_a = \pi_{\red}(a)$ for the operator of left multiplication with $a\in C_0(\X)$ on $L^2_{\Y}(\X)$. 

\begin{Lem} Assume $\X\curvearrowleft \G$. For $a\in \mathscr{O}_{\G}(\X)$ and $\omega\in \mathscr{O}(\widehat{\G})$, \[l_{\omega}L_a = L_{a_{(0)}}l_{\omega(a_{(1)}\,\cdot\,)} \quad \textrm{in } \mathcal{L}(L^2_{\Y}(\X)).\]
\end{Lem}
\begin{proof} Exercise.
\end{proof}

\begin{Def} Let $\X\overset{\alpha}{\curvearrowleft} \G$ be an action of $\G$ on the locally compact quantum space $\X$. The \emph{crossed product $^*$-algebra} \[\mathscr{O}(\mathbb{X}\rtimes \G)\] is the vector space  $\mathscr{O}_{\mathbb{G}}(\X)\aotimes \mathscr{O}(\widehat{\G})$ with the following $^*$-algebra structure:   \[(a\otimes\omega)(b\otimes\theta) = ab_{(0)}\otimes \omega(b_{(1)}\,\cdot\,)\theta,\quad (a\otimes\omega)^*= a^*_{(0)}\otimes\omega^*(a_{(1)}^*\,\cdot\,).\] 
\end{Def} 

Again, we leave it to the reader to check that, for example, the multiplication is associative. 

As in the case of smash products, we will use the shorthand notation \[a\omega = a\otimes \omega, \quad \omega a = a_{(0)}\otimes \omega(a_{(1)}\,\cdot\,)\] for $\omega \in \mathscr{O}(\widehat{\G}),a\in \mathscr{O}_{\G}(\X)$. Then we have for example that \[(a\omega)^* = \omega^* a^*.\] 

\begin{Lem} The universal C$^*$-envelope $C_0(\X\rtimes_u \G)$ of $\mathscr{O}(\X\rtimes \G)$ exists.
\end{Lem}
\begin{proof} If $\pi$ is a non-degenerate $^*$-representation of $\mathscr{O}(\X\rtimes \G)$, it follows from a similar argument as in the proof of Proposition \ref{PropUniLi} that for any $a\in \mathscr{O}_{\G}(\X)$, there exists $C_a>0$ such that \[\|\pi(ax)\xi\|\leq C_a\|\pi(x)\xi\|,\quad \forall x\in \mathscr{O}(\X\rtimes \G),\xi\in \Hsp_{\pi}.\] Hence there exists a non-degenerate $^*$-representation \[\pi:\mathscr{O}_{\G}(\X)\rightarrow \Hsp_{\pi}\] such that \[\pi(a)\pi(x) = \pi(ax),\quad a\in \mathscr{O}_{\G}(\X),x\in \mathscr{O}(\X\rtimes \G).\] Similarly, as any $\omega \in \mathscr{O}(\widehat{\G})$ lies in a finite-dimensional C$^*$-algebra, there exists a non-degenerate \[\pi: \mathscr{O}(\widehat{\G})\rightarrow B(\Hsp_{\pi})\] such that \[\pi(\omega)\pi(x) = \pi(\omega x),\quad \forall \omega\in \mathscr{O}(\widehat{\G}),x\in \mathscr{O}(\X\rtimes \G).\] It is then clear that in fact \[\pi(a \omega) = \pi(a)\pi(\omega),\quad a\in \mathscr{O}_{\G}(\X),\omega \in \mathscr{O}(\widehat{\G}),\] and so \[\|\pi(a \omega)\|\leq \|a\|_u\|\omega\|_{\mathscr{O}(\widehat{\G})}.\]
\end{proof}

\begin{Def} For $\X\curvearrowleft \G$, we define the \emph{full (or universal) crossed product C$^*$-algebra} as the universal C$^*$-envelope $C_0(\X\rtimes_u \G)$ of $\mathscr{O}(\X\rtimes \G)$.
\end{Def}

A similar construction is again possible on the reduced level. 

\begin{Def} The \emph{reduced crossed product C$^*$-algebra} $C_0(\X\rtimes_{\red} \G)$ is  \[ \big{\lbrack}  L_{\alpha(a)}(1\otimes l_{\omega})\mid a\in C_0(\X),\omega\in \mathscr{O}(\widehat{\G})\big{\rbrack}\subseteq \mathcal{L}(L^2_{\Y}(\X)\otimes L^2(\G)).\]
\end{Def}

\begin{Lem} We have that $C_0(\X\rtimes_{\red} \G)$ is a C$^*$-algebra, with \[\pi: \mathscr{O}(\X\rtimes \G)\rightarrow C_0(\X\rtimes_{\red} \G),\quad  a\omega \mapsto L_{\alpha(a)}  (1\otimes l_{\omega})\] an injective $^*$-homomorphism.  
\end{Lem} 
\begin{proof}
It is easily verified that $\pi$ is a $^*$-homomorphism. To see that it is injective, assume that $\sum_i L_{\alpha(a_i)}(1\otimes l_{\omega_i}) = 0$. Then \[\forall h\in \mathscr{O}(\G),\quad \sum_i a_{i(0)}\otimes a_{i(1)}l_{\omega_i} h = 0,\] hence \[\forall h\in \mathscr{O}(\G),\quad \sum_i a_{i(0)}\otimes a_{i(1)}\otimes a_{i(2)}l_{\omega_i} h = 0,\] and \[ \sum_i a_{i(0)}\otimes S(a_{i(1)})a_{i(2)}l_{\omega_i} h = \sum_i a_i \otimes l_{\omega_i} h = 0.\]
\end{proof}

In fact, we will show that one always has $C_0(\X\rtimes_u\G)\cong C_0(\X\rtimes_{\red}\G)$. We will need some preparations. 

We will in the following drop the notation $l$, and let $\mathscr{O}(\widehat{\G})$ act directly on $L^2(\G)$ or $L_{\Y}^2(\X)$. We also let $\mathscr{O}(\X\rtimes \G)$ act directly on $L^2_{\Y}(\X)\otimes L^2(\G)$. 

We will denote, for $\pi$ an irreducible representation of $\G$, \[p_{\pi} = \varphi_{\G}(\,\cdot\,\chi_{\pi}^*)\in \mathscr{O}(\widehat{\G}),\] which is a minimal central projection in $\mathscr{O}(\widehat{\G})$ by the defining property \eqref{EqElChi} of $\chi_{\pi}$. Note that $p_{\pi}L^2(\G)$ is then a finite-dimensional vector space. 

\begin{Lem}\label{LemBiDuCom} There exists on \[\mathscr{O}_{\G}(\X)\otimes \left( p_{\pi}B(L^2(\G))p_{\pi'}\right) \subseteq \mathcal{L}(L^2_{\Y}(\X)\otimes L^2(\G))\] a unique $\mathscr{O}(\G)$-comodule structure $\alpha_{\pi,\pi'}$ such that \[x\otimes (y\xi_{\G})(z\xi_{\G})^* \mapsto x_{(0)} \otimes (y_{(2)}\xi_{\G})(z_{(2)}\xi_{\G})^*\otimes S^{-1}(y_{(1)})x_{(1)}S^{-1}(z_{(1)})^*\] for $x\in \mathscr{O}_{\G}(\X), y\xi_{\G}\in p_{\pi}L^2(\G)$ and $z\in p_{\pi'}L^2(\G)$.  Moreover, the fixed point subspace of this comodule is precisely $p_{\pi}\mathscr{O}(\X\rtimes \G)p_{\pi'}$. 
\end{Lem} 

Here we mean by `fixed point subspace' the space of elements satisfying $\alpha_{\pi,\pi'}(x) = x\otimes 1_{\G}$. 

\begin{proof} If $e_i^{(\pi)}$ denotes an orthonormal basis of $\Hsp_{\pi}$, then it is clear that the $U_{\pi}(e_i^{(\pi)},e_j^{(\pi)})\xi_{\G}$ form a basis of $p_{\pi}L^2(\G)$, and hence $\alpha_{\pi,\pi'}$ is a well-defined map, which is immediately seen to be a comodule map. Assume now that \[\sum_i x_i\otimes (y_i\xi_{\G})(z_i\xi_{\G})^* \in  \mathscr{O}_{\G}(\X)\otimes \left( p_{\pi}B(L^2(\G))p_{\pi'}\right),\] with \[\alpha_{\pi,\pi'} \left(\sum_i x_i\otimes (y_i\xi_{\G})(z_i\xi_{\G})^*\right) = \sum_i x_i\otimes (y_i\xi_{\G})(z_i\xi_{\G})^* \otimes 1_{\G}.\] Then also \begin{multline*} \sum_i x_{i(0)}\otimes (y_{i(3)}\xi_{\G})(z_{i(3)}\xi_{\G})^* \otimes y_{i(2)}\otimes z_{i(2)}^* \otimes S^{-1}(y_{i(1)}) x_{i(1)}S^{-1}(z_{i(1)})^* \\ = \sum_i x_i\otimes (y_{i(2)}\xi_{\G})(z_{i(2)}\xi_{\G})^* \otimes y_{i(1)}\otimes z_{i(1)}^*\otimes 1_{\G}.\end{multline*} Applying the antipode to the third and fourth factor, and multiplying them respectively to the left and right of the last factor, we obtain \[\sum_i x_{i(0)}\otimes (y_i\xi_{\G})(z_i\xi_{\G})^* \otimes x_{i(1)} = \sum_i x_i\otimes (y_{i(2)}\xi_{\G})(z_{i(2)}\xi_{\G})^* \otimes y_{i(1)}z_{i(1)}^*.\]
Hence \[\sum_i x_{i(0)} \otimes (S(x_{i(1)}) y_i\xi_{\G})(z_i\xi_{\G})^*  = \sum_i x_i \otimes (S(z_{i(1)}^*)\xi_{\G})(z_{i(2)}\xi_{\G})^*.\] But writing $\omega_i = \varphi_{\G}(z_i^*\cdot \,)$,  an easy computation using strong left invariance shows \[ (S(z_{i(1)}^*)\xi_{\G})(z_{i(2)}\xi_{\G})^* = l_{\omega_i}.\]

Hence \[\sum_i x_{i(0)} \otimes (S(x_{i(1)}) y_i\xi_{\G})(z_i\xi_{\G})^* = \sum_i x_i \otimes l_{\omega_i},\] and \begin{multline*} \sum_i x_i \otimes (y_i\xi_{\G})(z_i\xi_{\G})^* \\=  \sum_i x_{i(0)} \otimes x_{i(1)}S(x_{i(2)}) (y_i\xi_{\G})(z_i\xi_{\G})^* \\= \sum_i x_{i(0)} \otimes x_{i(1)}l_{\omega_i}.\end{multline*} This shows that the fixed point subspace of $\mathscr{O}_{\G}(\X)\otimes \left( p_{\pi}B(L^2(\G))p_{\pi'}\right)$ lies in $p_{\pi}\mathscr{O}(\X\rtimes \G)p_{\pi'}$. We leave the reverse inclusion as an exercise to the reader. 
\end{proof}

\begin{Theorem}\label{TheoUnicrossC} The space $\mathscr{O}(\X\rtimes \G)$ has a unique C$^*$-completion. 
\end{Theorem} 

\begin{proof} As $\mathscr{O}(\X\rtimes \G) = \underset{p}{\bigcup}\,p\mathscr{O}(\X\rtimes \G)p$ with $p$ ranging over all central projections in $\mathscr{O}(\widehat{\G})$, it is sufficient to show that $p\mathscr{O}(\X\rtimes \G)p$ is a C$^*$-algebra for any such $p$. In turn, it is hence sufficient to show that $p_{\pi}\mathscr{O}(\X\rtimes \G)p_{\pi'}$ is a closed subspace of $C_0(\X\rtimes_{\red} \G)$ for all irreducible $\pi,\pi'$. 

However, if $a\in C(\G)_{\pi''}$, then \[p_{\pi}ap_{\pi''} = a_{(0)}\varphi_{\G}(a_{(1)}\,\cdot\,\chi_{\pi}^*)\varphi_{\G}(\,\cdot \,\chi_{\pi'}^*),\] which is zero unless $\pi$ is a subrepresentation of $\pi''\otimes \pi'$. By Frobenius reciprocity, the latter only happens when $\pi''\subseteq \pi\otimes \overline{\pi'}$. As there are only finitely many such $\pi''$ up to equivalence, we deduce that $p_{\pi}\mathscr{O}(\X\rtimes \G)p_{\pi'}$ lies in a finite direct sum of $\mathscr{O}_{\G}(\X)_{\pi''}\otimes \left( p_{\pi}B(L^2(\G))p_{\pi'}\right)$. However, by the proof of Lemma \ref{LemClosSpec} the latter space is closed in $C_0(\X)\otimes \left( p_{\pi}B(L^2(\G))p_{\pi'}\right)$. As $p_{\pi}\mathscr{O}(\X\rtimes \G)p_{\pi'}$ arises as the fixed point subspace of $\mathscr{O}_{\G}(\X)\otimes \left( p_{\pi}B(L^2(\G))p_{\pi'}\right)$ for a continuous comodule structure by Lemma \ref{LemBiDuCom}, it follows that $p_{\pi}\mathscr{O}(\X\rtimes \G)p_{\pi'}$ is closed. 
\end{proof} 

In the following, we will hence write simply $\X\rtimes \G$ for the crossed product. The following proposition is dual to Proposition \ref{PropDualAct}. 

\begin{Prop} Let $\X\overset{\alpha}{\curvearrowleft}\G$. There is a unique action of $\widehat{\G}$ on $\X\rtimes \G$ such that \[h \rhd (a\omega) = a\omega(\,\cdot\,h),\quad h\in \mathscr{O}(\G),\omega\in \mathscr{O}(\widehat{\G}),a\in \mathscr{O}_{\X}(\G).\]
\end{Prop}

\begin{proof} It is clear that $\rhd$ defines a unital $\mathscr{O}(\G)$-module on $\mathscr{O}(\X\rtimes \G)$. We leave it to the reader to check that it is then a module $^*$-algebra. It then follows that it extends to a module $^*$-algebra structure on $C_0(\X\rtimes \G)$ by Corollary \ref{CorExtModUn}.
\end{proof}

It follows that one can iterate by taking crossed products with alternatingly $\G$ and $\widehat{\G}$. However, this process essentially stabilizes after the second step. In the following, we will write $B_{00}(L^2(\G))$ for the $^*$-algebra of finite rank operators on $L^2(\G)$ with respect to the subspace $\mathscr{O}(\G)\xi_{\G}$. This means that elements of $B_{00}(L^2(\G))$ are linear combinations of rank one operators of the form $(y\xi_{\G})(z\xi_{\G})^*$ for $y,z\in \mathscr{O}(\G)$. 

The following theorem goes under the name of the `Takesaki-Takai-Baaj-Skandalis duality'.

\begin{Theorem}\label{TheoTakTak} Let $\X$ be a locally compact quantum space, and $\X\curvearrowleft \G$. Then \[\mathscr{O}(\X\rtimes \G\rtimes \widehat{\G}) \cong \mathscr{O}_{\G}(\X)\aotimes B_{00}(L^2(\G))\] equivariantly, where $ \mathscr{O}_{\G}(\X)\aotimes B_{00}(L^2(\G))  \curvearrowleft  \G$ by \[x \otimes (y\xi_{\G})(z\xi_{\G})^* \mapsto x_{(0)} \otimes (y_{(2)}\xi_{\G})(z_{(2)}\xi_{\G})^*\otimes S^{-1}(y_{(1)})x_{(1)}S^{-1}(z_{(1)})^*.\]

Similarly, if $\X\curvearrowleft \widehat{\G}$, then \[\mathscr{O}(\X\rtimes \widehat{\G}\rtimes \G) \cong  B_{00}(L^2(\G))\aotimes\mathscr{O}_{\G}(\X)\] equivariantly, where $B_{00}(L^2(\G))\aotimes \mathscr{O}_{\G}(\X) \curvearrowleft  \widehat{\G} $ by \[h \rhd (((y\xi_{\G})(z\xi_{\G})^*)\otimes x) = (yS^{-1}(h_{(1)})\xi_{\G})(z\sigma(h_{(3)}^*)\xi_{\G})^* \otimes S^{-2}(h_{(2)})\rhd x.\]
\end{Theorem} 

\begin{proof} We will only sketch a proof, leaving missing details to the reader.

Let $\X\curvearrowleft \G$. Recall the Woronowicz character $f$ from Theorem \ref{TheoWorChar}. Write \[U: \mathscr{O}(\G)\xi_{\G}\rightarrow \mathscr{O}(\G)\xi_{\G},\quad h\xi_{\G} \mapsto (f*S(h))\xi_{\G}.\] Using the properties of $f$, it is easy to see that $U$ is a unitary involution. Moreover, if $g,h\in \mathscr{O}(\G)$, we have \[UhU^* g\xi_{\G} = g(f*S(h))\xi_{\G}.\] It follows that $U\mathscr{O}(\G)U^*$ commutes elementwise with $\mathscr{O}(\G)$, and moreover an easy computation reveals that \[UhU^* l_{\omega} = l_{\omega(\,\cdot\, h_{(1)})}Uh_{(2)}U^*\] on $L^2(\G)$. It follows that we obtain a $^*$-homomorphism \[\mathscr{O}(\X\rtimes \G\rtimes \widehat{\G}) \rightarrow \mathscr{O}_{\G}(\X)\aotimes B(L^2(\G)),\quad a\omega h \mapsto \alpha(a)(1\otimes l_{\omega}UhU^*).\] A further easy computation shows that \[l_{\varphi_{\G}(g^*\,\cdot)} = (S(g_{(1)}^*)\xi_{\G})(g_{(2)}\xi_{\G})^*,\quad h,g\in \mathscr{O}(\G).\] It follows that the above $^*$-homomorphism lands in $\mathscr{O}_{\G}(\X) \aotimes B_{00}(L^2(\G))$. 

To see that it is injective, assume that $\sum a_i\omega_i h_i$ is sent to zero. We may assume here that the $a_i$ are linearly independent. However, then \[\sum_i a_i\otimes l_{\omega_i}Uh_iU^* = \sum_i a_{i(0)(0)} \otimes S(a_{i(0)(1)})a_{i(1)}l_{\omega_i}Uh_iU^* = 0,\] so all $l_{\omega_i}Uh_iU^*=0$. It is easy to check that the latter implies $\omega_i\otimes h_i=0$. 

Using that the $l_{\omega}UhU^*$ span $B_{00}(L^2(\G))$, a similar argument allows to conclude that the $^*$-homomorphism is surjective. 

Finally, to see that the $^*$-homomorphism is equivariant, note that we can write the coaction as \[x\otimes y \mapsto x_{(0)} \otimes \left(\widehat{W}^*(y\otimes x_{(1)})\widehat{W}\right),\] with $\widehat{W}$ the unitary defined by \[\widehat{W}(x\xi_{\G}\otimes \xi) = x_{(2)}\xi_{\G}\otimes x_{(1)}\xi.\] By what was already done in Lemma \ref{LemBiDuCom}, we are only left with checking that this implements the proper formula for the coaction on the $\mathscr{O}(\G)$-part of $\mathscr{O}(\X\rtimes \G\rtimes \widehat{\G})$. We leave this computation to the reader. 

Similarly, let $\X\curvearrowleft \widehat{\G}$. Then on $L^2_{\X}(\X\rtimes \widehat{\G})$, we have an action of $\mathscr{O}(\X\rtimes \widehat{\G}\rtimes \G)$ where the $\mathscr{O}(\X\rtimes \G)$-part acts by left multiplication and where \[l_{\omega}(ha) = \omega(h_{(2)})h_{(1)}a,\quad \omega \in \mathscr{O}(\widehat{\G}).\] Translating this to a representation on $\mathscr{O}(\G)\xi_{\G}\otimes \mathscr{O}_{\G}(\X)$ by means of the unitary $V$ from Proposition \ref{PropCrossImp}, we find after simplification the $^*$-homomorphism \[\mathscr{O}(\X\rtimes \widehat{\G}\rtimes \G)\rightarrow B_{00}(L^2(\G))\aotimes \mathscr{O}_{\G}(\X),\] \[hb\varphi_{\G}(\,\cdot g) \mapsto (hS^{-1}(g_{(1)})\xi_{\G})(\sigma(g_{(3)}^*)\xi_{\G})^* \otimes S^{-2}(g_{(2)})\rhd b.\] We leave it to the reader to check that this is an equivariant isomorphism. 
\end{proof}

By the universal property and the nuclearity of the C$^*$-algebra of compact operators $B_0(L^2(\G))$ on $L^2(\G)$, it follows that one then also has an equivariant $^*$-isomorphism \[C_0(\X\rtimes \G\rtimes_u \widehat{\G}) \cong C_0(\X_u) \otimes B_0(L^2(\G)).\] From this, one can then also conclude that 
\[C_0(\X\rtimes \G\rtimes_{\red}\widehat{\G}) \cong C_0(\X_{\red}) \otimes B_0(L^2(\G)).\] In fact, the above is only a very concrete instance of a theorem for \emph{regular} locally compact quantum groups, see \cite{BS93} and \cite[Section 9]{Tim08} where the result is proven in the context of multiplicative unitaries.

\section{Homogeneous actions}\label{SecHom}

In this section and the next, we will treat two specific kinds of actions, namely \emph{homogeneous} and \emph{free} actions. The results of the current section are taken from \cite{Pod95}, \cite{Boc95} and \cite{BDV06}.  We assume in this section that $\X$ is a \emph{compact} quantum space, as this condition is automatically required by the following definition.  

\begin{Def} An action $\mathbb{X}\overset{\alpha}{\curvearrowleft} \mathbb{G}$ is called \emph{homogeneous} (or \emph{ergodic}) if \[C(\X/\G) = \C1_{\X}.\]
\end{Def}

It is clear that if $X$ is a compact Hausdorff space and $G$ a compact Hausdorff group, then $X\overset{\alpha}{\curvearrowleft} G$ is homogeneous if and only if $\alpha$ is transitive (in the ordinary sense).

The notion of homogeneity is preserved under passing to universal or reduced actions. 

\begin{Lem} If $\mathbb{X}\overset{\alpha}{\curvearrowleft} \G$ is homogeneous, then also $\alpha_u$ and $\alpha_{\red}$ are homogeneous.
\end{Lem} 

\begin{proof} We have that \[\mathbb{Y} = \mathbb{X}/\mathbb{G} = \mathbb{X}_u/\mathbb{G}_u = \mathbb{X}_{\red}/\mathbb{G}_{\red}.\]
\end{proof} 

Since $\Y$ reduces to a point in the case of homogeneous actions, we obtain in particular that the conditional expectation $E_{\Y}$ becomes a \emph{state} on $C(\X)$. We introduce the adapted notation in the next definition. 

\begin{Def} Let $\X\overset{\alpha}{\curvearrowleft} \G$ be a homogeneous action. Then we define $\varphi_{\X}$ as the unique state on $C(\X)$ such that \[\forall a\in C(\X),\quad \varphi_{\X}(a)1_{\X} = (\id_{\X}\otimes \varphi_{\G})\alpha(a).\]
\end{Def} 

The following lemma follows immediately from the definition of $\varphi_{\X}$ and the invariance of $\varphi_{\G}$ on $C(\G)$. 

\begin{Lem}Let $\X\overset{\alpha}{\curvearrowleft} \G$ homogeneous.Then the state $\varphi_{\X}$ is invariant: \[\forall a\in C(\X),\quad (\varphi_{\X}\otimes \id_{\G})\alpha(a) = \varphi_{\X}(a)1_{\G}.\]
\end{Lem} 

We now turn to special classes of homogeneous actions. The terminology is taken from \cite{Pod95}. 

Recall first that a compact quantum subgroup $\mathbb{H}\subseteq \mathbb{G}$ corresponds to a surjective $^*$-homomorphism $\pi_{\mathbb{H}}: C(\G)\rightarrow C(\mathbb{H})$ such that \[(\pi_{\mathbb{H}}\otimes \pi_{\mathbb{H}})\circ \Delta_{\G} = \Delta_{\mathbb{H}}\circ \pi_{\mathbb{H}}.\] Then automatically $\pi_{\mathbb{H}}(\mathscr{O}(\G)) = \mathscr{O}(\mathbb{H})$, and $\pi_{\mathbb{H}}\circ S_{\G} = S_{\mathbb{H}}\circ \pi_{\mathbb{H}}$ for the respective antipodes. 

\begin{Def} Let $\mathbb{X}\overset{\alpha}{\curvearrowleft} \mathbb{G}$. One calls $\alpha$ of \emph{quotient type} if there exists a compact quantum subgroup $\mathbb{H}\subseteq \mathbb{G}$ with corresponding quotient map $\pi_{\mathbb{H}}: C(\G)\rightarrow C(\mathbb{H})$ and a $^*$-isomorphism 
\[\theta: C(\mathbb{X}) \rightarrow C(\mathbb{H}\backslash\mathbb{G}) = \{g\in C(\mathbb{G})\mid (\pi_{\mathbb{H}}\otimes \id_{\G}) \Delta(g) = 1_{\mathbb{H}}\otimes g\}\] such that \[(\theta\otimes \id_{\G})\circ \alpha = \Delta \circ \theta.\]
\end{Def}

Note that $\mathbb{H}\backslash \G$ has a natural right $\G$-action, implemented by the coaction of $C(\G)$, as we are in the situation of Example \ref{ExaHom} with $\X = \G$ and $\mathbb{H}$ acting on $\mathbb{G}$ by \[(\pi_{\mathbb{H}}\otimes \id_{\G})\circ \Delta: C(\G) \mapsto C(\mathbb{H})\otimes C(\G).\] 

\begin{Lem}\label{LemQuotHom} If $\X\overset{\alpha}{\curvearrowleft}\G$ is of quotient type, then $\alpha$ is homogeneous.
\end{Lem}

\begin{proof} If $\mathbb{H}$ is a compact quantum subgroup and $g\in C(\mathbb{H}\backslash \G)$ satisfies $\Delta(g) = g\otimes 1_{\G}$, then by applying $(\id_{\G}\otimes \varphi_{\G})$ we see that $g = \varphi_{\G}(g)1_{\G}$. 
\end{proof}

\begin{Lem}\label{LemEmbUn} If $\alpha$ is of quotient type, then $\alpha_u$ is of quotient type.
\end{Lem} 
\begin{proof} Let $\theta: C(\X) \rightarrow C(\mathbb{H}\backslash \G)\subseteq C(\G)$ be equivariant, for some compact quantum subgroup $\mathbb{H}\subseteq \G$. Note that for $a\in  \mathscr{O}_{\G}(\X)$, we have \[\theta(a) = (\varepsilon \otimes \id_{\G})\Delta(\theta(a))=  (\varepsilon\circ \theta\otimes \id_{\G})\alpha(a).\]  Hence we obtain that $\theta(\mathscr{O}_{\G}(\X)) \subseteq \mathscr{O}(\G)$, and a universal $^*$-homomorphism \[\theta_u: C(\X_u) \rightarrow C(\G_u).\] 

Let us show that this map is injective. Consider the map \[E_{\mathbb{H}\backslash \G}: \mathscr{O}(\G) \rightarrow \mathscr{O}_{\G}(\mathbb{H}\backslash \G),\quad g \mapsto (\varphi_{\mathbb{H}}\circ \pi_{\mathbb{H}}\otimes \id_{\G})\Delta(g),\] which indeed has the above range since the restriction to each $C(\G)_{\pi}$ is a right $\mathscr{O}(\G)$-comodule map into $C(\mathbb{H}\backslash \G)$.  Let $(\Hsp_u,\pi_u)$ be a faithful $^*$-representation of $C(\X_u)$. Then by the complete positivity of $E_{\mathbb{H}\backslash \G}$, we can make a new Hilbert space $\Ind_{\G}(\Hsp_u)$ by separation-completion of $\mathscr{O}(\G)\aotimes \Hsp_{u}$ with respect to \[\langle g\otimes \xi,h\otimes \eta \rangle = \langle \xi, (\pi_u \circ\theta^{-1})(E_{\mathbb{H}\backslash \G}(g^*h))\eta\rangle.\] It comes with a representation $\widetilde{\pi}_u$ of $C(\G_u)$ such that \[\widetilde{\pi}_u(g)(h\otimes \xi) = gh\otimes \xi,\quad g,h\in \mathscr{O}(\G),\xi\in \Hsp_u.\] Now an easy computation shows that \[\widetilde{\pi}_u\theta_u(a)(1_{\G}\otimes \xi) = \theta(a)\otimes \xi = 1_{\G}\otimes \pi_u(a)\xi,\quad a\in \mathscr{O}_{\G}(\X).\] It follows that also \[\widetilde{\pi}_u\theta_u(a)(1_{\G}\otimes \xi) = 1_{\G}\otimes \pi_u(a)\xi,\quad a\in C(\X_u).\] Since $\xi\mapsto 1_{\G}\otimes \xi$ is an isometric inclusion of $\Hsp_u$ into $\Ind(\Hsp_u)$, it follows that $\theta_u$ is injective. 

Let us now show that $\X_u \curvearrowleft \G_u$ is of quotient type. By universality, we have that $\mathbb{H}_u \subseteq \G_u$. With $\pi_{\mathbb{H},u}$ the corresponding quotient map, we then have \[1_{\mathbb{H}_u}\otimes \theta_u(a) = (\pi_{\mathbb{H},u}\otimes \id)\Delta(\theta_u(a))\] for $a\in \mathscr{O}_{\G}(\X)$, so by continuity we conclude that $C(\X_u) \subseteq C(\mathbb{H}_u \backslash \G_u)$. To obtain equality, we have to show that $\mathbb{H}_u \backslash \G_u = (\mathbb{H}\backslash \G)_u$. But choose $x \in C(\mathbb{H}_u \backslash \G_u)$, and choose a sequence $x_n\in \mathscr{O}(\G)$ such that $x_n \rightarrow x$. Then \[(\varphi_{\mathbb{H}}\pi \otimes \id_{\G})\Delta(x_n) \rightarrow (\varphi_{\mathbb{H}_u}\pi_{\mathbb{H},u} \otimes \id_{\G})\Delta(x) = x.\] As the left hand side elements are in $\mathscr{O}_{\G}(\mathbb{H}\backslash \G)$ for each $n$, we obtain $C((\mathbb{H}\backslash \G)_u) = C(\mathbb{H}_u\backslash \G_u)$. 
\end{proof}

On the other hand, it is easy to see that a homogeneous action is not always of quotient type. For example, let $\pi$ be an irreducible representation of $\G$ with dimension at least two. Then $\Ad_{\pi}$ is clearly a homogeneous action. However, it can not be of quotient type, as any $\mathscr{O}(\mathbb{H}\backslash \G)$ admits a character, the counit of $\mathscr{O}(\G)$. 

As another example, consider $\G$ non-coamenable, so that $C(\G_{u}) \neq C(\G_{\red})$. Then $C(\G_{\red})$ does not admit any characters, so $\G_{\red}\curvearrowleft \G_{\red}$ is not of quotient type! Of course, this latter example can be avoided by slightly relaxing the definition of quantum subgroup, but it turns out that, in general, the notion of quotient type is much too strong anyhow. A much larger class of homogeneous actions is obtained as follows.

\begin{Def} Let $\mathbb{X}\overset{\alpha}{\curvearrowleft} \mathbb{G}$. One calls $\alpha$ \emph{embeddable} if there exists a faithful $^*$-homomorphism \[\theta: C(\mathbb{X}) \hookrightarrow C(\mathbb{G})\] such that \[(\theta\otimes \id)\circ \alpha = \Delta \circ \theta.\]
\end{Def}

Clearly, any action of quotient type is embeddable. On the other hand, by considering adjoint actions one sees again that homogeneous does not imply embeddable. When seen as subalgebras of $C(\G)$, embeddable actions are also referred to as \emph{(right) coideal C$^*$-subalgebras}.

\begin{Lem} Let $\mathbb{X}\overset{\alpha}{\curvearrowleft} \G$ be embeddable. Then also $\alpha_{\red}$ is embeddable.
\end{Lem} 
\begin{proof} 
Let $L^2(\X)$ be the GNS-space of $C(\X)$ with respect to the invariant state $\varphi_{\X}$. By embeddability, we have that the natural inclusion $C(\X)\rightarrow C(\G)$ gives rise to an embedding $L^2(\X) \subseteq L^2(\G)$, sending GNS-vector to GNS-vector. 

If now $g\in \mathscr{O}(\G)$, then there exists a unique bounded operator $\rho(g)$ on $L^2(\G)$ such that \[\rho(g)h\xi_{\G} = hg\xi_{\G},\quad h\in \mathscr{O}(\G).\] Namely, with $J_{\G}$ the anti-unitary \[J_{\G}h\xi_{\G} = \sigma_{i/2}(h)^*\xi_{\G},\] we have \[\rho(g) = J_{\G}\sigma_{-i/2}(g)^*J_{\G}.\] Hence if $a\in C(\X)$ is an element which vanishes in $C(\X_{\red})$, then for $g\in \mathscr{O}(\G)$, we have \[\pi_{\red}(\theta(a))g\xi_{\G} = \rho(g) \pi_{\red}(\theta(a))\xi_{\G} = 0.\] It follows that $\theta$ descends to an equivariant map $C(\X_{\red}) \rightarrow C(\G_{\red})$, and hence $\alpha_{\red}$ is embeddable. 
\end{proof} 

Embeddable actions can be detected by the existence of at least one classical point (in the universal setting). 

\begin{Lem} Let $\mathbb{X}\overset{\alpha}{\curvearrowleft} \G$ be homogeneous. The following are equivalent: 
\begin{enumerate}
\item $\alpha_{\red}$ is embeddable.
\item $C(\mathbb{X}_u)$ has a character.
\end{enumerate}
\end{Lem} 
\begin{proof}
Assume first that $\alpha_{\red}$ is embeddable. Then we have $\theta(\mathscr{O}_{\G}(\X)) \subseteq \mathscr{O}(\G)$, leading to a $^*$-homomorphism $\theta_u: C(\X_u)  \rightarrow C(\G_u)$. Hence $\varepsilon \circ \theta_u$ is a character on $C(\X_u)$.
 
Conversely, assume that $C(\mathbb{X}_u)$ has a character $\chi$. We obtain an equivariant $^*$-homomorphism \[\theta_u: C(\mathbb{X}_u)\rightarrow C(\mathbb{G}_u),\quad a\mapsto (\chi\otimes \id)\circ \alpha_u.\] This leads to an equivariant $^*$-homomorphism \[\theta_{\alg}: \mathscr{O}_{\G}(\mathbb{X})\rightarrow \mathscr{O}(\mathbb{G}),\]
 and \[\varphi_{\G}(\theta_{\alg}(a)^*\theta_{\alg}(a)) = \chi(E_{\Y}(a^*a)).\] But, by homogeneity, $E_{\Y}$ takes values in $\C1_{\X}$, so \[E_{\Y}(a^*a)= \chi(E_{\Y}(a^*a))1_{\X}.\]Hence $\theta_{\red}: C(\X_{\red}) \hookrightarrow C(\G_{\red})$, since $E_{\Y}$ is faithful on $C(\X_{\red})$.
\end{proof}

As an example, let us consider Woronowicz's \emph{quantum $SU(2)$-groups} $SU_q(2)$ \cite{Wor87}. These are determined by a parameter $q\in \R\setminus\{0\}$, and the associated Hopf $^*$-algebra is the universal unital $^*$-algebra generated by two elements $\alpha,\gamma$ such that \[U = \begin{pmatrix} \alpha & -q\gamma^*\\ \gamma & \alpha^*\end{pmatrix}\] is unitary. The comultiplication is uniquely determined by requiring that $U$ is a unitary corepresentation. We note that $SU_q(2)$ is coamenable, hence there is only one C$^*$-completion of $SU_q(2)$. 

It can be shown, just as in the case $q=1$ which gives the classical group $SU(2)$, that the irreducible representations of $SU(2)$ can be labelled by the halfintegers $\frac{1}{2}\N$, in such a way that $U_0$ is the trivial representation, $U_{1/2}= U$ and \[U_{n}\otimes U_{1/2} = U_{n-\frac{1}{2}}\oplus U_{n+\frac{1}{2}}, \quad n\in \frac{1}{2}\N\setminus \{0\}.\] For example, the `spin $1$-representation' $U_1$ can be shown to be the 3-dimensional representation given by the unitary corepresentation \[U_1 = \begin{pmatrix} \alpha^2 & q(1+q^2)^{1/2} \gamma^*\alpha &  q^2(\gamma^*)^2 \\ -(1+q^2)^{1/2} \gamma \alpha & 1- (1+q^2)\gamma^*\gamma & q(1+q^2)^{1/2} \alpha^*\gamma^* \\ \gamma^2 & -(1+q^2)^{1/2} \alpha^*\gamma & (\alpha^*)^2\end{pmatrix}.\]

Choose now $x\in \R$ and write \[\begin{pmatrix} \omega_{-1} & \omega_0 & \omega_{1}\end{pmatrix} = \begin{pmatrix} \mathrm{sgn}(q)|q|^{-1/2} & \frac{|q|^x -|q|^{-x}}{(|q|+|q|^{-1})^{1/2}}& -|q|^{1/2}\end{pmatrix}U_1.\] Then it is easily seen that the unital algebra generated by $\omega_{-1},\omega_0$ and $\omega_1$ is in fact a $^*$-algebra, and that this is a right coideal $^*$-subalgebra $\mathscr{O}(S^2_{q,x})$ of $\mathscr{O}(SU_q(2))$. We call $S^2_{q,x}$ the \emph{Podle\'{s} sphere} at parameter $x$, see \cite{Pod87} (with a different parametrisation) and \cite{MNW91}.  

If one writes \[X = \omega_1, \quad Y = X^*,\quad  Z = (|q|+|q|^{-1})^{-1/2}\omega_0 - \frac{|q|^x-|q|^{-x}}{|q|+|q|^{-1}},\] then a tedious computation shows that these variables satisfy the relations $X^* = Y$, $Z^* = Z$, $XZ=q^2ZX$ and \[X^*X = (1-|q|^{x-1}Z)(1+|q|^{-x-1}Z),\quad XX^* = (1-|q|^{x+1}Z)(1+|q|^{-x+1}Z).\] One can show that these are in fact \emph{universal} relations for the $^*$-algebra $\mathscr{O}(S^2_{q,x})$. 

We claim that these quantum homogeneous spaces are not of quotient type. Indeed, as morphisms between Hopf algebras preserve the antipode, one sees that coideal subalgebras of quotient type are invariant under the antipode squared. However, for $\mathscr{O}(SU_q(2))$ one has \[S^2(\alpha) = \alpha,\quad S^2(\gamma) = q^2\gamma,\] and from the commutation relations one sees that this can never be implemented by an automorphism of $\mathscr{O}(S^2_{q,x})$.

However, one can consider the case `$x\rightarrow \infty$', leading to the coideal subalgebra $\mathscr{O}(S^2_{q,\infty})$ generated by  \[\begin{pmatrix} \omega_{-1} & \omega_0 & \omega_{1}\end{pmatrix} = \begin{pmatrix} 0 & 1& 0\end{pmatrix}U_1.\] It is not difficult to verify that then \[\mathscr{O}(S^2_{q,\infty}) = \mathscr{O}(S^1\backslash SU_q(2)),\] where the circle group $S^1$, whose associated Hopf $^*$-algebra is the Lorentz algebra $\C\lbrack z,z^{-1}\rbrack$ with $^*$-structure $z^* = z^{-1}$, is a quotient group of $SU_q(2)$ by \[\pi_{S^1}\begin{pmatrix} \alpha & -q\gamma^* \\ \gamma & \alpha^*\end{pmatrix} = \begin{pmatrix} z & 0 \\ 0 & z^{-1}\end{pmatrix}.\] If we write \[ X = -\frac{q}{(1+q^2)^{1/2}}\omega_{-1},\quad Y=X^*,\quad Z = -(q+q^{-1})^{-1}(\omega_0-1),\] then we find the universal relations $X=Y^*$, $Z^*=Z$, $XZ = q^2ZX$ and \[X^*X = -q^{-2}Z^2+q^{-1}Z,\quad XX^* = -q^2Z^2+qZ.\]

Let us now return again to general homogeneous actions $\X\curvearrowleft \G$. Consider an irreducible representation $\pi$ of $\G$. Our aim will be to prove Boca's result \cite[Theorem 17]{Boc95} that $C(\X)_{\pi}$ is finite-dimensional, with dimension bounded by $\dim_q(\pi)$. (In fact, Boca had the upper bound $\dim_q(\pi)^2$, which was improved to the optimal bound $\dim_q(\pi)$ in \cite{Tom08} and \cite{BDV06}).

For $T\in \Mor(\pi,\alpha)$ and $\xi\in \Hsp_{\pi}$, we will write \[T\xi = U_{\pi}(T,\xi).\] Then for $S,T \in \Mor(\pi,\alpha)$, it is easy to check that \[\sum_{i} U_{\pi}(S,e_i) U_{\pi}(T,e_i)^* \in C(\X/\G) = \C1_{\X},\] where $\{e_i\}$ is an orthonormal basis of $\Hsp_{\pi}$.  It follows that $\Mor(\pi,\alpha)$ is in a natural way a pre-Hilbert space with \[\langle T,S\rangle = \sum_i U_{\pi}(S,e_i)U_{\pi}(T,e_i)^*.\] 

We split the proof of Boca's result into a qualitative and a quantitative part, and base our proof on the approaches of \cite{BDV06} and \cite{Tom08}, see also \cite{DCY13} for the qualitative part.  

\begin{Theorem}\label{TheoFinBoc} If $\mathbb{X}\underset{\alpha}{\curvearrowleft} \G$ is homogeneous, then all isotypical components $C(\mathbb{X})_{\pi}$ are finite-dimensional.
\end{Theorem} 

\begin{proof} We first show that the Hilbert space norm and the operator norm on $C(\X)_{\pi}$ are equivalent.  Of course, \[\langle a\xi_{\X},a\xi_{\X}\rangle \leq \|a\|^2,\quad a\in C(\X)_{\pi}.\] On the other hand, let $\chi_{\pi}$ be the element defined by \eqref{EqElChi}, and write $\rho_{\pi} = \sigma(\chi_{\pi})$. Then we have for $a\in C(\X)_{\pi}$ that \[a = (\id_{\X}\otimes \varphi_{\G})(\alpha(a)(1_{\X}\otimes \chi_{\pi}^*)) = (\id_{\X}\otimes \varphi_{\G})((1_{\X}\otimes \rho_{\pi}^*)\alpha(a)),\] hence \begin{eqnarray*} a^*a  &=& ((\id_{\X}\otimes \varphi_{\G})((1_{\X}\otimes \rho_{\pi}^*)\alpha(a)))^*(\id_{\X}\otimes \varphi_{\G})((1_{\X}\otimes \rho_{\pi}^*)\alpha(a))  \\ &\leq &  (\id_{\X}\otimes \varphi_{\G})(\alpha(a)^* (1_{\X}\otimes \rho_{\pi}\rho_{\pi}^*)\alpha(a))  \\ &\leq & \|\rho_{\pi}\|^2 (\id_{\X}\otimes \varphi_{\G})\alpha(a^*a) \\ &=& \|\rho_{\pi}\|^2  \langle a\xi_{\X},a\xi_{\X}\rangle.\end{eqnarray*} 

In particular, the pre-Hilbert spaces $C(\X)_{\pi}\xi_{\X}$ come equipped with bounded anti-linear involutions \[\mathcal{S}: C(\X)_{\pi}\xi_{\X}\rightarrow C(\X)_{\bar{\pi}}\xi_{\X},\quad a\xi_{\X}\mapsto a^*\xi_{\X}.\] We will now show that the unit operator on $C(\X)_{\pi}\xi_{\X}$ is compact, hence $C(\X)_{\pi}$ finite-dimensional. 

Consider the isometry \[\mathcal{G}_{\alpha}: L^2(\X)\otimes L^2(\X) \rightarrow L^2(\X)\otimes L^2(\G),\quad a\xi_{\X}\otimes \eta \mapsto \alpha(a)(\eta\otimes \xi_{\G}).\]  It is easily seen that this operator splits into isometries \[\mathcal{G}_{\pi}: C(\X)_{\pi}\xi_{\X}\otimes L^2(\X) \rightarrow L^2(\X)\otimes C(\G)_{\pi}\xi_{\G}.\]  Write \[\sum_i \mathcal{S}^*\xi_i\otimes \eta_i = \mathcal{G}_{\pi}^*(\xi_{\X}\otimes \rho_{\pi}\xi_{\G}),\] where $\xi_i \in C(\X)_{\bar{\pi}}$ and $\eta_i \in L^2(\X)$.  Then we compute, for $w\in C(\X)_{\pi}$ and $\eta\in L^2(\X)$, \begin{eqnarray*} \langle \sum_{i} \mathcal{S}^*\xi_i \otimes \eta_i,w\xi_{\X}\otimes \eta\rangle &=& \langle \xi_{\X}\otimes \rho_{\pi}\xi_{\G},\alpha(w)(\eta\otimes \xi_{\G})\rangle \\ &=& \langle \xi_{\X},w_{(0)}\eta\rangle \varphi_{\G}(\rho_{\pi}^*w_{(1)}) \\ &=&  \langle \xi_{\X},w_{(0)}\eta\rangle \varphi_{\G}(w_{(1)}\chi_{\pi}^*) \\ &=& \langle \xi_{\X},w\eta\rangle\\ &=& \langle w^*\xi_{\X},\eta\rangle.\end{eqnarray*} On the other hand, we also have \begin{eqnarray*}  \langle  \sum_{i} \mathcal{S}^*\xi_i \otimes \eta_i,w\xi_{\X}\otimes \eta\rangle &=& \sum_i \langle \mathcal{S}^*\xi_i,w\xi_{\X}\rangle \langle\eta_i,\eta\rangle \\ &=& \left\langle \sum_i \langle \xi_i,w^*\xi_{\X}\rangle \eta_i,\eta\right\rangle.\end{eqnarray*} It follows that \[\zeta = \sum_i \langle \xi_i,\zeta\rangle \eta_i,\quad \forall \zeta\in C(\X)_{\bar{\pi}}\xi_{\X},\] and hence the unit operator on $C(\X)_{\bar{\pi}}\xi_{\X}$ is Hilbert-Schmidt.
\end{proof} 

The above theorem implies in particular that $\mathscr{O}_{\G}(\X)$ admits a modular automorphism. Indeed, if $\pi$ and $\pi'$ are not equivalent, then \[\varphi_{\X}(ab^*) = \varphi_{\X}(b^*a) = 0,\quad a\in C(\X)_{\pi},b\in C(\X)_{\pi'}.\] It follows that we can find on each (finite-dimensional!) $C(\X)_{\pi}$ a linear map  \[\sigma_{\X}: C(\X)_{\pi}\rightarrow C(\X)_{\pi}\] such that \[\varphi_{\X}(ab^*) = \varphi_{\X}(b^*\sigma_{\X}(a)),\quad a,b\in C(\X)_{\pi}.\] The map $\sigma_{\X}$ then extends to a linear endomorphism of $\mathscr{O}_{\G}(\X)$, and the faithfulness of $\varphi_{\X}$ allows one to deduce that $\sigma_{\X}$ is an automorphism satisfying \[\sigma_{\X}(a^*) = \sigma_{\X}^{-1}(a)^*.\] One then has \[\varphi_{\X}(ab) = \varphi_{\X}(b\sigma_{\X}(a)),\quad \forall a,b\in \mathscr{O}_{\G}(\X).\] 

\begin{Lem}\label{LemModAut} For $\X\overset{\alpha}{\curvearrowleft} \G$ homogeneous, one has \[\alpha(\sigma_{\X}(a)) = (\sigma_{\X}\otimes S^{-2})\alpha(a),\quad a\in \mathscr{O}_{\G}(\X).\] 
\end{Lem}
\begin{proof} Take $a,b\in \mathscr{O}_{\G}(\X)$ and $g\in \mathscr{O}(\G)$. Then, on the one hand, \begin{eqnarray*} \varphi_{\X}(b \sigma_{\X}(a)_{(0)}) \varphi_{\G}(g\sigma_{\X}(a)_{(1)}) &=& \varphi_{\X}(b_{(0)}\sigma_{\X}(a)_{(0)}) \varphi_{\G}(gS^{-1}(b_{(2)})b_{(1)}\sigma_{\X}(a)_{(1)}) \\ &=& \varphi_{\X}(b_{(0)} \sigma_{\X}(a)) \varphi_{\G}(gS^{-1}(b_{(1)})) \\ &=& \varphi_{\X}(ab_{(0)})\varphi_{\G}(gS^{-1}(b_{(1)})).\end{eqnarray*} On the other hand, \begin{eqnarray*} \varphi_{\X}(b \sigma_{\X}(a_{(0)})) \varphi_{\G}(gS^{-2}(a_{(1)})) &=&  \varphi_{\X}(a_{(0)}b) \varphi_{\G}(gS^{-2}(a_{(1)})) \\ &=& \varphi_{\X}(a_{(0)}b_{(0)}) \varphi_{\G}(gS^{-2}(a_{(1)}b_{(1)}S(b_{(2)}))) \\ &=&  \varphi_{\X}(ab_{(0)}) \varphi_{\G}(gS^{-1}(b_{(1)})).\end{eqnarray*} By faithfulness of $\varphi_{\X}$ and $\varphi_{\G}$, this implies the result.  
\end{proof} 
\begin{Cor}\label{CorPosF} There exists an invertible positive operator $F_{\pi}$ on $\Mor(\pi,\alpha)$ such that \[\sigma_{\X}(U_{\pi}(T,\xi)) = U_{\pi}(F_{\pi}T,Q_{\pi}\xi).\]
\end{Cor} 

\begin{proof} Define $F_{\pi}T$ by \[(F_{\pi}T)(\xi) =  \sigma_{\X}(T(Q_{\pi}^{-1}\xi)),\] which makes sense as a linear map from $\Hsp_{\pi}$ to $C(\X)$. If we can show that it is equivariant, we have shown that $F_{\pi}$ exists as an invertible linear map. 

From the defining property of $f$ and $Q_{\pi}$ in Definition \ref{TheoWorChar}, we have $S^{-2}(U_{\pi}(\xi,\eta)) = U_{\pi}(Q_{\pi}^{-1}\xi,Q_{\pi}\eta)$. Hence  \begin{eqnarray*}  \alpha((F_{\pi}T)(\xi)) &=& \alpha(\sigma_{\X}(T(Q_{\pi}^{-1}\xi))) \\ &=& (\sigma_{\X}\otimes S^{-2})\alpha(T(Q_{\pi}^{-1}\xi)) \\ &=& (\sigma_{\X} \circ T\otimes S^{-2})\delta_{\pi}(Q_{\pi}^{-1}\xi) \\ &=& (\sigma_{\X} \circ T \circ Q_{\pi}^{-1}\otimes \id_{\G})\delta_{\pi}(\xi) \\ &=& (F_{\pi}T\otimes \id_{\G})\delta_{\pi}(\xi).\end{eqnarray*}

Finally, to see that $F_{\pi}$ is positive, fix $T \in \Mor(\pi,\alpha)$ and an orthonormal basis $\{e_i\}$ of $\Hsp_{\pi}$, and write \[A_{ij} =  \varphi_{\X}\left(U_{\pi}(T,e_j)^*U_{\pi}(T,e_i)\right).\] Then $A$ is a positive matrix. We now compute
\begin{eqnarray*}\langle F_{\pi}^{-1}T,T\rangle &=& \sum_i \varphi_{\X}(U_{\pi}(T,e_i)(\sigma_{\X}^{-1}(U_{\pi}(T,Q_{\pi}e_i)))^*) \\&=& \sum_i \varphi_{\X}(U_{\pi}(T,Q_{\pi}e_i)^*U_{\pi}(T,e_i)) \\&=& \sum_{i,j} \langle Q_{\pi}e_i,e_j\rangle \varphi_{\X}(U_{\pi}(T,e_j)^*U_{\pi}(T,e_i)) \\&=& Tr(Q_{\pi}A) \\& \geq &0.\end{eqnarray*} 
\end{proof} 

Note that since the $F_{\pi}$ are positive, we can define \[\sigma_z^{\X}(U_{\pi}(T,\xi)) = U_{\pi}(F_{\pi}^{iz}T,Q_{\pi}^{iz}\xi),\quad z\in \C.\] As then, for $n\in \Z$ and $a,b\in \mathscr{O}_{\G}(\X)$, \[\sigma_{-in}^{\X}(ab) = \sigma_{\X}^n(ab) = \sigma_{\X}^n(a)\sigma_{\X}^n(b) = \sigma_{-in}^{\X}(a)\sigma_{-in}^{\X}(b),\] it follows by analyticity that the $\sigma_z^{\X}$ form a complex one-parametergroup of algebra automorphisms of $\mathscr{O}_{\G}(\X)$. A similar analyticity argument shows that \[\sigma_z^{\X}(a)^* = \sigma_{\bar{z}}^{\X}(a^*),\quad a\in \mathscr{O}_{\G}(\X).\]

\begin{Cor}\label{CorOrthErg} With respect to $\langle a,b\rangle = \varphi_{\X}(a^*b)$, the $C(\X)_{\pi}$ are mutually orthogonal for non-equivalent $\pi$. Moreover, \[\varphi_{\X}(U_{\pi}(T,\xi)U_{\pi}(S,\eta)^*) = \frac{\langle S,T\rangle \langle \eta,Q_{\pi}\xi\rangle}{\Tr(Q_{\pi})}\]\[\varphi_{\X}(U_{\pi}(S,\eta)^* U_{\pi}(T,\xi)) = \frac{\langle S,F_{\pi}^{-1}T\rangle \langle \eta,\xi\rangle}{\Tr(Q_{\pi})}.\] 
\end{Cor} 
\begin{proof} The first orthogonality relation follows from \begin{eqnarray*}&& \hspace{-2cm}  \varphi_{\X}(U_{\pi}(T,\xi)U_{\pi}(S,\eta)^*) \\ &=& \sum_{k,l}\varphi_{\X}(U_{\pi}(T,e_k)U_{\pi}(S,e_l)^*)\varphi_{\G}(U_{\pi}(e_k,\xi)U_{\pi}(e_l,\eta)^*) \\ &=& \sum_k \varphi_{\X}(U_{\pi}(T,e_k)U_{\pi}(S,e_k)^*)\frac{\langle \eta,Q_{\pi}\xi\rangle}{\Tr(Q_{\pi})} \\ &=& \frac{\langle S,T\rangle \langle \eta,Q_{\pi}\xi\rangle}{\Tr(Q_{\pi})}.\end{eqnarray*} 
Also the mutual orthogonality of the different $C(\X)_{\pi}$ follows from this computation. The second orthogonality relation follows from using the concrete formula and defining property of the modular automorphism.
\end{proof} 

\begin{Def} The \emph{quantum multiplicity} of $\pi$ in $C(\X)$ is defined as \[\mathrm{mult}_q(\pi) = \sqrt{\mathrm{Tr}(F_{\pi})\mathrm{Tr}(F_{\pi}^{-1})}.\]
\end{Def} 

We denote $\mathrm{mult}(\pi) = \dim(\Mor(\pi,\alpha))$. 

\begin{Theorem}\label{TheoEstpi} Let $\X\curvearrowleft \G$ be homogeneous. Then for each irreducible representation $\pi$ of $\G$, one has \[\mult(\pi) \leq \mult_q(\pi) \leq \dim_q(\pi).\] 
\end{Theorem}
\begin{proof} Since $F_{\pi}$ is a positive, invertible operator, the inequality $\mult(\pi)\leq \mult_q(\pi)$ is immediate. 

On the other hand, for $T\in \Mor(\pi,\alpha)$, write \[T^{\dag}(\xi^*) = T(\xi)^*,\quad \xi \in \Hsp_{\pi}.\] Then clearly $T^{\dag}\in \Mor(\bar{\pi},\alpha)$. We  have \[U_{\pi}(T,\xi)^* = U_{\bar{\pi}}(T^{\dag},\xi^*),\] and so \begin{multline*} U_{\bar{\pi}}(F_{\bar{\pi}}(T^{\dag}),Q_{\bar{\pi}}(\xi^*)) = \sigma_{\X} (U_{\bar{\pi}}(T^{\dag},\xi^*)) = (\sigma_{\X}^{-1}(U_{\pi}(T,\xi)))^*   \\= U_{\pi}(F_{\pi}^{-1}T,Q_{\pi}^{-1}\xi)^* = U_{\bar{\pi}}((F_{\pi}^{-1}T)^{\dag},(Q_{\pi}^{-1}\xi)^*).\end{multline*}
Since we know $\dim_q(\pi) = \Tr(Q_{\bar{\pi}}) = \Tr(Q_{\pi}^{-1})$, we must have also $\Tr(F_{\bar{\pi}}) = \Tr(F_{\pi}^{-1})$.  

It hence suffices to show that $\Tr(F_{\pi}^{-1})\leq \Tr(Q_{\pi})$. But choose an orthonormal basis $T_i$ in  $\Mor(\pi,\alpha)$, and an orthonormal basis $\{e_j\}$ in $\Hsp_{\pi}$. Write $U_{\X,\pi} = \left(U_{\pi}(T_i,e_j)\right)_{i,j}$, which is a rectangular matrix over $C(\X)$. Then \[\left(U_{\X,\pi}U_{\X,\pi}^*\right)_{i,k} = \sum_{j} U_{\pi}(T_i,e_j)U_{\pi}(T_k,e_j)^* = \langle T_k,T_i\rangle = \delta_{i,k}.\] It follows that $U_{\X,\pi}$ is a coisometry, and so $\|U_{\X,\pi}\|\leq 1$.

But write \[A_{ij}^{(k)} = \varphi_{\X}\left(U_{\pi}(T_k,e_j)^*U_{\pi}(T_k,e_i)\right).\] Then from the proof of Corollary \ref{CorPosF}, we find that \begin{eqnarray*} \Tr(F_{\pi}^{-1}) &=& \sum_k \langle F_{\pi}^{-1}T_k,T_k\rangle \\ &=& \sum_k \Tr(Q_{\pi}A^{(k)}) \\ &\leq & \Tr(Q_{\pi}) \|\sum_k A^{(k)}\| \\ &\leq & \Tr(Q_{\pi}) \|(\sum_k (U_{\pi}(T_k,e_j))^*U_{\pi}(T_k,e_i))_{j,i}\| \\ &=& \Tr(Q_{\pi}) \|U_{\X,\pi}^*U_{\X,\pi}\| \\ &\leq & \Tr(Q_{\pi}).\end{eqnarray*}
\end{proof} 

Note that in general, there is no connection between $\mult(\pi)$ and $\dim(\pi) = \dim(\Hsp_{\pi})$, the classical dimension of $\pi$. Indeed, one can construct examples where $\mult(\pi)>\dim(\pi)$, see Section \ref{SecTor}. 

When $\G = G$ is a compact Hausdorff group, one can show that $\varphi_{\X}$ must be tracial, see \cite{HLS81}. See also \cite{Pin13} for more information concerning growth rates of the $\mult_{q}(\pi)$.

Another consequence of Boca's theorem is the atomic character of the crossed product $C_0(\X\rtimes \G)$. 

\begin{Theorem}[\cite{Boc95}]\label{TheoBocaAtom} Let $\X\overset{\alpha}{\curvearrowleft} \G$ be homogeneous. Then there exists a set $I$ and Hilbert spaces $\Hsp_i$ for $i\in I$ such that \[C_0(\X\rtimes \G) \cong \oplus_{i\in I} B_0(\Hsp_i)\] 
\end{Theorem}
\begin{proof} It follows from the proof of Theorem \ref{TheoUnicrossC} and the finite-dimensionality of the isotypical components that, for each projection $p$ which is a finite sum of distinct $p_{\pi}$'s, the $^*$-algebra $pC_0(\X\rtimes \G)p$ is a finite-dimensional C$^*$-algebra. We leave it to the reader to check that this implies that $C_0(\X\rtimes \G)$ can be realized as a C$^*$-subalgebra of a C$^*$-algebra of compact operators on a Hilbert space, and must hence be of the above form. 
\end{proof}

One can ask if the index set $I$ has any other meaning, apart from being the set of equivalence classes of irreducible representations of $C_0(\X\rtimes \G)$. It turns out that there is indeed another interpretation of $I$ as classifying the equivariant Hilbert modules of $C(\X)$. This is sometimes referred to as the \emph{Green-Julg isomorphism}, see \cite{Ver02} and \cite{Tom08}. 

\begin{Def}[\cite{BS93}] Let $\X\overset{\alpha}{\curvearrowleft} \G$ be an action of $\G$ on the locally compact quantum space $\X$. A \emph{$\G$-equivariant (right) Hilbert $C_0(\X)$-module} consists of a (right) Hilbert $C_0(\X)$-module $\mathscr{E}$ together with a coaction \[\alpha: \mathscr{E}\rightarrow \mathscr{E}\otimes C(\G)\] satisfying
\begin{itemize}
\item the coassociativity condition \[(\id_{\mathscr{E}}\otimes \Delta)\alpha = (\alpha\otimes \id_{\G})\alpha,\] 
\item  the density condition \[\lbrack \alpha(\mathscr{E})(1_{\X}\otimes C(\G))\rbrack = \mathscr{E}\otimes C(\G),\] 
\item the compatibility condition \[\langle \alpha(\xi),\alpha(\eta)\rangle_{\X\times \G} = \alpha(\langle \xi,\eta\rangle_{\X}).\]
\end{itemize}
\end{Def} 
Here $\mathscr{E}\otimes C(\G)$ is seen as the external tensor product of respectively the Hilbert $C_0(\X)$-module $\mathscr{E}$ and the Hilbert $C(\G)$-module $C(\G)$. 

\begin{Lem} If $\mathscr{E}$ is a $\G$-equivariant Hilbert $C_0(\X)$-module, then \[\alpha(\xi a) = \alpha(\xi)\alpha(a),\quad \xi\in \mathscr{E},a\in C_0(\X).\] 
\end{Lem} 

\begin{proof} Note that for $g\in C(\G)$, $\xi,\eta\in \mathscr{E}$ and $a\in C_0(\X)$, one has \begin{multline*} \langle \alpha(\xi)(1\otimes g),\alpha(\eta a)\rangle_{\X\times \G} = (1_{\X}\otimes g^*)\alpha(\langle \alpha(\xi),a\eta\rangle_{\X}) \\=\langle \alpha(\xi)(1\otimes g),\alpha(a)\alpha(\eta)\rangle_{\X\times \G}.\end{multline*} By the density condition, we find $\alpha(\eta a) = \alpha(\eta)\alpha(a)$. 
\end{proof}

Classically, the notion of equivariant Hilbert $C_0(\X)$-module corresponds to that of \emph{equivariant Hilbert space bundle} over $\X$.  

We will now see that in the case of $\X$ homogeneous, there is a one-to-one correspondence between $C_0(\X\rtimes \G)$-representations and equivariant Hilbert $C_0(\X)$-modules.

We begin by showing that any equivariant Hilbert module is a direct sum of \emph{irreducible} ones. 

\begin{Def} Let $\X\curvearrowleft \G$. An equivariant Hilbert module $\mathscr{E}$ is called \emph{irreducible} if any equivariant Hilbert $C_0(\X)$-submodule is either $\{0\}$ or $\mathscr{E}$. It is called \emph{indecomposable} if it is not isomorphic to the direct sum of two non-trivial equivariant Hilbert $C_0(\X)$-modules. 
\end{Def}

\begin{Theorem}\label{TheoEqHilb} Assume $\X\curvearrowleft \G$ homogeneous. Then any equivariant Hilbert module $\mathscr{E}$ is a direct sum of indecomposable ones. Moreover, if $\mathscr{E}$ is indecomposable, then it is irreducible, finitely generated projective as a $C(\X)$-module, and with finite-dimensional isotypical components.
\end{Theorem}

\begin{proof} First note that we can define the notion of $\pi$-isotypical component $\mathscr{E}_{\pi}$ for $\mathscr{E}$, with $\pi$ an irreducible $\G$-representation. Exactly the same proof as for Lemma \ref{LemClosSpec} shows that each $\mathscr{E}_{\pi}$ is closed in norm, and the same proof as for Theorem \ref{TheoDensPod} shows that their linear span is dense in $\mathscr{E}$. 

Choose now some $\pi$ for which $\mathscr{E}_{\pi}$ is non-zero, and let $V$ be a non-trivial, finite-dimensional $C(\G)$-subcomodule of $\mathscr{E}_{\pi}$.  Let $\mathscr{F}$ be the Hilbert $C(\X)$-submodule spanned by $V$. We claim that $\alpha$ restricts to $\mathscr{F}$, and that the isotypical components of $\mathscr{F}$ are finite-dimensional.

The fact that $\alpha$ restricts to $\mathscr{F}$ follows immediately from \[\alpha(V\mathscr{O}_{\G}(\X)) \subseteq (V\mathscr{O}_{\G}(\X))\aotimes \mathscr{O}(\G).\] To see that the spectral subspaces of $\mathscr{F}$ are finite-dimensional, note that by a similar argument as in Lemma \ref{LemClosSpec} \[\mathscr{F}_{\rho} = \left\lbrack  (\id_{\mathscr{E}}\otimes \varphi_{\G})(\alpha(V)\alpha(\mathscr{O}_{\G}(\X))(1_{\X}\otimes \chi_{\rho}^*))\right\rbrack.\] However, since $\alpha(V) \subseteq V\otimes C(\G)_{\pi}$, we have $(\id_{\mathscr{E}}\otimes \varphi_{\G})(\alpha(\eta)\alpha(a)(1_{\X}\otimes \chi_{\rho}^*))$ non-zero for some $a\in C(\X)_{\theta}$ only if $\rho$ is a subrepresentation of $\pi\otimes \theta$, which by Frobenius reciprocity means that $\theta$ is a subrepresentation of $\bar{\rho}\otimes \pi$. As the $C(\X)_{\theta}$ are finite-dimensional, this shows that $\mathscr{F}_{\rho}$ is finite-dimensional. 

We now show that $\mathscr{F}$ is complemented in $\mathscr{E}$ as an equivariant Hilbert $C(\X)$-module, and that $\mathscr{F}$ is a (finite) direct sum of indecomposable equivariant Hilbert $C(\X)$-modules. This will clearly imply that $\mathscr{E}$ is a direct sum of indecomposable $C(\X)$-modules.

In fact, let \[\mathcal{K}_{\G}(\mathscr{F}) = \{T\in \mathcal{K}(\mathscr{F})\mid \alpha(T\xi) = (T\otimes \id_{\G})\alpha(\xi)\},\] where $\mathcal{K}(\mathscr{F})$ is the C$^*$-algebra of compact operators on $\mathscr{F}$. We claim that $\mathcal{K}_{\G}(\mathscr{F})$ is a finite-dimensional C$^*$-algebra. Clearly, it is a norm-closed algebra, while \begin{eqnarray*} \langle \alpha(T^*\eta),\alpha(\xi)(1_{\X}\otimes g)\rangle_{\X\times \G}  &=& \alpha(\langle \eta,T\xi\rangle_{\X})(1_{\X}\otimes g)  \\ &=&  \langle \alpha(\eta),\alpha(T\xi)(1_{\X}\otimes g)\rangle_{\X\times \G}\\ &=& \langle \alpha(\eta),(T\otimes \id_{\G})\alpha(\xi)(1_{\X}\otimes g)\rangle_{\X\times \G} \\ &=& \langle (T^*\otimes \id_{\G})\alpha(\eta),\alpha(\xi)(1_{\X}\otimes g)\rangle_{\X\times \G}\end{eqnarray*} shows that it is a C$^*$-algebra. As the elements in $\mathcal{K}_{\G}(\mathscr{F})$ are $C(\X)$-linear and preserve the isotypical components, it follows that each element is determined by its restriction to the finite-dimensional space $\mathscr{F}_{\pi}$, hence $\mathcal{K}_{\G}(\mathscr{F})$ is finite-dimensional.

Let $p$ be the unit in $\mathcal{K}_{\G}(\mathscr{F})$. We want to show that $p$ is the unit operator on $\mathscr{F}$. However, by considering $\id_{\mathscr{F}}-p$, it is sufficient to show that $\mathcal{K}_{\G}(\mathscr{F})$ is non-zero. For this, take $\eta \in \mathscr{F}_{\pi}$ a non-zero vector, so $\alpha(\eta) \in \mathscr{F}_{\pi}\otimes C(\G)_{\pi}$. Let \[T = (\id_{\X}\otimes \varphi_{\G})(\alpha(\eta)\alpha(\eta)^*) \in \mathcal{K}(\mathscr{F}).\] Then clearly $T$ is compact, while an easy computation shows that \[\alpha(T\xi) = (T\otimes \id_{\G})\alpha(\xi),\quad \forall \xi\in \mathscr{E}.\] If $T = 0$, then $\alpha(\eta)\alpha(\eta)^*=0$ by faithfulness of $\varphi_{\G}$ on $\mathscr{O}(\G)$, and hence $\alpha(\eta)=0$.  But as $\mathscr{F}_{\pi}$ is an $\mathscr{O}(\G)$-comodule, we then have $\eta = (\id_{\mathscr{F}}\otimes \varepsilon)\alpha(\eta) = 0$, a contradiction. 

To prove the remainder of the theorem, let $\mathscr{E}$ be indecomposable. If $\mathscr{F}$ is an equivariant submodule of $\mathscr{E}$, it follows from the above that it is complemented in $\mathscr{E}$. Hence $\mathscr{E}$ is irreducible. It is finitely generated projective as $\mathcal{K}(\mathscr{E})$ has a unit by the above arguments. 
\end{proof} 

The next lemma provides a particular way to construct equivariant Hilbert $C(\X)$-modules, which afterwards we will show produces \emph{all} of them. 

\begin{Lem} Let $\Hsp_{\pi}$ be a representation of $\G$. Then $\Hsp_{\pi}\otimes C(\X)$ is a $\G$-equivariant Hilbert $C(\X)$-module for the Hilbert $C(\X)$-module structure \[\langle \xi\otimes a,\eta\otimes b\rangle_{\X} = \langle \xi,\eta\rangle a^*b\] and the coaction \[\alpha_{\pi}(\xi\otimes a) = \delta_{\pi}(\xi)_{13}\alpha(a)_{23}.\] 
\end{Lem}
\begin{proof} Exercise.
\end{proof} 

\begin{Theorem}\label{TheoTensEq} Assume $\X\overset{\alpha}{\curvearrowleft} \G$ homogeneous. Then any irreducible $\G$-equivariant Hilbert $C(\X)$-module arises as a $\G$-equivariant Hilbert submodule of $\Hsp_{\pi}\otimes C(\X)$ for some irreducible $\G$-representation $\pi$. 
\end{Theorem}
 
\begin{proof} Let $\mathscr{E}$ be an irreducible equivariant Hilbert $C(\X)$-module. Take any representation $\pi$ such that $\mathscr{E}_{\pi}\neq \{0\}$, and any non-zero $\eta\in \mathscr{E}_{\pi}$. Endow $C(\G)_{\bar{\pi}}$ with any Hilbert space norm making the comodule map \[C(\G)_{\bar{\pi}}\rightarrow C(\G)_{\bar{\pi}}\otimes C(\G),\quad g\mapsto g_{(2)}\otimes S^{-1}(g_{(1)})\] unitary. Then it is easy to check that the linear map \[T: \mathscr{E}\rightarrow C(\G)_{\bar{\pi}}\otimes C(\X),\quad\xi \mapsto \eta_{(1)}^*\otimes \eta_{(0)}^*\xi\] is a right $C(\X)$-linear map such that \[\alpha_{\bar{\pi}}(T\xi) = (T\otimes \id_{\G})\alpha(\xi).\] As $T$ is a linear map between projective, finitely generated Hilbert $C(\X)$-modules, it must necessarily be adjointable, and the same proof as in Theorem \ref{TheoEqHilb} shows that $T^*$ is again $\G$-equivariant. It follows that $T^*T \in \mathcal{K}_{\G}(\mathscr{E}) = \C\id_{\mathscr{E}}$. As $T$ is clearly non-zero, it follows that we can scale $T$ such that it is an isometry. This realizes $\mathscr{E}$ as a $\G$-equivariant Hilbert submodule of $C(\G)_{\bar{\pi}}\otimes C(\X)$. 

Now by the orthogonality relations in $C(\G)$ we can conclude that $C(\G)_{\bar{\pi}}$ is a direct sum of the irreducible unitary representations $\Hsp_{\pi}$. It then follows that $\mathscr{E}$ must also embed as a $\G$-equivariant Hilbert $C(\X)$-submodule of $\Hsp_{\pi}\otimes C(\X)$.  
\end{proof}

Assume now that $\mathscr{E}$ is an irreducible $\G$-equivariant Hilbert $C(\X)$-module for $\X\overset{\alpha}{\curvearrowleft} \G$ homogeneous. It follows from the above proofs that we can endow $\mathscr{E}^* = \mathcal{K}(\mathscr{E},C(\X))$ with an inner product such that \[\langle \xi^*,\eta^* \rangle \id_{\mathscr{E}} = (\id_{\mathscr{E}}\otimes \varphi_{\G})(\alpha(\xi)\alpha(\eta)^*).\]  We let $L^2(\mathscr{E}^*)$ be the separation-completion of $\mathscr{E}^*$ with the above inner product. 

\begin{Theorem} Let $\X\overset{\alpha}{\curvearrowleft} \G$ homogeneous, and let $\mathscr{E}$ be an irreducible $\G$-equivariant Hilbert $C(\X)$-module. Then there exists a unique irreducible $^*$-representation $\pi_{\mathscr{E}}$ of $C_0(\X\rtimes \G)$ on $L^2(\mathscr{E}^*)$ such that \[\pi_{\mathscr{E}}(a\omega) \eta^* = a(\id_{\mathscr{E}^*}\otimes \omega)(\alpha(\eta)^*),\quad a\in \mathscr{O}_{\G}(\X),\omega \in \mathscr{O}(\widehat{\G}),\eta\in \mathscr{E}.\] Moreover, any irreducible $C_0(\X\rtimes \G)$-representation arises in this way, and two irreducible equivariant Hilbert $C(\X)$-modules are isomorphic if and only if the associated $C_0(\X\rtimes \G)$-representations are equivalent. 
\end{Theorem} 

\begin{proof} We leave it as an exercise to check that $\pi_{\mathscr{E}}$ exists. 

To see that it is irreducible, assume that $T$ is a $C_0(\X\rtimes \G)$-intertwiner on $L^2(\mathscr{E}^*)$. Then necessarily $T$ must preserve the finite-dimensional isotypical components $\mathscr{E}_{\pi}^* \subseteq \mathscr{E}^*$. Hence $T$ induces a $\mathscr{O}_{\G}(\X)$-linear map on $\sum_{\pi} \mathscr{E}_{\pi}$ by $T'\xi = (T\xi^*)^*$. By the proof of Theorem \ref{TheoEqHilb}, we can write the unit operator on $\mathscr{E}$ as a finite sum $\sum_i c_{ij}\eta_i\eta_j^*$ with $\eta_i$ in a fixed $\mathscr{E}_{\pi}$ and $c_{ij}\in \C$. It follows that for any $\xi\in \sum_{\pi} \mathscr{E}_{\pi}$, we have \[T'\xi =  \sum_i c_{ij}(T'\eta_i)\langle \eta_j,\xi\rangle_{\X},\] so that $T'$ extends to an operator in $\mathcal{K}(\mathscr{E})$. As this extension is clearly still $\G$-equivariant by continuity, the irreducibility of $\mathscr{E}$ implies that $T'$ is a scalar. 

The same argument shows that inequivalent equivariant Hilbert $C(\X)$-modules produce inequivalent $C_0(\X\rtimes \G)$-representations. 

Finally, to see that any irreducible $C_0(\X\rtimes \G)$-representation arises in this way, recall that by definition $C_0(\X\rtimes \G)$ is faithfully represented on the Hilbert space $L^2(\X)\otimes L^2(\G)$ by \[a\omega \mapsto \alpha(a)(\id_{\X}\otimes l_{\omega}).\] However, conjugating this representation by means of the unitary $U_{\alpha}^*$ from Lemma \ref{LemIntroImplUn}, we see that it is equivalent to the representation $\theta$ where \[\theta(a\omega)(b\xi_{\X}\otimes g\xi_{\G}) = \omega(b_{(1)}g_{(2)})ab_{(0)}\xi_{\X}\otimes g_{(1)}\xi_{\G},\quad b\in \mathscr{O}_{\G}(\X),g\in \mathscr{O}(\G).\] In particular, we see that this representation restricts to each $L^2(\X)\otimes C(\G)_{\pi}$, and that the latter are direct sums of $^*$-representations on $L^2(\X)\otimes \Hsp_{\pi}$ by \[\theta_{\pi}(a\omega)(b\xi_{\X}\otimes \eta) = \omega(b_{(1)}\eta_{(1)})ab_{(0)}\xi_{\X}\otimes \eta_{(0)},\quad b\in \mathscr{O}_{\G}(\X),\eta\in \Hsp_{\pi}.\]

It is hence sufficient to show that each $L^2(\X)\otimes V_{\bar{\pi}}$ is equivalent to a direct sum of representations of the form $L^2(\mathscr{E}^*)$. However, consider the $\G$-equivariant Hilbert $C(\X)$-module $\Hsp_{\pi}\otimes C(\X)$. Then we can separate and complete $(\Hsp_{\pi}\otimes C(\X))^*$ into a Hilbert space $L^2((\Hsp_{\pi}\otimes C(\X))^*)$ with $C_0(\X\rtimes \G)$-representation by means of the inner product \[\langle (\xi\otimes a)^*,(\eta\otimes b)^*\rangle\id_{\Hsp_{\pi}} = \langle a^*\xi_{\X},b^*\xi_{\X}\rangle_{L^2(\X)} \langle \xi^*,\eta^*\rangle_{\Hsp_{\bar{\pi}}}\] and the $^*$-representation \[\theta_{\pi}^*(a\omega)((\xi\otimes b)^*) = \omega(b_{(1)}^*\xi_{(1)}^*)\xi_{(0)}^*\otimes ab_{(0)}^*,\]
where $a,b\in \mathscr{O}_{\G}(\X),\xi,\eta\in \Hsp_{\pi},\omega \in \mathscr{O}(\widehat{\G})$. By this construction, we have $L^2((\Hsp_{\pi}\otimes C(\X))^*)\cong \oplus L^2(\mathscr{E}_i^*)$ as $C_0(\X\rtimes \G)$-representations if $\Hsp_{\pi}\otimes C(\X)\cong \oplus \mathscr{E}_i$. However, we also have that the map \[L^2((\Hsp_{\pi}\otimes C(\X))^*) \rightarrow L^2(\X)\otimes \Hsp_{\bar{\pi}},\quad (\xi\otimes a)^* \mapsto a^*\xi_{\X}\otimes \xi^*\] is a unitary equivalence of the $C_0(\X\rtimes \G)$-representations $\theta_{\pi}^*$ and $\theta_{\bar{\pi}}$. This concludes the proof. 
\end{proof} 

As an example, consider a quantum homogeneous space of quotient type, \[\X = \mathbb{H}\backslash \G.\] Denote \[\pi_{\mathbb{H}}: C(\G) \rightarrow C(\mathbb{H})\] for the quotient map. Define, for $\pi$  a representation of $\mathbb{H}$,
\[ \Hsp_{\pi} \underset{\mathbb{H}}{\square} C(\G) = \{x\in \Hsp_{\pi}\otimes C(\G)\mid (\delta_{\pi}\otimes \id_{\G})x = (\id_{\pi}\otimes (\pi_{\mathbb{H}}\otimes \id_{\G})\Delta)x\},\] which is a closed subspace of the Hilbert $C(\G)$-module $\Hsp_{\pi}\otimes C(\G)$.  Then we find that \[1_{\mathbb{H}}\otimes x^*y = ((\pi_{\mathbb{H}}\otimes \id_{\G})\Delta)x^*y,\quad x,y\in \Hsp_{\pi}\otimes C(\G),\] so that $x^*y \in C(\mathbb{H}\backslash \G)$. Together with the coaction \[x \mapsto (\id_{\pi}\otimes \Delta)x,\quad x \in \Hsp_{\pi} \underset{\mathbb{H}}{\square} C(\G),\] this turns $\Hsp_{\pi} \underset{\mathbb{H}}{\square} C(\G)$ into a $\G$-equivariant right Hilbert $C(\mathbb{H}\backslash \G)$-module. 

\begin{Lem}If $\pi$ is irreducible as an $\mathbb{H}$-representation, then $\Hsp_{\pi} \underset{\mathbb{H}}{\square} C(\G)$ is irreducible as a $\G$-equivariant right Hilbert $C(\mathbb{H}\backslash \G)$-module. Moreover, two such $\G$-equivariant Hilbert $C(\mathbb{H}\backslash \G)$-modules are equivalent if and only if the corresponding $\mathbb{H}$-representations are equivalent.
\end{Lem}

\begin{proof} It suffices to show that, for $\pi$ a general $\mathbb{H}$-representation, the only $\G$-equivariant compact operators on $\Hsp_{\pi} \underset{\mathbb{H}}{\square} C(\G)$ are of the form \[ \xi \mapsto (T'\otimes\id_{\G}) \xi\] for $T'\in \mathrm{End}(\pi)$.  

Now, such a $\G$-equivariant compact operator can be represented as left multiplication with an element \[T \in B(\Hsp_{\pi})\otimes C(\G).\]

We claim that \begin{equation}\label{EqEqT}(\id_{\pi}\otimes \Delta)T  = T\otimes 1_{\G}.\end{equation} For this, one first verifies by a small computation that elements of the form \begin{equation}\label{EqFormInv}\varphi_{\mathbb{H}}(\xi_{(1)}\pi_{\mathbb{H}}(S(g_{(1)}))) \xi_{(0)}\otimes g_{(2)}\end{equation} are in $\Hsp_{\pi}\underset{\mathbb{H}}{\square} C(\G)$ for any $\xi\in \Hsp_{\pi}$ and $g\in \mathscr{O}(\G)$. As $\mathscr{O}(\G) \rightarrow \mathscr{O}(\mathbb{H})$ is surjective, it follows that the right $C(\G)$-module spanned by $\Hsp_{\pi}\underset{\mathbb{H}}{\square} C(\G)$ is dense in $\Hsp_{\pi}\otimes C(\G)$. As then \eqref{EqEqT} holds when interpreted as left multiplication operators on $\Hsp_{\pi}\underset{\mathbb{H}}{\square} C(\G)$, it holds as well as left multiplication operators on $\Hsp_{\pi}\otimes C(\G)$, implying \eqref{EqEqT}. 
 
Applying $\varphi_{\G}$ to the last leg of \eqref{EqEqT}, we find that \[T = T'\otimes 1_{\G} \in B(\Hsp_{\pi})\otimes 1_{\G}.\] 

Note now that the linear span of \begin{equation}\label{EqSpanFirst}\{(\id_{\pi}\otimes \omega)x \mid x\in \Hsp_{\pi}\underset{\mathbb{H}}{\square} C(\G),\omega\in C(\G)^*\}\end{equation} determines a subcomodule of $\Hsp_{\pi}$. By irreducibility of $\Hsp_{\pi}$, it must be either zero or $\Hsp_{\pi}$. However, since $\pi_{\mathbb{H}}: C(\G) \rightarrow C(\mathbb{H})$ is surjective, the elements of the form \eqref{EqFormInv} can not all be zero, showing that the linear span of \eqref{EqSpanFirst} is equal to $\Hsp_{\pi}$. 

Take now $\sum_i e_i \otimes f_i \in \Hsp_{\pi}\underset{\mathbb{H}}{\square} C(\G)$. As $T$ preserves $\Hsp_{\pi} \underset{\mathbb{H}}{\square} C(\G)$, we find \[\sum_i \delta_{\pi}(T'e_i) \otimes f_i = \sum_i T'e_i \otimes (\pi_{\mathbb{H}}\otimes \id_{\G})\Delta(f_i) = \sum_i (T'\otimes \id_{\G})\delta_{\pi}(e_i)\otimes f_i.\] As the first leg of $\Hsp_{\pi}\underset{\mathbb{H}}{\square} C(\G)$ spans $\Hsp_{\pi}$, we deduce that $T'$ is a selfintertwiner of $\pi$.
\end{proof}

We will show later, see Corollary \ref{CorEqModQuot} that \emph{all} irreducible $\G$-equivariant Hilbert $C(\mathbb{H}\backslash \G)$-modules arise in this way.

Let us now return to general homogeneous actions. Our next goal is to define a proper type of weak isomorphism between them. We first recall the following definition.

\begin{Def} We call two C$^*$-algebras $C_0(\X)$ and $C_0(\Y)$ \emph{strongly Morita equivalent} if there exists a full right Hilbert $C_0(\X)$-module $\mathscr{E}$ together with an isomorphism $C_0(\Y) \cong \mathcal{K}(\mathscr{E})$. 
\end{Def}

Recall that a Hilbert $C_0(\X)$-module $\mathscr{E}$ is called full if $\lbrack \langle \mathscr{E},\mathscr{E}\rangle_{\X}\rbrack = C_0(\X)$. 

The above definition can be upgraded to $\G$-spaces. 

\begin{Def} We say that two actions $\X\overset{\alpha_{\X}}{\curvearrowleft} \G$ and $\mathbb{Y}\overset{\alpha_{\Y}}{\curvearrowleft} \G$ are \emph{$\G$-equivariantly (strongly) Morita equivalent} if there exists a $\G$-equivariant full right Hilbert $C_0(\X)$-module $\mathscr{E}$ together with an isomorphism $C_0(\Y) \cong \mathcal{K}(\mathscr{E})$ such that \[\alpha_{\Y}(y)\alpha_{\mathscr{E}}(\xi) = \alpha_{\mathscr{E}}(y\xi),\quad \forall y\in C_0(\Y),\xi\in \mathscr{E}.\]  
\end{Def}

Here we have surpressed the notation for the identification of $C_0(\Y)$ with $\mathcal{K}(\mathscr{E})$. We call $\mathscr{E}$ as above a \emph{$\G$-equivariant Hilbert equivalence bimodule} between $C_0(\X)$ and $C_0(\Y)$.  

\begin{Prop} Let $\X\overset{\alpha}{\curvearrowleft}\G$ be homogeneous. Let $\mathscr{E}$ be an irreducible $\G$-equivariant Hilbert $C_0(\X)$-module. Then writing $C(\Y) := \mathcal{K}(\mathscr{E})$, there exists a unique homogeneous action $\Y \overset{\alpha}{\curvearrowleft} \G$ such that $\mathscr{E}$ is a $\G$-equivariant Hilbert equivalence bimodule between $C_0(\X)$ and $C_0(\Y)$.
\end{Prop} 
\begin{proof} As $\mathscr{E}$ is irreducible, we know already that $\mathscr{E}$ is a finitely generated projective $C(\X)$-module. Hence we can define a unique coaction \[\alpha: \mathcal{K}(\mathscr{E})\rightarrow \mathcal{K}(\mathscr{E})\otimes C(\G)\] such that $\alpha(T\xi) = \alpha(T)\alpha(\xi)$ for all $\xi\in \mathscr{E}$ and $T\in \mathcal{K}(\mathscr{E})$. 

Clearly $\alpha$ satisfies the coaction property. To see that it satisfies the Podle\'{s} condition, note that for $\xi\in \mathscr{E}_{\pi}$ for some irreducible $\G$-representation $\pi$, we have \[\xi\otimes 1_{\G} = \xi_{(0)}\otimes S^{-1}(\xi_{(2)})\xi_{(1)}.\] Hence $\lbrack (\id_{\mathscr{E}}\otimes C(\G))\alpha(\mathscr{E})\rbrack = \mathscr{E}\otimes C(\G)$, and so \[\lbrack (\id_{\mathscr{E}}\otimes C(\G))\alpha(\mathcal{K}(\mathscr{E}))\rbrack = \lbrack (\id_{\mathscr{E}}\otimes C(\G))\alpha(\mathscr{E})\alpha(\mathscr{E})^*\rbrack = \mathcal{K}(\mathscr{E})\otimes C(\G).\] By irreducibility of $\mathscr{E}$, it follows also immediately that, writing $C(\Y)  = \mathcal{K}(\mathscr{E})$, the action $\Y \overset{\alpha}{\curvearrowleft} \G$ is homogeneous. 

The only thing which remains to be shown is that $\mathscr{E}$ is full as a right Hilbert $C(\X)$-module. However, write $I = \langle \mathscr{E},\mathscr{E}\rangle_{\X}$, which is a $\G$-invariant closed 2-sided ideal in $C(\X)$. Then for a non-zero $x\in I_{\pi}$ for some irreducible $\G$-representation $\pi$, also $(\id_{\X}\otimes \varphi_{\G})\alpha(x^*x) \in I$ non-zero. As the latter element is a non-zero multiple of  $1_{\X}$ by homogeneity of $\X\curvearrowleft \G$, we find $I = C(\X)$. 
\end{proof}

For example, in \cite{DC12} it was shown that the family of Podle\'{s} spheres is closed under $SU_q(2)$-equivariant Morita equivalence, and that two Podle\'{s} spheres $S_{q,x}^2$ and $S_{q,y}^2$ are $SU_q(2)$-equivariantly Morita equivalent if and only if $x-y \in \Z$ or $x+y\in \Z$. 

To end, let us discuss \emph{fusion rules} for $\G$-equivariant Hilbert modules. These were introduced for ergodic compact Hausdorff group actions on von Neumann algebras in \cite{Was89}, and extended to compact quantum groups in \cite{Tom08}. 

\begin{Def} Let $\X\overset{\alpha}{\curvearrowleft} \G$ be a homogeneous action. Let $\{\mathscr{E}_i\} _{i\in I}$ be a maximal collection of irreducible $\G$-equivariant Hilbert $C(\X)$-modules. For $\pi$ a $\G$-representation, we call the matrix $M_{\alpha}(\pi)$ such that \[\Hsp_{\pi}\otimes \mathscr{E}_i \cong \oplus_{i,j} M_{\alpha}(\pi)_{i,j} \mathscr{E}_j\] the \emph{fusion matrix} for $\pi$.
\end{Def} 

Note that since the isotypical components of $\Hsp_{\pi}\otimes \mathscr{E}_i$ are finite-dimensional by Theorem \ref{TheoEqHilb}, we have $M_{\alpha}(\pi)_{i,j}=0$ for all but a finite number of $j$ when $i$ is fixed. Symmetrically, we have that $M(\pi)_{i,j}=0$ for all but a finite number of $i$ when $j$ is fixed, by the following proposition.

\begin{Prop} If $\mathscr{E}$ and $\mathscr{F}$ are irreducible $\G$-equivariant Hilbert $C(\X)$-modules, and $\pi$ an irreducible $\G$-representation, then $\mathscr{E}$ appears as an equivariant submodule of $\Hsp_{\pi}\otimes \mathscr{F}$ if and only if $\mathscr{F}$ appears as an equivariant submodule of $\Hsp_{\bar{\pi}}\otimes \mathscr{E}$. 
\end{Prop} 
\begin{proof} Let $T$ be a $\G$-equivariant isometric map $\mathscr{F}\rightarrow \Hsp_{\bar{\pi}}\otimes \mathscr{E}$. Since both target and range are finitely generated projective $C(\X)$-modules, $T$ has an adjoint $T^*$ which is again $\G$-equivariant. It is then easy to check that we have, with $\{e_i\}$ an orthonormal basis of $\Hsp_{\pi}$, a  $\G$-equivariant right $C(\X)$-linear map \[\mathscr{E}\rightarrow \Hsp_{\pi}\otimes \Hsp_{\bar{\pi}}\otimes \mathscr{E},\quad \xi\mapsto \sum_i e_i\otimes e_i^*\otimes \xi.\] But then also \[\mathscr{E}\rightarrow \Hsp_{\pi}\otimes \mathscr{F},\quad \xi\mapsto \sum_i e_i \otimes T^*(e_i^*\otimes \xi)\] a non-zero $\G$-equivariant right $C(\X)$-linear map. By irreducibility of $\mathscr{E}$, the latter map must be a scalar multiple of a $\G$-equivariant isometric imbedding. 
\end{proof}

By the above finiteness property, we can multiply the matrices $M(\pi)$, and it is then easy to see that \[M(\pi)M(\pi') = M(\pi\otimes \pi'),\quad M(\pi\oplus \pi') = M(\pi)+M(\pi').\] 

As an example, let us consider the fusion rules for homogeneous actions of quotient type.

\begin{Prop}\label{PropFusSub} Let $\mathbb{H}\subseteq \G$ be an inclusion of compact quantum groups. If $\pi$ is a $\G$-representation, and $\pi'$ an $\mathbb{H}$-representation, then \[\Hsp_{\pi}\otimes (\Hsp_{\pi'} \underset{\mathbb{H}}{\square} C(\G)) \underset{\theta}{\cong} (\Hsp_{\pi_{\mid \mathbb{H}}}\otimes \Hsp_{\pi'})\underset{\mathbb{H}}{\square} C(\G).\] 
\end{Prop} 

Here $\delta_{\pi_{\mid \mathbb{H}}} = ( \id_{\Hsp_{\pi}}\otimes \pi_{\mathbb{H}})\delta_{\pi}$, with $\pi_{\mathbb{H}}:C(\G) \rightarrow C(\mathbb{H})$ the quotient map. Then $\pi_{\mid \mathbb{H}}$ is a representation of $\mathbb{H}$, which we call the restriction of $\pi$ to $\mathbb{H}$. 

\begin{proof} We leave it to the reader to check that \[\theta(\xi\otimes (\sum_i \eta_i\otimes g_i)) \mapsto \sum_i \xi_{(0)}\otimes \eta_i \otimes \xi_{(1)}g_i\] is the sought-after $\G$-equivariant isometric isomorphism. 
\end{proof} 

\begin{Cor}\label{CorEqModQuot} Let $\mathbb{H}\subseteq \G$ be an inclusion of compact quantum groups. Then any irreducible $\G$-equivariant Hilbert $C(\mathbb{H}\backslash \G)$-module is of the form $\Hsp_{\pi}\underset{\mathbb{H}}{\square}C(\G)$ for some irreducible  $\mathbb{H}$-representation $\pi$. 
\end{Cor} 

\begin{proof} Let $\epsilon$ be the trivial representation of $\G$. Then by definition $\Hsp_{\epsilon}\underset{\mathbb{H}}{\square} C(\G) = C(\mathbb{H}\backslash \G)$. Hence by Proposition \ref{PropFusSub}, we find $\Hsp_{\pi} \otimes C(\mathbb{H}\backslash \G) \cong \Hsp_{\pi_{\mid \mathbb{H}}}\underset{\mathbb{H}}{\square} C(\G)$. Since the correspondence $\pi' \mapsto \Hsp_{\pi'} \underset{\mathbb{H}}{\square} C(\G)$ is functorial, and since any $\G$-equivariant $C(\mathbb{H}\backslash \G)$-module $\mathscr{E}$ appears as a direct summand of some $\Hsp_{\pi} \otimes C(\mathbb{H}\backslash \G)$ by Theorem \ref{TheoTensEq}, the corollary follows.
\end{proof}

It follows that we can index a maximal collection of mutually inequivalent irreducible $\G$-equivariant Hilbert $C(\mathbb{H}\backslash \G)$-modules by a maximal collection $I=\{\pi\}$ of mutually inequivalent irreducible $\mathbb{H}$-representations. The fusion rules $M_{\mathbb{H}\backslash \G}(\pi)_{\pi',\pi''}$ for $C(\mathbb{H}\backslash \G)$ are then determined by \[ \pi_{\mid \mathbb{H}} \otimes \pi' \cong \oplus_{\pi''} M_{\mathbb{H}\backslash \G}(\pi)_{\pi',\pi''} \pi'',\] where $\pi',\pi''\in I$.

\section{Free actions}

In this section, we will consider free actions of compact quantum groups. Recall that an action of a compact Hausdorff group $G$ on a locally compact space $X$ is called \emph{free} if  \[\forall x\in X,\quad \{g\in G\mid xg = x\} = \{e_G\},\] that is, the stabilizer group $G_x$ of any point is trivial.

\begin{Lem} The action $X\overset{\alpha}{\curvearrowleft} G$ is free if and only if \[\lbrack (C_0(X)\otimes 1_G)\alpha(C_0(X))\rbrack = C_0(X)\otimes C(G).\]
\end{Lem} 
 \begin{proof} The action $\alpha$ is free if and only if the continuous map \[\mathrm{Can}:X\times G\mapsto X\times X, \quad (x,g) \mapsto (x,xg)\] is injective. But this map is injective if and only if the $^*$-homomorphism \[ \mathrm{Can}:C_0(X)\otimes C_0(X) \rightarrow C_0(X)\otimes C(G),\quad  F \mapsto \left(\mathrm{Can}(F): (x,g)\mapsto F(x,xg)\right)\] is surjective. Now note that for $f,g\in C_0(X)$, we have \[\mathrm{Can}(f\otimes g) = (f\otimes 1_G)\alpha(g).\] This proves the lemma. 
 \end{proof} 

The above lemma is the inspiration for the following definition, introduced in \cite{Ell00}.

\begin{Def}Let $\X$ be a locally compact quantum space, and $\G$ a compact quantum group with $\mathbb{X}\overset{\alpha}{\curvearrowleft} \mathbb{G}$.  We call $\alpha$ a \emph{free action} if \[\lbrack (C_0(\mathbb{X})\otimes 1_{\mathbb{G}})\alpha(C_0(\mathbb{X}))\rbrack = C_0(\mathbb{X})\otimes C(\mathbb{G}).\]
\end{Def}

We now want to give several equivalent characterisations of freeness. The key will be the notion of the \emph{Galois} (or \emph{canonical}) map(s). We will in the following lemma use the interior tensor product of Hilbert bimodules, see \cite{Lan95}. Recall from Definition \ref{DefRightYMod} that if $\X \overset{\alpha}{\curvearrowleft}\mathbb{G}$ and $\Y= \X/\G$, we can form a right Hilbert $C_0(\Y)$-module $L^2_{\Y}(\X)$. This carries a representation of $C_0(\X)$ and hence also $C_0(\Y)$ by left multiplication. 

\begin{Lem} Let $\mathbb{X}\overset{\alpha}{\curvearrowleft}\mathbb{G}$ be an action of $\G$ on a locally compact quantum space $\X$, and write $\Y = \X/\G$. Then there exists an isometry of Hilbert $C_0(\Y)$-modules \[\mathcal{G}_{\alpha}:L^2_{\mathbb{Y}}(\mathbb{X})\underset{C_0(\mathbb{Y})}{\otimes} L^2_{\mathbb{Y}}(\mathbb{X}) \rightarrow  L^2_{\mathbb{Y}}(\mathbb{X})\otimes L^2(\mathbb{G}),\] called the \emph{Galois} (or \emph{canonical}) \emph{isometry}, such that \[\mathcal{G}_{\alpha}(a\otimes b) =  \alpha(a)(b\otimes 1_{\G}),\quad a,b\in C_0(\X).\]
\end{Lem} 
\begin{proof} This follows from the simple computation
\begin{eqnarray*} &&\hspace{-2cm}\langle  \alpha(c)(d\otimes 1_{\G}), \alpha(a)(b\otimes 1_{\G}) \rangle_{\Y} \\ &=& (E_{\mathbb{Y}}\otimes \varphi_{\G})((d^*\otimes 1_{\G})\alpha(c^*a)(b\otimes 1_{\G}))  \\ &=& E_{\mathbb{Y}}(d^*E_{\mathbb{Y}}(c^*a)b) \\ &=& \langle c\otimes d,a\otimes b\rangle_{\Y}.\end{eqnarray*}
\end{proof}

The following theorem gives several equivalent characterisations of freeness. We will not present the proof here, for which we refer to the literature, see \cite{DCY13}. 

We recall that $\mathcal{K}(\mathscr{E})$ denotes the C$^*$-algebra of compact operators on a Hilbert C$^*$-module. We denote by $\mathcal{L}(\mathscr{E})$ the C$^*$-algebra of all adjointable operators. 

\begin{Theorem}\label{TheoFreeEq} Let $\X$ be a locally compact quantum space, endowed with a compact quantum group action $\mathbb{X}\overset{\alpha}{\curvearrowleft} \mathbb{G}$. Then the following are equivalent. 
\begin{enumerate}
\item The action is free.
\item The Galois map $\mathcal{G}_{\alpha}$ is unitary.
\item The natural $^*$-homomorphism \[C_0(\X\rtimes \G) \rightarrow  \mathcal{L}(L^2_{\Y}(\X)),\quad a\omega\mapsto L_al_{\omega}\] is a $^*$-isomorphism $C_0(\X\rtimes \G) \rightarrow  \mathcal{K}(L^2_{\Y}(\X))$, i.e.\ the action is \emph{saturated}. 
\end{enumerate}
\end{Theorem}

The notion of saturated action was introduced in \cite{Phi87} for $\G$ an ordinary compact group, and the equivalence of 1.\ and 3.\ above was proven in this case in \cite{Wah10}.

In particular, we deduce from the theorem that freeness does not depend on which concrete completion of $\mathscr{O}_{\G}(\X)$ one is using.

\begin{Exa} Let $\mathbb{H}\subseteq \mathbb{G}$ be a compact quantum subgroup 
by \[\pi_{\mathbb{H}}: C(\G)\twoheadrightarrow C(\mathbb{H}).\] Then the action $\G\overset{\alpha}{\curvearrowleft} \mathbb{H}$, given by \[\alpha = (\id_{\G}\otimes \pi_{\mathbb{H}})\circ \Delta: C(\mathbb{G})\rightarrow C(\G)\otimes C(\mathbb{H}),\] is free.  
\end{Exa} 
\begin{proof} Exercise. 
\end{proof}

The following lemma provides a `trivial' class of free actions. 

\begin{Lem}\label{LemCrossFree} Let $\X$ be a locally compact space, and assume $\X\curvearrowleft \widehat{\G}$. Then $(\X\rtimes \widehat{\G})\curvearrowleft \G$ is free.
\end{Lem} 

In the above, $\X\rtimes \widehat{\G}$ refers to any $
\G$-equivariant completion of $\mathscr{O}(\X\rtimes \widehat{\G})$.

\begin{proof} We verify the Ellwood condition:
\begin{align*} \lbrack \alpha(C_0(\X\rtimes\widehat{\G}))(C_0(\X\rtimes \widehat{\G})\otimes 1_{\G})\rbrack  &\supseteq \alpha(\mathscr{O}(\X\rtimes\widehat{\G}))(\mathscr{O}(\X\rtimes \widehat{\G})\otimes 1_{\G}) \\ &\supseteq (\mathscr{O}_{\G}(\X)\otimes 1_{\G})\Delta(\mathscr{O}(\G))(\mathscr{O}(\G)\mathscr{O}_{\G}(\X)\otimes 1_{\G}) \\ &= \mathscr{O}_{\G}(\X)\mathscr{O}(\G)\mathscr{O}_{\G}(\X)\aotimes \mathscr{O}(\G) \\ &= \mathscr{O}(\X\rtimes \widehat{\G})\aotimes \mathscr{O}(\G). \end{align*} 
\end{proof}

The following theorem from \cite{BDH13} connects the above analytic theory to the algebraic theory of Galois actions. For more details on the latter, we refer to \cite{Sch04}. 

\begin{Theorem}\label{TheoGalAnAlg} Let $\X$ be a compact quantum space, and $\X\overset{\alpha}{\curvearrowleft} \G$. Then the action is free if and only if the \emph{algebraic} Galois map \[G_{\alpha}: \mathscr{O}_{\G}(\X)\underset{C(\mathbb{Y})}{\otimes} \mathscr{O}_{\G}(\mathbb{X}) \rightarrow  \mathscr{O}_{\G}(\X)\otimes \mathscr{O}(\mathbb{G}),\quad a\otimes b\mapsto \alpha(a)(b\otimes 1_{\G})\]  is an isomorphism. 
\end{Theorem}

In the above theorem, $\mathscr{O}_{\G}(\X)\underset{C(\mathbb{Y})}{\otimes} \mathscr{O}_{\G}(\mathbb{X})$ is the ordinary algebraic balanced tensor product (over the algebra $C(\Y)$). It is not clear if the above theorem still holds for $\X$ non-compact.

\section{Quantum torsors}\label{SecTor}

Freeness is in a sense at the opposite of homogeneity. Nevertheless, it turns out that in the quantum setting, there is a very interesting class of non-trivial actions which are both free and homogeneous. 

\begin{Def} An action $\mathbb{X}\overset{\alpha}{\curvearrowleft} \mathbb{G}$ on a compact quantum space is called a \emph{quantum torsor} (or \emph{Galois object})  if 
\begin{enumerate}
\item $\alpha$ is free,
\item $\alpha$ is homogeneous, and
\item $C(\X)\neq \{0\}$.
\end{enumerate}
\end{Def} 

In the classical context, these actions are very easy to describe. 

\begin{Lem} If $X$ is a compact Hausdorff space, and $X\overset{\alpha}{\curvearrowleft} G$ a (quantum) torsor for a compact Hausdorff group $G$, then there exists an equivariant homeomorphism $X\cong G$, where $G$ acts on $G$ by right translation. 
\end{Lem} 

However, this lemma is no longer true in the quantum setting! The most famous example is the \emph{quantum torus}.

\begin{Exa}
Let $\theta \in [0,2\pi]$. Put \[C(\mathbb{T}_{\theta}^2) = C^*(U,V\mid U,V \textnormal{\textit{unitary, }} UV = e^{i\theta} VU.\},\] where $\mathbb{T}_{\theta}^2$ is called a \emph{quantum torus} (at parameter $\theta$). Then $\mathbb{T}_{\theta}^2 \curvearrowleft \mathbb{T}^2$ by \[\alpha_{(w,z)}(U) = wU,\quad \alpha_{(w,z)}V = zV.\] 
We claim that this is a free and homogeneous action. Indeed, by the commutation relations, any element in $C(\mathbb{T}_{\theta}^2)$ can be approximated by a linear combination of elements of the form $U^nV^m$ with $m,n\in \Z$. But an easy computation shows that \[\int_{\mathbb{T}^2} \alpha_{(w,z)}(U^nV^m) \rd w\rd z = \delta_{m,0}\delta_{n,0}.\] It follows that the conditional expectation \[E: C(\mathbb{T}_{\theta}^2) \rightarrow C(\mathbb{T}_{\theta}^2/\mathbb{T}^2)\] has range in the scalars, and so $C(\mathbb{T}_{\theta}^2/\mathbb{T}^2) = \C$. This shows that the action is homogeneous. 

To see that it is free, let us write $u,v$ for the canonical generators of $C(\mathbb{T}^2)$. Then it is easy to see that the coaction associated to $\mathbb{T}_{\theta}^2 \curvearrowleft \mathbb{T}^2$  is given by \[\alpha: C(\mathbb{T}_{\theta}^2) \rightarrow C(\mathbb{T}_{\theta}^2) \otimes C(\mathbb{T}^2),\quad U^mV^n \mapsto U^mV^n \otimes u^m v^n.\] Hence, by the commutation relations between $U$ and $V$, we have \[\lbrack \alpha(C(\mathbb{T}_{\theta}^2))(C(\mathbb{T}^2_{\theta})\otimes 1)\rbrack \supseteq \{U^k V^l \otimes u^mv^n\mid k,l,m,n\in \Z\}.\] As the latter set has dense linear span in $C(\mathbb{T}_{\theta}^2)\otimes C(\mathbb{T}^2)$, it follows that $\alpha$ is free. 

Finally, to have that $\mathbb{T}_{\theta}^2$ is really a quantum torsor, we have to check that $\mathbb{T}_{\theta}^2$ is not trivial. But consider on $l^2(\Z\times \Z)$ the  unitary operators \[\mathbf{U}e_{m,n} = e_{m+1,n},\quad \mathbf{V}e_{m,n} = e^{-im \theta} e_{m,n+1}.\] Then it is easily checked that $\mathbf{U}$ and $\mathbf{V}$ satisfy the relations of the generators $U$ and $V$ of $C(\mathbb{T}_{\theta}^2)$. Hence $\mathbb{T}_{\theta}^2$ is not trivial. 
\end{Exa} 

The above  is an instance of a general construction, whereby Galois objects are obtained from unitary 2-cocycles on discrete (quantum) groups.

\begin{Def} Let $\G$ be a compact quantum group. A (normalized) \emph{unitary 2-cocycle} on $\widehat{\G}$ is a functional \[\omega: \mathscr{O}(\G)\aotimes \mathscr{O}(\G) \rightarrow \C\] such that, with respect to the convolution $^*$-algebra structure on the dual of $\mathscr{O}(\G)\aotimes \mathscr{O}(\G)$, the functional $\omega$ is a unitary, such that the 2-cocycle condition is satisfied, meaning that for all $g,h,k\in \mathscr{O}(\G)$, \[\omega(g_{(1)}h_{(1)},k)\omega(g_{(2)},h_{(2)}) =  \omega(g,h_{(1)}k_{(1)})\omega(h_{(2)},k_{(2)}),\] and such that $\omega$ is normalized, meaning that \[\omega(1_{\G},g) = \omega(g,1_{\G}) = \varepsilon(g),\quad \forall g\in \mathscr{O}(\G).\] 
\end{Def}

Note that in the case of $\G = \widehat{\Gamma}$ with $\Gamma$ a discrete group, this reduces to the ordinary notion of a cocycle: with $\omega_{g,h} = \omega(g,h)$ for $g,h\in \Gamma \subseteq \C\lbrack \Gamma\rbrack = \mathscr{O}(\widehat{\Gamma})$, the defining relations of $\omega_{g,h}$ say that each $|\omega_{g,h}| = 1$ with \[\omega_{gh,k}\omega_{g,h} = \omega_{g,hk}\omega_{h,k},\quad \forall g,h,k\in \Gamma.\]

\begin{Lem}\label{LemRepDefCoc} Let $\G$ be a compact quantum group, and let $\omega$ be a unitary 2-cocycle for $\G$. Define $\mathscr{O}(\G)_{\omega} = \mathscr{O}(\G_{\omega})$ to be the vector space $\mathscr{O}(\G)$ equipped with the new product \[g\underset{\omega}{\cdot} h = \omega(g_{(2)},h_{(2)}) g_{(1)}h_{(1)}\] and the $^*$-structure \[g^{\#} = \chi_{\omega}(g_{(2)}^*) g_{(1)}^*,\] where $\chi_{\omega}$ is the functional \[\chi_{\omega}: \mathscr{O}(\G)\rightarrow \C,\quad g\mapsto \omega^*(S^{-1}(g_{(2)}),g_{(1)})=\overline{\omega(g_{(2)}^*,S(g_{(1)})^*)}.\] Then $\mathscr{O}(\G)_{\omega}$ is a unital $^*$-algebra with unit $1_{\G}$.
\end{Lem} 

\begin{proof} We leave it as an exercise to check that the product is associative - this is in fact equivalent with $\omega$ satisfying the cocycle identity. We also leave it as an exercise to check that $1_{\G}$ is the unit. 

It is a bit harder to check that we have a $^*$-algebra. Endow $\mathscr{O}(\G)$ with the pre-Hilbert space structure \[\langle a,b\rangle = \varphi_{\G}(a^*b).\] Then we compute that \begin{align*} \langle a,g\underset{\omega}{\cdot} b\rangle &= \omega(g_{(2)},b_{(2)}) \langle a,g_{(1)}b_{(1)}\rangle \\ &= \omega(g_{(2)},b_{(2)})\varphi_{\G}(a^*g_{(1)}b_{(1)}) \\ &= \omega(g_{(4)},S^{-1}(a_{(3)}^*g_{(3)}) a_{(2)}^*g_{(2)}b_{(2)}) \varphi_{\G}(a_{(1)}^*g_{(1)}b_{(1)}) \\ &= \omega(g_{(3)},S^{-1}(a_{(2)}^*g_{(2)}))\varphi_{\G}(a_{(1)}^*g_{(1)}b) \\ &= \left\langle \overline{\omega(g_{(3)},S^{-1}(a_{(2)}^*g_{(2)}))}g_{(1)}^*a_{(1)},b\right\rangle.\end{align*}

We see that the operation $\pi_{\omega}(g)$ of left $\omega$-multiplication with $g$ has an adjoint operator \[\pi_{\omega}(g)^* a = \overline{\omega(g_{(3)},S^{-1}(a_{(2)}^*g_{(2)}))} g_{(1)}^*a_{(1)}.\] Since $\pi_{\omega}(g)^*1_{\G} = g^{\#}$, it is now sufficient to prove that, for all $a\in \mathscr{O}(\G)$, \[\pi_{\omega}(g)^*a = g^{\#}\underset{\omega}{\cdot}a\] or alternatively, \[\overline{\omega(g_{(3)},S^{-1}(g_{(2)})S(a_{(2)})^*)}g_{(1)}^*a_{(1)} = \chi_{\omega}(g_{(3)}^*)\omega(g_{(2)}^*,a_{(2)})g_{(1)}^*a_{(1)}.\] Clearly, it is enough to prove that for all $a,g\in \mathscr{O}(\G)$,\[\overline{\omega(g_{(2)},S^{-1}(g_{(1)})S(a)^*)} = \chi_{\omega}(g_{(2)}^*)\omega(g_{(1)}^*,a).\]  But we have \begin{align*}\omega(g_{(2)},S^{-1}(g_{(1)})S(a)^*) &=  \omega(g_{(2)},S^{-1}(g_{(1)})_{(1)}S(a)^*_{(1)}) \\ &\hspace{3cm} \cdot (\omega\omega^*)(S^{-1}(g_{(1)})_{(2)},S(a)^*_{(2)})
\\ &= 
\omega(g_{(2)},S^{-1}(g_{(1)})_{(1)}S(a)^*_{(1)})\omega(S^{-1}(g_{(1)})_{(2)},S(a)^*_{(2)}) \\ &\hspace{3cm} \cdot
\omega^*(S^{-1}(g_{(1)})_{(3)},S(a)^*_{(3)})\\ &\overset{(*)}{=} \omega(g_{(2)(1)}S^{-1}(g_{(1)})_{(1)},S(a)^*_{(1)})\omega(g_{(2)(2)},S^{-1}(g_{(1)})_{(2)})\\&\hspace{3cm} \cdot 
\omega^*(S^{-1}(g_{(1)})_{(3)},S(a)^*_{(2)})\\&= \omega(g_{(4)}S^{-1}(g_{(3)}),S(a)^*_{(1)})\omega(g_{(5)},S^{-1}(g_{(2)}))\\&\hspace{3cm} \cdot 
\omega^*(S^{-1}(g_{(1)}),S(a)^*_{(2)})  \\ &\overset{(**)}{=} \omega(g_{(3)},S^{-1}(g_{(2)}))\omega^*(S^{-1}(g_{(1)}),S(a)^*) \\&= \overline{\chi_{\omega}(g_{(2)}^*)}\overline{\omega(g_{(1)}^*,a)},
\end{align*}
where in $(*)$ we used the 2-cocycle identity for $\omega$, and in $(**)$ the fact that $\omega$ is normalized. This completes the proof. 
\end{proof} 

\begin{Lem} There exists a left Hopf $^*$-algebraic coaction \[\alpha: \mathscr{O}(\G_{\omega})\rightarrow \mathscr{O}(\G)\otimes \mathscr{O}(\G_{\omega}),\quad g\mapsto \Delta(g).\] 
\end{Lem} 
\begin{proof} Exercise. 
\end{proof} 

In the following, we show that $\mathscr{O}(\G_{\omega})$ has a universal C$^*$-envelope, and that $\G_{\omega}$ is a (left) quantum torsor for $\G$. 

\begin{Theorem}  Let $\G$ be a compact quantum group, equipped with a unitary 2-cocycle $\omega$ on $\widehat{\G}$. Then the universal C$^*$-envelope $C(\G_{\omega,u})$ of $\mathscr{O}(\G_{\omega,u})$ exists, and \[\alpha_u: C(\G_{\omega,u})\rightarrow C(\G_u)\otimes C(\G_{\omega,u})\] is a free and homogeneous action $\G_u \curvearrowright \G_{\omega,u}$. 
\end{Theorem} 

\begin{proof} Since $\mathscr{O}(\G_{\omega})$ is identical to $\mathscr{O}(\G)$ as a comodule, it follows that $\mathscr{O}(\G\backslash \G_{\omega}) = \C$. The proof of Proposition \ref{PropUniLi} now shows that $\mathscr{O}(\G_{\omega})$ admits a universal C$^*$-envelope. The resulting coaction $\alpha_u$ is then again homogeneous.

To see that $\alpha_u$ is free, we note that the map \[\mathrm{Can}: \mathscr{O}(\G_{\omega})\otimes \mathscr{O}(\G_{\omega})\rightarrow \mathscr{O}(\G)\otimes \mathscr{O}(\G_{\omega}),\quad g\otimes h \mapsto \alpha(g)(1_{\G}\otimes h) = g_{(1)}\otimes g_{(2)}\underset{\omega}{\cdot} h\] is bijective - we leave it to the reader to check that the inverse is given by \[\mathrm{Can}^{-1}(g,h) = \omega^*(g_{(2)},S(g_{(3)})h_{(2)}) g_{(1)}\otimes S(g_{(4)})h_{(1)}.\]
\end{proof} 

To show now that $\G_u\curvearrowright \G_{\omega,u}$ is a torsor, we have to show that $\G_{\omega,u}$ is not trivial.

\begin{Prop} Let $\G$ be a compact quantum group, equipped with a unitary cocycle $\omega$ on $\widehat{\G}$. Then there exists a unique, bounded and faithful $^*$-representation \[\pi_{\omega}:\mathscr{O}(\G_{\omega})\rightarrow B(L^2(\G)), \quad \pi_{\omega}(g)h\xi_{\G} = (g\underset{\omega}{\cdot} h)\xi_{\G},\quad  g,h\in \mathscr{O}(\G).\]
\end{Prop} 
\begin{proof} We have already proven in Lemma \ref{LemRepDefCoc} that $\pi_{\omega}$ exists as a $^*$-representation on $\mathscr{O}(\G)\xi_{\G}$. It then follows as in the proof of Proposition \ref{PropUniLi} that $\pi_{\omega}$ extends to a representation on $L^2(\G)$. 
\end{proof} 

Let us now return to general quantum torsors. The following lemma is a special case of Theorem \ref{TheoGalAnAlg}.

\begin{Lem}\label{LemTorGal} Let $\X\overset{\alpha}{\curvearrowleft}\G$ be a quantum torsor. Then the maps \[\mathrm{Can}_r: \mathscr{O}_{\G}(\X) \aotimes \mathscr{O}_{\G}(\X)\rightarrow \mathscr{O}_{\G}(\X)\aotimes \mathscr{O}(\G),\quad a\otimes b \mapsto \alpha(a)(b\otimes 1)\] and  \[\mathrm{Can}_l: \mathscr{O}_{\G}(\X) \aotimes \mathscr{O}_{\G}(\X)\rightarrow \mathscr{O}_{\G}(\X)\aotimes \mathscr{O}(\G),\quad a\otimes b \mapsto (a\otimes 1)\alpha(b)\] are bijective.  
\end{Lem} 
\begin{proof} As $\X\overset{\alpha}{\curvearrowleft} \G$ is free and homogeneous, we have a unitary map \[\mathcal{G}_{\alpha}: L^2(\X) \otimes L^2(\X)\rightarrow L^2(\X)\otimes L^2(\G)\] such that \[a\xi_{\X}\otimes \eta \mapsto \alpha(a)(\eta\otimes \xi_{\G}).\] We need to show that it restricts to an isomorphism on the algebraic level. But as we have noted in the proof of Theorem \ref{TheoFinBoc}, $\mathcal{G}_{\alpha}$ splits into unitaries \[\mathcal{G}_{\pi}: C(\X)_{\pi}\xi_{\X}\otimes L^2(\X)\rightarrow L^2(\X) \otimes C(\G)_{\pi}\xi_{\G}.\] If now $h\in C(\G)_{\pi}$, we claim that $\mathcal{G}_{\pi}^*(\xi_{\X}\otimes h\xi_{\G}) \in C(\X)_{\pi}\otimes C(\X)_{\bar{\pi}}$. Indeed,  the isotypical components of $C(\X)$ are orthogonal, and for $\rho$ not equivalent with $\bar{\pi}$, we have for $b\in C(\X)_{\pi}, c\in C(\X)_{\rho}$ that \begin{multline*}\langle \mathcal{G}_{\pi}^*(1_{\X}\otimes h\xi_{\G}),b\xi_{\X}\otimes c\xi_{\X}\rangle = \langle 1_{\X}\otimes h\xi_{\G},b_{(0)}c\xi_{\X}\otimes b_{(1)}\xi_{\G}\rangle  \\ = \varphi_{\X}(b_{(0)}c)\varphi_{\G}(h^*b_{(1)}) = 0,\end{multline*} since $C(\X)_{\pi}^* = C(\X)_{\bar{\pi}}$. It follows by finite-dimensionality of the $C(\X)_{\pi}$ that we can write \[\sum_i a_i\xi_{\X}\otimes b_i\xi_{\X} = \mathcal{G}_{\pi}^*(\xi_{\X}\otimes h\xi_{\G}),\] where in the left hand side $a_i \in C(\X)_{\pi}$ and $b_i \in C(\X)_{\bar{\pi}}$. But then we have, for $x\in \mathscr{O}_{\G}(\X)$ arbitrary, \[\alpha(a_i)(b_ix\otimes 1_{\G}) = x\otimes h.\] It follows that $\mathrm{Can}_r$ is surjective. As $\mathcal{G}_{\alpha}$ is isometric, $\mathrm{Can}_r$ is also injective.

The second statement concerning $\mathrm{Can}_l$ follows immediately since \[(b\otimes 1)\alpha(a) = \left(\alpha(a^*)(b^*\otimes 1)\right)^*.\]
\end{proof}

The quantum torsor condition can also be expressed in terms of the elements $U_{\pi}(T,\xi)$, introduced in Section \ref{SecHom}.

\begin{Theorem}\label{TheoUnSpec} Let $\X\overset{\alpha}{\curvearrowleft}\G$ be a non-trivial homogeneous action. Then $\X\overset{\alpha}{\curvearrowleft} \G$ is a quantum torsor if and only if for each irreducible $\pi$, the $C(\X)$-valued matrices $(U_{\pi}(T_i,e_j))_{i,j}$ are unitary, where $\{e_j\}$ is an orthonormal basis of $\Hsp_{\pi}$ and $\{T_i\}$ an orthonormal basis of $\Mor(\pi,\alpha)$.
\end{Theorem}  
Note that in general the $(U_{\pi}(T_i,e_j))_{i,j}$ need not be \emph{square} matrices, but their unitarity can still be guaranteed by the non-commutativity of $C(\X)$. 

\begin{proof} We know that each matrix $(U_{\pi}(T_i,e_j))_{i,j}$ is a coisometry. If then $\X\overset{\alpha}{\curvearrowleft}\G$ is a quantum torsor, it hence suffices to show that $(U_{\pi}(T_i,e_j))_{i,j}$ has a left inverse. But fix a non-zero $\eta\in \Hsp_{\pi}$. As follows from the proof of Lemma \ref{LemTorGal}, we can write \[\mathrm{Can}_l^{-1}(1_{\X}\otimes U_{\pi}(e_k,\eta))= \sum_{i,j}  U_{\pi}(T^{(j)}_{k,i},e_j)^*\otimes U_{\pi}(T_i,\eta)\] for certain $T^{(j)}_{k,i}\in \Mor(\pi,\alpha)$. We then have \begin{eqnarray*} && \hspace{-2cm}\mathrm{Can}_l\left(\sum_{i,j} U_{\pi}(T^{(j)}_{k,i},e_j)^*\otimes U_{\pi}(T_i,\eta)\right) \\ &=& \sum_{i,j,p} U_{\pi}(T^{(j)}_{k,i},e_j)^*U_{\pi}(T_i,e_p)\otimes U_{\pi}(e_p,\eta) \\ &=& 1_{\X}\otimes U_{\pi}(e_k,\eta).\end{eqnarray*} It follows that $(\sum_j U_{\pi}(T^{(j)}_{k,i},e_j)^*)_{k,i}$ is an inverse to $(U_{\pi}(T_i,e_j))_{i,j}$. 

Conversely, if all $(U_{\pi}(T_i,e_j))_{i,j}$ are unitary, note that, for $a\in \mathscr{O}_{\G}(\X)$, \[ \mathrm{Can}_l\left(\sum_{i} aU_{\pi}(T_i,e_k)^*\otimes U_{\pi}(T_i,e_j)\right) = a\otimes U_{\pi}(e_j,e_k).\] It follows that $\mathrm{Can}_l$ is surjective, and hence $\X\overset{\alpha}{\curvearrowleft}\G$ free. 
\end{proof} 

The above characterisation allows to give a purely numerical characterisation of quantum torsors. It also explains the terminology `action of full quantum multiplicity' for quantum torsors, used in \cite{BDV06} where the following corollary is proven.

\begin{Cor}Let $\X\overset{\alpha}{\curvearrowleft}\G$ be a homogeneous action. Then $\X$ is a quantum $\G$-torsor if and only if $\mult_q(\pi) = \dim_q(\pi)$ for all irreducible representations $\pi$ of $\G$. 
\end{Cor} 

\begin{proof} If all $U_{\pi}$ are unitary, the string of inequalities at the end of the proof of Theorem \ref{TheoEstpi} turn into equalities. Hence $\mathrm{Tr}(F_{\pi}^{-1}) = \mathrm{Tr}(Q_{\pi})$ for all irreducible $\G$-representations $\pi$, and $\mult_q(\pi) = \dim_q(\pi)$. 

Conversely, if this identity holds for all $\pi$, then the estimates for $\mathrm{Tr}(F_{\pi}^{-1})$ in the proof of Theorem \ref{TheoEstpi} show that \[(\id_{\Hsp_{\pi}}\otimes \varphi_{\X})(U_{\X,\pi}^*U_{\X,\pi}) = \id_{\Hsp_{\pi}}.\]  As $U_{\X,\pi}^*U_{\X,\pi}$ is a projection, and as $\varphi_{\X}$ is faithful on $\mathscr{O}_{\G}(\X)$, it follows that $U_{\X,\pi}^*U_{\X,\pi}$ is the unit, and hence $U_{\X,\pi}$ a unitary. 
\end{proof}

It is proven in \cite[Section 4]{BDV06} that, on the other hand, we have the equality $\mult(\pi) = \dim(\pi)$ for all irreducible $\G$-representations $\pi$ if and only if $\X = \G_{\omega}$ for some unitary 2-cocycle on $\widehat{\G}$. 

A quantum torsor turns out to actually be a quantum \emph{bi}torsor in a canonical way. That is, from the quantum torsor $\X\curvearrowleft \G$ one can construct a new compact quantum group $\mathbb{H}$ and an action $\mathbb{H}\curvearrowright \X$ such that $\X$ is a left quantum $\mathbb{H}$-torsor and such that the actions of $\mathbb{H}$ and $\mathbb{G}$ commute. We will briefly discuss this construction in the following pages. We refer to \cite{Sch96} for this construction in the setting of Hopf algebras, see also \cite{Bic03}, and \cite{BDV06} for the construction in the setting of compact quantum groups by means of Tannaka-Kre$\breve{\textrm{\i}}$n techniques. 

So, let $\X\overset{\alpha}{\curvearrowleft} \G$ be a quantum torsor. Consider $\mathscr{O}_{\G}(\X^{\op}) = \mathscr{O}_{\G}(\X)^{\op}$, the opposite $^*$-algebra of $\mathscr{O}_{\G}(\X)$. We write the elements in $\mathscr{O}_{\G}(\X)^{\op}$ as $x^{\op}$, so that \[x^{\op}y^{\op} = (yx)^{\op}, \quad (x^{\op})^* = (x^*)^{\op}, \quad x,y\in \mathscr{O}_{\G}(\X).\] For $T\in \Mor(\pi,\alpha)$ and $\xi\in \Hsp_{\pi}$, we further write \[U_{\pi}(\xi,T) = (U_{\pi}(F_{\pi}^{1/2} T, Q_{\pi}^{-1/2}\xi)^*)^{\op} \in \mathscr{O}_{\G}(\X^{\op})\] Note that this formula is inspired by the formula $S(U_{\pi}(\xi,\eta))^* = U_{\pi}(\eta,\xi)$ which holds for compact quantum groups, except that the unitary antipode $R(a) = f^{1/2}*S(a)*f^{-1/2}$ has been changed by the formal operation of `taking the opposite'. To further strengthen this analogy, we will henceforth also write \[S(U_{\pi}(T,\xi)) = U_{\pi}(\xi,T)^*,\quad S(U_{\pi}(\xi,T)) = U_{\pi}(T,\xi)^*.\] Note however that, in contrast, $U_{\pi}(\xi,T)$ is antilinear in both $\xi$ and $T$!

\begin{Lem}\label{LemSBij} The maps \[S: \mathscr{O}(\X) \rightarrow  \mathscr{O}(\X^{\op}),\quad S: \mathscr{O}(\X^{\op}) \rightarrow  \mathscr{O}(\X)\] are bijective anti-homomorphisms satisfying \[S(S(x)^*)^* = x.\]
\end{Lem} 
\begin{proof} The only thing which is not immediately clear is if the $S$ are anti-homomorphisms. It is sufficient to prove this for $S$ as a map from $\mathscr{O}(\X)$ to $\mathscr{O}(\X^{\op})$. This follows from the fact that we can write \[S(x)^{\op} = (\id_{\X}\otimes f^{-1})\alpha(\sigma_{-i/2}^{\X}(x)),\]  where $f^{z}$ denote the Woronowicz characters.
\end{proof}

\begin{Lem}\label{LemAntStar} For $T\in \Mor(\pi,\alpha)$ and $\xi\in \Hsp_{\pi}$, one has \[S(U_{\pi}(T,\xi)^*) = U_{\pi}(Q_{\pi}\xi,F_{\pi}^{-1}T),\quad S(U_{\pi}(\xi,T)^*) = U_{\pi}(F_{\pi}T,Q_{\pi}^{-1}\xi).\] 
\end{Lem} 
\begin{proof} We use again the notation $T^{\dag}(\xi^*) = T(\xi)^*$. Then it follows from the computation in the proof of Theorem \ref{TheoEstpi} that $F_{\bar{\pi}}(T^{\dag}) = (F_{\pi}^{-1}T)^{\dag}$. Hence \begin{eqnarray*} S(U_{\pi}(T,\xi)^*) &=& S(U_{\bar{\pi}}(T^{\dag},\xi^*)) \\ &=& U_{\bar{\pi}}(F_{\bar{\pi}}^{1/2}T^{\dag},Q_{\bar{\pi}}^{-1/2}\xi^*)^{\op} \\ &=& U_{\bar{\pi}}((F_{\pi}^{-1/2}T)^{\dag},(Q_{\pi}^{1/2}\xi)^*)^{\op} \\ &=& U_{\pi}(F_{\pi}^{-1/2} T,Q_{\pi}^{1/2}\xi)^{*\op} \\ &=& U_{\pi}(Q_{\pi}\xi,F_{\pi}^{-1}T).\end{eqnarray*} 
The other equation follows from $S(S(x^*)^*)=x$ for all $x\in \mathscr{O}_{\G}(\X)$. 
\end{proof}

\begin{Lem}\label{LemLeftCoac} There is a unique left action $\G\overset{\beta}{\curvearrowright} \X^{\op}$ such that \[\beta(U_{\pi}(\xi,T))= \sum_i U_{\pi}(\xi,e_i)\otimes U_{\pi}(e_i,T),\] for $e_i$ an arbitrary orthonormal basis of $\Hsp_{\pi}$. 
\end{Lem} 
\begin{proof} Let $R$ be the unitary antipode on $\mathscr{O}(\G)$, as recalled just before Lemma \ref{LemSBij}. Then it is easy to verify that $R$ is an anti-homomorphism which commutes with the $^*$-operation and such that $\Delta \circ R = (R\otimes R)\circ \Delta^{\op}$. It follows that we have a left coaction \[\beta: \mathscr{O}_{\G}(\X^{\op}) \rightarrow \mathscr{O}(\G)\aotimes \mathscr{O}_{\G}(\X^{\op}),\quad x^{\op}\mapsto R(x_{(1)}) \otimes (x_{(0)})^{\op}.\] We leave it to the reader to check that $\beta$ acts on the $U_{\pi}(\xi,T)$ as in the statement of the lemma. 
\end{proof}

\begin{Cor}\label{CorUnSpecFlip} The action $\G\overset{\beta}{\curvearrowright} \X^{\op}$ makes $\X^{\op}$ into a left quantum $\G$-torsor. Moreover, the $C(\X)$-valued matrices $(U_{\pi}(e_i,T_j))_{i,j}$ are unitary, with $\{e_i\}$ an orthonormal basis of $\Hsp_{\pi}$ and $\{T_j\}$ an orthonormal basis of $\Mor(\pi,\alpha)$.
\end{Cor}
\begin{proof} It follows from the formula for $\beta$ in the proof of Lemma \ref{LemLeftCoac} that $\beta$ is homogeneous and free. Hence $\X^{\op}$ is a left quantum $\G$-torsor.

Moreover,  we calculate that \[\beta(\sum_i U_{\pi}(e_i,T_k)^* U_{\pi}(e_i,T_l)) =  1_{\G}\otimes \sum_i U_{\pi}(e_i,T_k)^* U_{\pi}(e_i,T_l),\] so that $\sum_i U_{\pi}(e_i,T_k)^* U_{\pi}(e_i,T_l)$ is a scalar $M_{kl}$. But then, using Lemma \ref{LemAntStar} and the orthogonality relations from Corollary \ref{CorOrthErg}, we find  \begin{eqnarray*}M_{kl} &=& \varphi_{\X}\left(S(\sum_i U_{\pi}(e_i,T_k)^* U_{\pi}(e_i,T_l))^*\right) \\ &=&  \sum_i \varphi_{\X}(U_{\pi}(F_{\pi}T_k,Q_{\pi}^{-1}e_i)^*U_{\pi}(T_l,e_i)) \\ &=& \sum_i \frac{\langle F_{\pi}T_k,F_{\pi}^{-1}T_l\rangle \langle Q_{\pi}^{-1}e_i,e_i\rangle}{\Tr(Q_{\pi})} \\ &=& \langle T_k,T_l\rangle \\ &=& \delta_{kl}. 
\end{eqnarray*}
It follows that $(U_{\pi}(e_i,T_j))_{i,j}$ is an isometry. A similar calculation reveals that \[\varphi_{\X}((\sum_iU_{\pi}(e_k,T_i)U_{\pi}(e_l,T_i)^*)^{\op}) = \delta_{kl},\] so necessarily $(U_{\pi}(e_i,T_j))_{i,j}$ must be a unitary.
  \end{proof}

\begin{Lem} There exists a unital $^*$-homomorphism \[\gamma: \mathscr{O}(\G) \rightarrow \mathscr{O}_{\G}(\X^{\op})\aotimes \mathscr{O}_{\G}(\X),\quad U_{\pi}(\xi,\eta) \mapsto \sum_i U_{\pi}(\xi,T_i)\otimes U_{\pi}(T_i,\eta),\] where $T_i$ is an orthonormal basis of $\Mor(\pi,\alpha)$. 
\end{Lem} 
\begin{proof} For $g\in \mathscr{O}(\G)$, write \[g_{[1]}\otimes g_{[2]}= (S^{-1}\otimes \id)\mathrm{Can}_l^{-1}(1_{\X}\otimes g) \quad \in \mathscr{O}(\X^{\op})\aotimes \mathscr{O}(\X).\] Using the antimultiplicativity of $S$, we see that \[(gh)_{[1]}\otimes (gh)_{[2]} = g_{[1]}h_{[1]}\otimes g_{[2]}h_{[2]}.\] Now from the proof of Theorem \ref{TheoUnSpec}, it follows that \begin{eqnarray*} U_{\pi}(\xi,\eta)_{[1]} \otimes U_{\pi}(\xi,\eta)_{[2]} &=& \sum_i S^{-1}(U_{\pi}(T_i,\xi)^*) \otimes U_{\pi}(T_i,\eta) \\ &=&  \sum_i (S(U_{\pi}(T_i,\xi))^* \otimes U_{\pi}(T_i,\eta) \\ &=& \sum_i U_{\pi}(\xi,T_i)\otimes U_{\pi}(T_i,\eta),\end{eqnarray*} which shows that $\gamma$ is a well-defined homomorphism. 

To see that it respects the $^*$-structure, write again $T^{\dag}(\xi^*) = (T\xi)^*$. Then it follows from the orthogonality relations from Corollary \ref{CorOrthErg} that $\langle T^{\dag},S^{\dag}\rangle = \langle S,F_{\pi}^{-1}T\rangle$. Hence if $\{T_i\}$ is an orthonormal basis of $\Mor(\Hsp_{\pi},\alpha)$, the set $\{(F_{\pi}^{1/2}T_i)^{\dag}\}$ is an orthonormal basis of $\Mor(\bar{\pi},\alpha)$. Furthermore, by definition of the contragredient representation we can write \[U_{\pi}(\xi,\eta)^* = U_{\bar{\pi}}((Q_{\pi}^{-1}\xi)^*, \eta^*).\] Hence \[\gamma(U_{\pi}(\xi,\eta)^*) = \sum_i  U_{\bar{\pi}}((Q_{\pi}^{-1}\xi)^*,(F_{\pi}^{1/2}T_i)^{\dag})\otimes U_{\bar{\pi}}((F_{\pi}^{1/2}T_i)^{\dag},\eta^*).\] On the other hand, we have \begin{eqnarray*} \gamma(U_{\pi}(\xi,\eta))^* &=&\sum_i U_{\pi}(\xi,T_i)^*\otimes U_{\pi}(T_i,\eta)^* \\ &=& \sum_i U_{\pi}(F_{\pi}^{1/2}T_i,Q_{\pi}^{-1/2}\xi)^{\op}\otimes U_{\bar{\pi}}(T_i^{\dag},\eta^*) \\ &=&\sum_i U_{\bar{\pi}}(Q_{\bar{\pi}}^{1/2} (Q_{\pi}^{-1/2}\xi)^*,F_{\bar{\pi}}^{-1/2}(F_{\pi}^{1/2}T_i)^{\dag})\otimes U_{\bar{\pi}}(T_i^{\dag},\eta^*)\\ &=& \sum_i U_{\bar{\pi}}( (Q_{\pi}^{-1}\xi)^*,(F_{\pi}T_i)^{\dag})\otimes U_{\bar{\pi}}(T_i^{\dag},\eta^*),\end{eqnarray*} which is easily seen to be equal to the expression for $\gamma(U_{\pi}(\xi,\eta)^*)$.
\end{proof}

\begin{Def}  Let $\X\overset{\alpha}{\curvearrowleft}\G$ be a quantum torsor. For $T,T'\in \Mor(\pi,\alpha)$ with $\pi$ irreducible, we define \[U_{\pi}(T,T') =  \sum_i U_{\pi}(T,e_i)\otimes U_{\pi}(e_i,T')\in \mathscr{O}_{\G}(\X)\aotimes \mathscr{O}_{\G}(\X^{\op}),\] where $e_i$ is an arbitrary orthonormal basis of $\Hsp_{\pi}$. We define $\mathscr{O}(\mathbb{H}_{\X})$ as the vector space spanned by the $U_{\pi}(T,T')$. 
\end{Def}

\begin{Theorem} Let  $\X\overset{\alpha}{\curvearrowleft}\G$ be a quantum torsor. Then $\mathscr{O}(\mathbb{H}_{\X})$ is a unital $^*$-algebra. Moreover, there exists a unique Hopf $^*$-algebra structure $\Delta$ on $\mathscr{O}(\mathbb{H}_{\X})$ such that \[\Delta(U_{\pi}(T,T')) = \sum_i U_{\pi}(T,T_i)\otimes U_{\pi}(T_i,T'),\] for $T_i$ an orthonormal basis of $\Mor(\pi,\alpha)$. 
\end{Theorem} 

\begin{proof} From the concrete form of the $U_{\pi}(T,T')$ and the formulas for $\alpha$ and $\beta$ on the $U_{\pi}(T,\xi)$ and $U_{\pi}(\xi,T)$ respectively, it follows straightforwardly that \[\mathscr{O}(\mathbb{H}_{\X}) = \{z\in  \mathscr{O}_{\G}(\X)\aotimes \mathscr{O}_{\G}(\X^{\op})\mid (\alpha\otimes \id_{\X^{\op}})z = (\id_{\X}\otimes \beta)z\}.\] This shows that $\mathscr{O}(\mathbb{H}_{\X})$ is a unital $^*$-algebra. 

To show that $\Delta$ is a well-defined $^*$-homomorphism, we note that \[ \sum_i U_{\pi}(T,T_i)\otimes U_{\pi}(T_i,T') = (\id_{\X}\otimes \gamma\otimes \id_{\X^{\op}})(\alpha\otimes \id_{\X^{\op}})U_{\pi}(T,T').\] It is clearly coassociative. 

Define \[\varepsilon: \mathscr{O}(\mathbb{H}_{\X})\rightarrow  \C,\quad U_{\pi}(T,T')\mapsto \langle T',T\rangle.\] Then $\varepsilon$ is a $\C$-linear map, and clearly provides a counit for $\Delta$. It is then automatically a $^*$-homomorphism.  

Finally, write \[S(U_{\pi}(T,T')) = \sum_i S(U_{\pi}(e_i,T')) \otimes S(U_{\pi}(T,e_i)).\] We claim that $S$ is a well-defined linear map from $\mathscr{O}(\mathbb{H}_{\X})$ to itself, providing an antipode for $\Delta$ on $\mathscr{O}(\mathbb{H}_{\X})$. Well-definedness follows since \[S(U_{\pi}(T,T'))^* = \sum_i U_{\pi}(T',e_i)\otimes U_{\pi}(e_i,T) = U_{\pi}(T',T).\] Theorem \ref{TheoUnSpec} and Corollary \ref{CorUnSpecFlip} now guarantee that, with $\{T_i\}$ an orthonormal basis of $\Mor(\pi,\alpha)$, the matrix $(U_{\pi}(T_i,T_j))_{i,j}$ is unitary, hence $S$ satisfies the antipode condition for $\mathscr{O}(\mathbb{H}_{\X})$. 
\end{proof}

When $\mathscr{O}(\X) = \mathscr{O}(\G)_{\omega}$ for some unitary 2-cocycle $\omega$ on $\widehat{\G}$, one can show that $\mathscr{O}(\G)_{\omega}^{\op}$ can be identified with $\mathscr{O}({}_{\omega}\G)$, which is the vector space $\mathscr{O}(\G)$ equipped with the $^*$-algebra structure \[g\underset{\omega^{-1}}{\cdot} h = \omega^*(g_{(1)},h_{(1)})g_{(2)}h_{(2)},\]\[g^{\circ} = \widetilde{\chi}_{\omega}(g_{(1)}^*)g_{(2)}^*,\quad \widetilde{\chi}_{\omega}(g) = \omega(g_{(2)},S^{-1}(g_{(1)})).\] The corresponding right action \[ \mathscr{O}({}_{\omega}\G) \rightarrow \mathscr{O}({}_{\omega}\G)\aotimes \mathscr{O}(\G),\quad g \mapsto g_{(1)}\otimes g_{(2)}\] then turns it into a right Galois object. From this, it is not difficult to see that the Hopf $^*$-algebra $\mathscr{O}(\mathbb{H}_{\X})$, defined correspondingly as above for left Galois objects, is isomorphic to the coalgebra $\mathscr{O}(\G)$ equipped with the new $^*$-algebra structure \[g\underset{\omega}{*} h = \omega^*(g_{(1)},h_{(1)})\omega(g_{(3)},h_{(3)}) g_{(2)}h_{(2)},\]\[g^{\dag} = \widetilde{\chi}_{\omega}(g_{(1)}^*)\chi_{\omega}(g_{(3)}^*)g_{(2)}^*.\]

It is not clear whether the closure of $\mathscr{O}(\mathbb{H}_{\X})$ in $C(\X)\otimes C(\X^{\op})$ defines a compact quantum group. Indeed, for this it is necessary to show that $\Delta$ extends to this closure, which is equivalent with showing that the map $\gamma$ extends to a $^*$-homomorphism \[\gamma: C(\G) \rightarrow C(\X^{\op})\otimes C(\X).\] 

However, if we are working either on the universal or the reduced level, there is no problem.

\begin{Theorem} The Hopf $^*$-algebra $\mathscr{O}(\mathbb{H}_{\X})$ admits a universal completion $C(\mathbb{H}_{\X,u})$, which has the structure of a compact quantum group such that $\mathscr{O}(\mathbb{H}_{\X})= \mathscr{O}(\mathbb{H}_{\X,u})$. The reduced C$^*$-algebra $C(\mathbb{H}_{\X,\red})$ can then be identified with the closure of $\mathscr{O}(\mathbb{H}_{\X})$ in $C(\X_{\red})\otimes C(\X^{\op}_{\red})$.
\end{Theorem}
\begin{proof} The universal completion of $\mathscr{O}(\mathbb{H}_{\X})$ exists, as it is generated by the matrix entries of the unitary matrices $(U_{\pi}(T_i,T_j))_{i,j}$. By the universal property, we have that the inclusion map from $\mathscr{O}(\mathbb{H}_{\X})$ to $C(\X)\otimes C(\X^{\op})$ extends to $C(\mathbb{H}_{\X,u})$, proving that $\mathscr{O}(\mathbb{H}_{\X})$ embeds in its universal completion. It is then again standard to show that in fact  $\mathscr{O}(\mathbb{H}_{\X})= \mathscr{O}(\mathbb{H}_{\X,u})$. 

To show that the natural inclusion map from $\mathscr{O}(\mathbb{H}_{\X})$ to $C(\X_{\red})\otimes C(\X^{\op}_{\red})$ extends to an embedding of $C(\mathbb{H}_{\X,\red})$, it suffices to show that the Haar state on $\mathscr{O}(\mathbb{H}_{\X})$ can be realized as a faithful state on $C(\X_{\red})\otimes C(\X^{\op}_{\red})$. However, this is easily seen to be achieved by the state $\varphi_{\X}\otimes \varphi_{\X^{\op}}$, where $\varphi_{\X^{\op}}(a^{\op}) =  \varphi_{\X}(a)$. 
\end{proof}

It is now easy to continue with $\mathscr{O}(\mathbb{H}_{\X})$, and to show that for example $\mathbb{X}$ has the \emph{left} $\mathbb{H}_{\X}$-action \[U_{\pi}(T,\xi) \mapsto \sum_{i} U_{\pi}(T,T_i) \otimes U_{\pi}(T_i,\xi),\] making it into a left quantum $\mathbb{H}_{\X}$-torsor. In fact, if we write \[\mathbb{G}_{rr} = \mathbb{G},\quad \mathbb{G}_{lr} = \mathbb{X}, \quad \mathbb{G}_{rl} = \mathbb{X}^{\op}, \quad \mathbb{G}_{ll} = \mathbb{H}_{\X},\] one can construct 8 unital $^*$-homomorphisms \[\Delta_{ij}^k: \mathscr{O}(\mathbb{G}_{ij}) \rightarrow \mathscr{O}(\G_{ik})\otimes \mathscr{O}(\G_{kj}),\qquad i,j,k\in \{r,l\}\] according to the above obvious pattern. The quadruple $\{\G_{ij}\}$ and the octuple $\{\Delta_{ij}^k\}$ then form a \emph{Hopf Galois-system}, see \cite{Bic03,Gru04}, and also see \cite{Bic14} for a description of the total algebra $\oplus_{i,j\in\{l,r\}} \mathscr{O}(\G_{i,j})$ as a (connected) \emph{cogroupoid}.

Of course, the compact quantum groups $\mathbb{G}$ and $\mathbb{H}_{\X}$ which are obtained above have a very close connection to each other. We will not provide proofs for the following statements, and refer to \cite{BDV06} for more information.

We first introduce the following terminology.

\begin{Def} Let $\G$ and $\mathbb{H}$ be two compact quantum groups. We call $\G$ and $\mathbb{H}$ \emph{monoidally equivalent} if there exists a quantum $\mathbb{H}$-$\G$-bitorsor $\X$, that is, a right quantum $\G$-torsor $\X\overset{\alpha}{\curvearrowleft} \G$ with a left quantum $\mathbb{H}$-torsor structure $\mathbb{H}\overset{\beta}{\curvearrowright} \X$ such that the two actions commute. 
\end{Def} 

It can be shown that $\mathbb{H}$ is uniquely determined once the quantum $\G$-torsor $\X$ has been specified - namely, it(s algebraic core) must be isomorphic to $\mathbb{H}_{\X}$. On the other hand, there can be many quantum bitorsors linking two monoidally equivalent compact quantum groups. For example, any of the quantum tori is a quantum $\mathbb{T}^2$-$\mathbb{T}^2$-bitorsor. 

When $\X$ is a quantum $\mathbb{H}$-$\G$-bitorsor, its algebraic core with respect to $\G$ is the same as the one with respect to $\mathbb{H}$, and we denote it then simply by $\mathscr{O}(\X)$.

\begin{Def} Let $\mathbb{H}\overset{\beta}{\curvearrowright}\X \overset{\alpha}{\curvearrowleft} \G$ be a quantum $\mathbb{H}$-$\G$-bitorsor. For $\Hsp_{\pi}$ a unitary $\mathbb{H}$-representation, we define \[\Ind_{\G}(\Hsp_{\pi}) = \{x \in \Hsp_{\pi}\aotimes \mathscr{O}(\X) \mid (\delta_{\pi}\otimes \id_{\X})x = (\id_{\Hsp_{\pi}}\otimes \beta)x\}.\] 
\end{Def} 

\begin{Lem} The vector space $\Ind_{\G}(\Hsp_{\pi})$ is finite-dimensional, and a unitary $\G$-representation for \[\langle \sum_i \xi_i\otimes a_i,\sum_j \eta_j\otimes b_j\rangle = \sum_{i,j} \varphi_{\mathbb{H}}(a_i^*b_j)\langle \xi_i,\eta_j\rangle,\]\[x\mapsto (\id_{\X}\otimes \alpha)x.\] 
\end{Lem} 

\begin{Theorem} Let $\G$ and $\mathbb{H}$ be two monoidally equivalent quantum groups, and let $\mathbb{H}\overset{\beta}{\curvearrowright}\X \overset{\alpha}{\curvearrowleft} \G$ be a quantum $\mathbb{H}$-$\G$-bitorsor. Then the map \[\Hsp_{\pi} \rightarrow \Ind_{\G}(\Hsp_{\pi})\] provides a unitary equivalence between the categories of unitary $\mathbb{H}$-representations and unitary $\G$-representations. Moreover, we have natural unitaries \[\Ind_{\G}(\Hsp_{\pi})\otimes \Ind_{\G}(\Hsp_{\pi'})\cong\Ind_{\G}(\Hsp_{\pi} \otimes \Hsp_{\pi'}) ,\]\[ (\sum_i \xi_i\otimes \eta_i)\otimes (\sum_j \eta_j\otimes b_j) \mapsto (\sum_{i,j} \xi_i\otimes\eta_j\otimes a_ib_j).\]
\end{Theorem}

One can then show, in a precise way, that such `monoidal equivalences' between the representation categories of $\mathbb{H}$ and $\G$ are (up to equivalence and isomorphism) classified by the quantum $\mathbb{H}$-$\G$-bitorsors.

As an example we consider the \emph{free orthogonal quantum groups} (of irreducible type), introduced in \cite{VDW96}. These compact quantum groups generalize at the same time the $SU_q(2)$ and the $O_N^+$.

For $T$ a matrix with values in a C$^*$-algebra, we will write $\overline{T}$ for the matrix with $\overline{T}_{ij} = T_{ij}^{\;\,*}$. 

\begin{Def} Take $F\in GL_n(\C)$ such that $F\bar{F} \in \R$. The (universal) \emph{free orthogonal quantum group} $O^+(F)$ is the compact quantum group defined by the C$^*$-algebra \[C(O^+(F))= C^*(U_{ij}\mid 1\leq i,j\leq n, \,U \textnormal{\textit{ unitary}}, F\overline{U}F^{-1} = U),\] endowed with the unique coproduct such that 
 \[\Delta(U_{ij}) = \sum_k U_{ik}\otimes U_{kj}.\]
\end{Def} 

It is easily seen that the above definition makes sense: first of all, we can define a unital $^*$-algebra $\mathscr{O}(O^+(F))$ as the universal $^*$-algebra generated by the above generators and relations, and one immediately checks that it becomes a well-defined Hopf $^*$-algebra with the above coproduct. As the generators assemble into a unitary matrix, the universal enveloping C$^*$-algebra exists, and automatically defines a compact quantum group. What is however not clear is if $\mathscr{O}(O^+(F)) \subseteq C(O^+(F))$: for this one needs the Tannaka-Kre$\breve{\textrm{\i}}$n machinery. 

The condition on $F$ is made to ensure that the canonical unitary corepresentation $U$ is irreducible. Note further that $O^+(F)$ is unchanged under the transformation $F \mapsto zF$ for $z\in \C\setminus \{0\}$. We may hence assume that $F\bar{F} = \epsilon$ with $\epsilon \in \{\pm 1\}$. As $O^+(F)$ is also unchanged under the mapping $F\mapsto  GF\overline{G}^{-1}$ for $G\in \mathrm{GL}(n)$, one can easily show by a polar decomposition argument, see \cite[Section 5]{BDV06}, that one may always assume $F$ to be of the form \[F_{\mathbf{\epsilon},\mathbf{\lambda}}e_i = \epsilon_i\lambda_i e_{\bar{\imath}},\]where $i\mapsto \bar{\imath}$ is an involution on $\{1,2,\ldots,n\}$, and with $\epsilon_i \in \{+,-\}$  and $\lambda_i >0$ satisfying $\epsilon_i\epsilon_{\bar{\imath}} = \epsilon$ and $\lambda_{\bar{\imath}}\lambda_i = 1$.

If $F =I_N$ is a unit matrix, we find back the $O^+_N$. These are precisely the free orthogonal quantum groups which are of \emph{Kac type}, that is, whose Haar state is tracial. On the other hand, when $F$ is a 2-by-2-matrix one can restrict to the case of \[F_q =  \begin{pmatrix} 0  & |q|^{1/2} \\ -\sgn(q)|q|^{-1/2} & 0 \end{pmatrix},\] for some $q\in \lbrack -1,1\rbrack\setminus \{0\}$. In this case $C(O^+_u(F_q)) = C(SU_q(2))$. As remarked before, $SU_q(2)$ is coamenable, hence there is only one C$^*$-completion of $C(SU_q(2))$. This is no longer true for the $O^+(F)$ with $\dim(F) \geq 3$. 

The following theorem establishes an important relationship between all $O^+(F)$.

\begin{Theorem}\label{TheoCompOrth} The family $\{O^+(F)\}$ is complete w.r.t.~ monoidal equivalence. 
Moreover, $O^+(F_1)$ is monoidally equivalent with $O^+(F_2)$ if and only if $c_{F_1}= c_{F_2}$, where we write \[c_F = \mathrm{sign}(F\overline{F})Tr(F^*F).\]
\end{Theorem} 

In particular, $O^+(F)$ is monoidally equivalent to $SU_q(2)$ for $-q-q^{-1} = c_F$, and it follows that the irreducible representations of any $O^+(F)$ can be labeled by the half-integers $\frac{1}{2}\N$. The above also gives examples of coamenable $\G$ being monoidally equivalent to non-coamenable $\mathbb{H}$.

It is not hard to give an explicit description of the quantum bitorsor between two monoidally equivalent $O^+(F_1)$ and $O^+(F_2)$: it is given by 
\[C(O^+_u(F_1,F_2))= C^*\Big{(}_{ij}\mid \begin{array}{lll}1\leq i \leq  \textnormal{dim}(F_1),1\leq j \leq \textnormal{dim}(F_2)\\  U \textnormal{\textit{ unitary}}, F_1\overline{U}F_2^{-1} = U\end{array}\Big{)},\] with the obvious coproduct structure. The hard part consists in showing that this C$^*$-algebra is not trivial, see \cite{BDV06}.

\section{A duality between free and homogeneous actions}

In this section, we discuss a relation between freeness and homogeneity in a general context. This goes back to ideas already present in \cite{Was89}. 

Let $\X \overset{\alpha}{\curvearrowleft} \G$ be a homogeneous action, and assume that $\{\mathscr{E}_i\}_{i\in I}$ is a maximal family of mutually inequivalent $\G$-equivariant right Hilbert $\G$-modules. Write \[C_0(\X_{\stab})= \mathcal{K}(\oplus_i \mathscr{E}_i) = \oplus_{i,j} \mathcal{K}(\mathscr{E}_i,\mathscr{E}_j).\] From Section \ref{SecHom}, it easily follows that we can endow $C_0(\X_{\stab})$ with a coaction \[\alpha: C_0(\X_{\stab}) \rightarrow C_0(\X_{\stab})\otimes C(\G)\]  such that \[\alpha(\xi\eta^*) = \alpha(\xi)\alpha(\eta)^*,\quad \xi\in \mathscr{E}_j,\eta\in \mathscr{E}_i.\] 

Our goal is to show the following.

\begin{Theorem} Let $\X \overset{\alpha}{\curvearrowleft} \G$ be a homogeneous action. Then $\X_{\stab}\overset{\alpha}{\curvearrowleft} \G$ is free. 
\end{Theorem} 

\begin{proof} We have to show that \[\lbrack (C_0(\X_{\stab})\otimes 1_{\G})\alpha(\X_{\stab})\rbrack = C_0(\X_{\stab})\otimes C(\G).\] For this, it is sufficient to show that \[\lbrack  \sum_i (\mathscr{E}_i^*\otimes 1_{\G})\alpha(\mathscr{E}_i)\rbrack = C(\X)\otimes C(\G).\] However, as any $\mathscr{E}_i$ appears as a $\G$-equivariant direct $C(\X)$-Hilbert module summand of some $(\Hsp_{\pi}\otimes C(\X),\alpha_{\pi})$, for $\pi$ a $\G$-representation, it is enough to show that the linear span over all $\pi$ of the \[((\Hsp_{\pi}\otimes \mathscr{O}_{\G}(\X))^*\otimes 1_{\G})\alpha_{\pi}(\Hsp_{\pi}\otimes \mathscr{O}_{\G}(\X))\] is dense in $C(\X)\otimes C(\G)$. But as these elements are of the form \[ \langle v,w_{(0)}\rangle x^*y_{(0)}\otimes w_{(1)}y_{(1)},\quad v,w\in \Hsp_{\pi},x,y\in \mathscr{O}_{\G}(\X),\] it follows that we can obtain all elements of the form $x^*y_{(0)}\otimes gy_{(1)}$ with $x,y\in \mathscr{O}_{\G}(\X)$ and $g\in \mathscr{O}(\G)$, and hence all elements of the form $x\otimes g$ with $x\in \mathscr{O}_{\G}(\X)$ and $g\in \mathscr{O}(\G)$. 
\end{proof} 

Remark that $C_0(\X_{\stab}/\G)\cong c_0(I)$, since $\mathscr{K}(\mathscr{E}_i,\mathscr{E}_j)^{\G} = \delta_{i,j}\C\id_{\mathscr{E}_i}$ by irreducibility and mutual inequivalence of the $\mathscr{E}_i$. We hence obtain the following corollary.

\begin{Cor} Any homogeneous $\X\curvearrowleft\G$ is $\G$-equivariantly Morita equivalent with a free action $\X_{\stab}\curvearrowleft \G$ such that $\X_{\stab}/\G$ is a discrete set.
\end{Cor} 

One can in fact show that this gives a one-to-one correspondence between homogeneous actions, up to equivariant Morita equivalence, and indecomposable free actions with a classical discrete set of quantum orbits, up to equivariant isomorphism. Here an indecomposable action is one which can not be written as a direct sum of two actions.

Let us discuss some examples. In \cite{DCY15}, a classification was provided of all (universal) homogeneous actions of the free orthogonal quantum groups $O^+(F)$ in terms of certain combinatorial data, which we introduce in the next definition. For full proofs of the remaining results of this section, we refer the reader to \cite{DCY15}. 

\begin{Def} Let $\delta\in \R_0$. A \emph{$\delta$-reciprocal random walk} consists of a quadruple $(\Gamma,w,\sgn,i)$ where 
 \begin{itemize}
\item $\Gamma=(\Gamma^{(0)},\Gamma^{(1)},s,t)$ is a graph with source and target maps $s$ and $t$, 
\item $w$ is a weight function $w\colon\Gamma^{(1)}\rightarrow \R_0^+$,
\item $\sgn$ a sign function $\sgn\colon\Gamma^{(1)}\rightarrow \{\pm 1\}$, 
\item $i$ is an involution  $e\mapsto \overline{e}$ on $\Gamma^{(1)}$ interchanging source and target, 
\end{itemize}
s.t.
\begin{itemize}
\item for all $e$, $w(e)w(\bar{e}) = 1$, 
\item for all $e$, $\sgn(e)\sgn(\bar{e}) = \sgn(\delta)$,
\item for all $v$, $\sum_{s(e)=v}  \frac{1}{|\delta|}w(e) = 1$.
\end{itemize}
\end{Def}

Note that if $\delta<0$, the condition $\sgn(e)\sgn(\bar{e}) = \sgn(\delta)$ implies that the set of loops at a vertex must be even. As for the terminology, the `reciprocality' refers to the reciprocality of the weight function under the involution, while the `random walk' part refers to the fact that the normalized weights $\frac{1}{|\delta|}w$ provide probability measures on each $s^{-1}(v)$, that is, we are given probabilities to leave a vertex along a certain edge. 

The first part of the following lemma is proven by straightforward estimates, while the second part is a straightforward application of Frobenius-Perron theory.

\begin{Lem} Let $(\Gamma,w,\sgn,i)$ be a $\delta$-reciprocal random walk. Let $M(\Gamma)$ be the adjacency matrix of $\Gamma$. Then \[\|M(\Gamma)\|\leq |\delta|,\] and in particular $\Gamma$ has bounded degree: 
\[ \sup_{v\in \Gamma^{(0)}} \#\{e\in \Gamma^{(1)}\mid s(e)= v\} <\infty.\] Conversely, if $\Gamma$ is a graph of bounded degree, then there exists a $\delta$-reciprocal random walk on $\Gamma$ for some $\delta>0$.
\end{Lem}

We recall that when considering the (irreducible) free orthogonal quantum groups $O^+(F)$, we may assume that $F$ is of the form $F_{\mathbf{\epsilon},\mathbf{\lambda}}$ as discussed above Theorem \ref{TheoCompOrth}. Here we assume fixed an involution on $\{1,2,\ldots, n\}$, signs $\{\epsilon_i\}$ and positive numbers $\lambda_i$ such that $\epsilon_i\epsilon_{\bar{\imath}}= \epsilon$ for a constant sign $\epsilon$, and $\lambda_i\lambda_{\bar{\imath}} = 1$.

\begin{Def} Write $c_{\mathbf{\epsilon},\mathbf{\lambda}} = -\epsilon \sum_i \lambda_i^2$. Let $(\Gamma,w,\sgn,i)$ be a $c_{\mathbf{\epsilon},\mathbf{\lambda}}$-reciprocal random walk. Then we define $\mathscr{O}(\X^{(\Gamma)})$ to be the universal $^*$-algebra generated by a copy of the finite support functions on $\Gamma^{(0)}$, whose Dirac functions we write $\delta_v$ for $v\in \Gamma^{(0)}$, together with a collection of generators $U_{e,i}$ for $e\in \Gamma^{(1)}$ and $1\leq i\leq n$, such that \[U_{e,i} = \delta_{s(e)}U_{e,i}\delta_{t(e)}\] and \[\sum_{g\in t^{-1}(w)} U_{g,i}^*U_{g,j} = \delta_{i,j}\delta_{w},\quad w\in \Gamma^{(0)}, 1\leq i,j\leq n,\]\[\sum_{i=1}^n U_{e,i}U_{f,i}^* = \delta_{e,f} \delta_{s(e)},\quad e,f\in \Gamma^{(1)},\]\[ U_{e,i}^* = \frac{\epsilon_i\lambda_i}{\sgn(e)\sqrt{w(e)}} U_{\bar{e},\bar{\imath}}.\]
\end{Def} 

Note that the sums in the above definition are well-defined because $\Gamma$ has bounded degree.

\begin{Lem} There exists a unique Hopf $^*$-algebraic coaction \[\alpha:\mathscr{O}(\X^{(\Gamma)}) \rightarrow \mathscr{O}(\X^{(\Gamma)}) \aotimes \mathscr{O}(O^+(F_{\mathbf{\epsilon},\mathbf{\lambda}}))\] such that $\alpha(\delta_v)= \delta_v\otimes 1_{\G}$ and \[\alpha(U_{g,i}) = \sum_{j} U_{g,j}\otimes U_{ji}.\] 
\end{Lem} 
\begin{proof} Straightforward.
\end{proof}

One can easily verify also that  $\X^{(\Gamma)}\curvearrowleft O^+(F_{\epsilon,\lambda}$ only depends on $(\Gamma,w)$, and not on the choice of involution or sign function. 

\begin{Theorem} Let $\X \curvearrowleft O^+(F_{\mathbf{\epsilon},\mathbf{\lambda}})$ be homogeneous. Then there exists a unique $c_{\mathbf{\epsilon},\mathbf{\lambda}}$-reciprocal random walk $(\Gamma,w,\sgn,i)$ (up to isomorphism of $(\Gamma,w)$) such that \[\X_{\stab}\curvearrowleft O^+(F_{\mathbf{\epsilon},\mathbf{\lambda}}) \cong \X^{(\Gamma)}\curvearrowleft O^+(F_{\epsilon,\lambda})\] (more precisely, we have an $O^+(F_{\mathbf{\epsilon},\mathbf{\lambda}})$-equivariant $^*$-isomorphism $\mathscr{O}_{\G}(\X_{\stab}) \cong \mathscr{O}(\X^{(\Gamma)})$.) Moreover, any $c_{\mathbf{\epsilon},\mathbf{\lambda}}$-reciprocal random walk with $\Gamma$ connected arises in this way.
\end{Theorem}

In fact, the graph $\Gamma$ associated to the homogeneous action $\X \overset{\alpha}{\curvearrowleft} O^+(F)$ is nothing but the graph whose vertices are labelled by a set $I$ parametrizing the irreducible $O^+(F)$-equivariant Hilbert $C(\X)$-modules, and where there are $M_{\alpha}(\pi_{1/2})_{i,j}$ edges from $i$ to $j$, where $\pi_{1/2}$ is the generating spin $1/2$-representation of $O^+(F)$, and where $M_{\alpha}(\pi_{1/2})$ is the matrix of fusion rules for $\X$. The precise associated weights on the graph can be obtained from the modular data of the action.

Let us look at some examples of reciprocal random walks and associated homogeneous actions.  

\begin{Exa} Write the fundamental unitary corepresentation of $O_N^+$ associated to the spin $1/2$-representation $\pi_{1/2}$ as $U^{(N)}$. Consider $O_{N-1}^+ \subseteq O_{N}^+$ by the quotient map \[C(O_{N}^+) \rightarrow C(O_{N-1}^+),\quad U^{(N)} \mapsto \begin{pmatrix} U^{(N-1)} & 0 \\ 0 & 1\end{pmatrix}.\]  Then $O_{N-1}^+\backslash O_N^+$ is an instance of a quantum homogeneous space of quotient type. Hence, by the remarks following Proposition \ref{PropFusSub}, the graph $\Gamma$ of the reciprocal random walk associated to $O^+_{N-1}\backslash O_N^+$ has vertices labelled by $\mathbb{N}$, corresponding to the irreducible representations of $O_{N-1}^+$, and has edges determined as

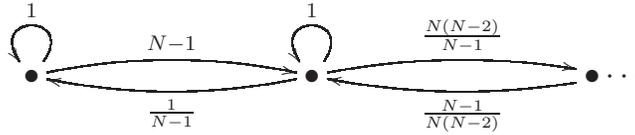
\begin{figure}[h]
\centerline{
\xymatrix@C=8em{ \bullet \ar@(ur,ul)[]_{1}  \ar @/^/[r]^{N-1}  & \ar @/^/[l]^{\frac{1}{N-1}} \bullet \ar@(ur,ul)[]_1 \ar @/^/[r]^{\frac{N(N-2)}{N-1}} &\bullet \ar @/^/[l]^{\frac{N-1}{N(N-2)}}  \cdots}
}
\end{figure}

since the restriction of the spin $1/2$-corepresentation $U^{(N)}$ of $O_N^+$ to $O_{N-1}^+$ splits by construction as the corepresentation $\eta\oplus U^{(N-1)}$, with $\eta$ the trivial corepresentation. The weights are then uniquely determined by the fact that the loops must have weight 1 by the reciprocality relation $w(e)w(\bar{e}) = 1$, and the weight of the other edges is determined inductively by the reciprocality and the random walk condition. It would be interesting to see if $O_{N-1}^+\backslash O_N^+$ is equivariantly isomorphic to the free quantum sphere $S^{N-1}_+$.
\end{Exa}

\begin{Exa} Assume that $(\mathbf{\epsilon}',\mathbf{\lambda}')$ is such that $c_{\mathbf{\epsilon},\mathbf{\lambda}}= c_{\mathbf{\epsilon}',\mathbf{\lambda}'}$. Then we know that there is the quantum torsor $O^+(F_{\mathbf{\epsilon}',\mathbf{\lambda}'},F_{\mathbf{\epsilon},\mathbf{\lambda}}) \curvearrowleft O^+(F_{\mathbf{\epsilon},\mathbf{\lambda}})$. Its associated reciprocal random walk is the graph with one vertex and $n$ edges, equipped with the weights $w(i) = (\lambda_i')^2$. For example, with $O^+(F_{\mathbf{\epsilon},\mathbf{\lambda}}) = SU_q(2)$ and $\mathbf{\lambda}' = (q_1,q_1^{-1},q_2,q_2^{-1})$ with $|q_1|+|q_1|^{-1}+|q_2|+|q_2|^{-1} = |q|+|q|^{-1}$ and $\mathbf{\epsilon}' = (\sgn(q),1,\sgn(q),1)$, we have the graph  
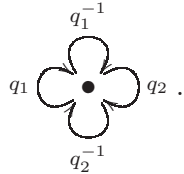
\[\xymatrix{ \bullet \ar@(ul,ur)^{q_1^{-1}} \ar@(dr,dl)^{q_2^{-1}}  \ar@(ru,rd)^{q_2} \ar@(ld,lu)^{q_1}}.\]

\end{Exa} 

\begin{Exa} For the Podle\'{s} sphere $S_{q,x}^2 \curvearrowleft SU_q(2)$, we obtain the reciprocal random walk 
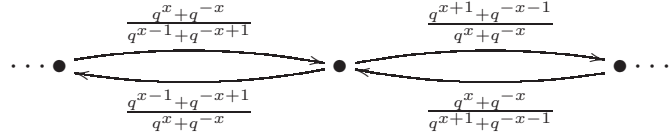
\[
\xymatrix@C=8em{ \cdots \bullet \ar @/^/[r]^{\frac{q^{x}+q^{-x}}{q^{x-1}+q^{-x+1}}}  & \ar @/^/[l]^{\frac{q^{x-1}+q^{-x+1}}{q^{x}+q^{-x}}} \bullet \ar @/^/[r]^{\frac{q^{x+1}+q^{-x-1}}{q^{x}+q^{-x}}} & \bullet \ar @/^/[l]^{\frac{q^{x}+q^{-x}}{q^{x+1}+q^{-x-1}}}  \cdots}\]

\end{Exa}

We note that the embeddable quantum homogeneous spaces for $SU_q(2)$ were classified in \cite{Tom08,TomEr}. In \cite{DCY15}, it was shown that there exists $q_0<1$ such that \emph{all} quantum homogeneous spaces for $SU_q(2)$ are equivariantly $SU_q(2)$-Morita equivalent to an embeddable homogeneous $SU_q(2)$-action when $q_0<q\leq 1$. This result is obtained as a direct consequence of the fact that there exists $\delta>2$ such that any graph with norm $\leq \delta$ automatically has norm $\leq 2$.


\begin{thebibliography}{99}
\bibitem{BS93} S. Baaj and G. Skandalis, Unitaires multiplicatifs et dualit\'{e} pour les produits crois\'{e}s de C$^*$-alg\`{e}bres, \emph{Ann. Sci. \'{E}cole Norm. Sup.} (4), \textbf{26} (4) (1993), 425--488.
\bibitem{BBC07} T. Banica, J. Bichon and B. Collins, Quantum permutation groups: a survey, \emph{Banach Center Publ.} \textbf{78} (2007), 13--34.
\bibitem{BaG10} T. Banica and D. Goswami, Quantum Isometries and Noncommutative Spheres, \emph{Comm. Math. Phys.} \textbf{298} (2) (2010), 343--356.
\bibitem{BDH13} P. Baum, K. De Commer and P. Hajac, Free actions of compact quantum group on unital C$^*$-algebras, arXiv:1304.2812. 
\bibitem{BMT01} E. B\'{e}dos, G.J. Murphy and L. Tuset, Co-amenability of compact quantum groups, \emph{J. Geom. Phys.} \textbf{40} (2) (2001), 130--153. 
\bibitem{Bic03} J. Bichon, Hopf-Galois systems, \emph{J. Alg.} \textbf{264} (2003), 565--581. 
\bibitem{BDV06} J. Bichon, A. De Rijdt and S. Vaes, Ergodic coactions with large multiplicity and monoidal equivalence of quantum groups, \emph{Comm. Math. Phys.} \textbf{262} (2006), 703--728.
\bibitem{Bic14} J. Bichon, Hopf-Galois objects and cogroupoids, \emph{Rev. Un. Mat. Argentina} \textbf{55} (2) (2014), 11--69. 
\bibitem{Boc95} F. Boca, Ergodic actions of compact matrix pseudogroups on C$^*$-algebras, in \emph{Recent Advances in Operator Algebras}, \emph{Ast\'{e}risque} \textbf{232} (1995), 93--110.
\bibitem{BDS13} M. Brannan, M. Daws and E. Samei, Completely bounded representations of convolution algebras of locally compact quantum groups, \emph{M\"{u}nster Journal of Mathematics} \textbf{6} (2013), 445--482.
\bibitem{DC12} K. De Commer, Equivariant Morita equivalences between Podle\'{s} spheres, \emph{Banach Center Publications} \textbf{98} (2012), 85--105.
\bibitem{DCY13}  K. De Commer and M. Yamashita, A construction of finite index C$^*$-algebra inclusions from free actions of compact quantum groups, \emph{Publ. Res. Inst. Math. Sci.} \textbf{49} (4) (2013), 709--735.
\bibitem{DCY15} K. De Commer and M. Yamashita, Tannaka–Kre$\breve{\textrm{\i}}$n duality for compact quantum homogeneous spaces. II. Classification of quantum homogeneous spaces for quantum $SU(2)$, \emph{J. Reine Angew. Math.}, \textbf{2015} (708), 143--171 (2015).
\bibitem{DLRZ02} S. Doplicher, R. Longo, J.E. Roberts, L. Zsid\'{o}, A remark on quantum group actions and nuclearity, \emph{Rev. Math. Phys.} \textbf{14} (7--8) (2002), 787--796. 
\bibitem{Ell00} D. Ellwood, A  new  characterisation  of  principal  actions, \emph{J. Funct. Anal.} \textbf{173} (2000), 49--60.
\bibitem{Gab14} O. Gabriel, Fixed points of compact quantum groups actions on Cuntz algebras, \emph{Ann. Henri Poincar\'{e}} \textbf{15} (5) (2014), 1013--1036. 
\bibitem{GoJ13} D. Goswami, and S. Joardar, Rigidity of action of compact quantum groups on compact connected manifolds, \emph{arXiv:1309.1294v2} (2013).
\bibitem{Gru04}  C. Grunspan, Hopf-Galois systems and Kashiwara algebras, \emph{Comm. Algebra} \textbf{32} (2004), 3373--3389.
\bibitem{HLS81} R. H{\o}egh-Krohn, M. Landstad, E. St{\o}rmer, Compact ergodic groups of automorphisms, \emph{Ann. of Math.} \textbf{114} (2) (1981), 75--86.
\bibitem{Hua13} H. Huang, Faithful compact quantum group actions on connected compact metrizable spaces, \emph{J. Geom. Phys.} \textbf{70} (2013), 232--236.
\bibitem{KNW92} Y. Konishi, M. Nagisa and Y. Watatani, Some remarks on actions of compact matrix quantum groups on C$^*$-algebras, \emph{Pacific J. Math.} \textbf{153} (1) (1992), 119--127.
\bibitem{Lan95} E.C. Lance, Hilbert C$^*$-Modules: A Toolkit for Operator Algebraists, \emph{Cambridge University Press}, Cambridge, England (1995), ix+130 pp.
\bibitem{Li09} H. Li, Compact quantum metric spaces and ergodic actions of compact quantum groups, \emph{J. Funct. Anal.} \textbf{256} (10) (2009), 3368--3408.
\bibitem{Maj95} S. Majid, Foundations of quantum group theory, \emph{Cambridge University Press} (1995), 1-640pp.
\bibitem{MNW91} T. Masuda, Y. Nakagami and J. Watanabe, Noncommutative Differential Geometry on the Quantum Two Sphere of Podle\`{s}. I: An Algebraic Viewpoint, \emph{K-theory} \textbf{5} (1991), 151--175. 
\bibitem{MVD98} A. Maes and A. Van Daele, Notes on compact quantum groups, \emph{Nieuw Arch. Wisk.} (4) \textbf{16}, no. 1--2 (1998), 73--112. 
\bibitem{Mur90} G.J. Murphy, C$^*$-algebras and operator theory, \emph{Academic Press}, Inc., Boston, MA (1990), x+286 pp.
\bibitem{Phi87} N.C. Phillips, Equivariant $K$-Theory and Freeness of Group Actions on C$^*$-Algebras, \emph{Lecture Notes in Math.} \textbf{1274}, Springer-Verlag, Berlin (1987).
\bibitem{Pin13} C. Pinzari, Growth rates of dimensional invariants of compact quantum groups and a theorem of H{\o}egh-Krohn, Landstad and St{\o}rmer, \emph{Proc. Amer. Math. Soc.} \textbf{141} (2013), 895--907. 
\bibitem{Pod87} P. Podle\'{s}, Quantum spheres, \emph{Lett. Math. Phys.} \textbf{14} (1987), 193--202.
\bibitem{Pod95} P. Podle\'{s}, Symmetries of Quantum Spaces. Subgroups and Quotient
Spaces of Quantum $SU(2)$ and $SO(3)$ Groups, \emph{Commun. Math. Phys.} \textbf{170} (1995), 1--20.
\bibitem{Sch96} P. Schauenburg, Hopf Bigalois extensions, \emph{Comm. in Algebra} \textbf{24} (1996), 3797--3825.
\bibitem{Sch04} P. Schauenburg, Hopf-Galois and Bi-Galois extensions, \emph{Galois theory, Hopf algebras, and semiabelian categories, Fields Inst. Comm.} \textbf{43} AMS (2004), 469--515.
\bibitem{Sol11} P.M. So\l tan, On actions of compact quantum groups, \emph{Illinois J. Math.} \textbf{55} (3) (2011), 953--962. 
\bibitem{Tim08} T. Timmermann, An  invitation  to  quantum  groups  and  duality.   From  Hopf  algebras  to multiplicative unitaries and beyond, \emph{EMS Textbooks in Mathematics} (2008), 407 pp.
\bibitem{Swe69} M. E. Sweedler, Hopf algebras, \emph{Mathematics Lecture Note Series}, W. A. Benjamin, Inc., New York (1969), vii+336 pp. 
\bibitem{Tom08} R. Tomatsu, Compact quantum ergodic systems, \emph{J. Funct. Anal.} \textbf{254} (1) (2008), 1--83.
\bibitem{TomEr} R. Tomatsu, Erratum to `Compact quantum ergodic systems', http://www.math.sci.hokudai.ac.jp/~tomatsu/errata.pdf.
\bibitem{Vae01} S. Vaes, The unitary implementation of a locally compact quantum group action, \emph{J. Funct. Anal.} \textbf{180} (2001), 426--480.
\bibitem{VDW96} A. Van Daele and S. Wang, Universal quantum groups, \emph{Internat. J. Math.} \textbf{7} (1996), 255--263.
\bibitem{Ver02} R. Vergnioux, KK-th\'{e}orie \'{e}quivariante et op\'{e}rateur de Julg-Valette pour les groupes quantiques, thesis, universit\'{e} Paris 7 (2002).
\bibitem{Wah10} C. Wahl, Index theory for actions of compact Lie groups on C$^*$-algebras, \emph{J. Op. Theory} \textbf{63} (1) (2010), 217--242.
\bibitem{Wan98} S. Wang, Quantum symmetry groups of finite spaces, \emph{Comm. Math. Phys.} \textbf{195} (1998), 195--211.
\bibitem{Was89} A. Wassermann, Ergodic actions of compact groups on operator algebras. I. General theory, \emph{Ann. of Math.} (2) \textbf{130} (2)  (1989), 273--319.
\bibitem{Wor87} S. L. Woronowicz, Twisted $SU(2)$-group. An example of a noncommutative differential calculus, \emph{Publ. Res. Inst. Math. Sci.} \textbf{23} (1) (1987), 117--181.
\bibitem{Wor95} S.L. Woronowicz, Compact quantum groups, \emph{Sym\'{e}tries quantiques} (Les Houches, 1995), North-Holland, Amsterdam (1998),  845--884. 
\end{thebibliography}
\end{document}